\documentclass{amsart}

\usepackage{euscript}
\usepackage{amssymb}
\usepackage{latexsym}
\usepackage{amsmath,amsfonts}
\usepackage{amsthm,amstext}
\usepackage{amsbsy,amsxtra,amsopn}
\usepackage{enumerate}
\usepackage[dvips]{color}

\definecolor{refkey}{gray}{0.5}
\definecolor{labelkey}{gray}{0.5}

\newtheorem{theorem}{Theorem}[section]
\newtheorem{proposition}[theorem]{Proposition}
\newtheorem{lemma}[theorem]{Lemma}
\newtheorem{corollary}[theorem]{Corollary}

\theoremstyle{definition}
\newtheorem{definition}[theorem]{Definition}
\newtheorem{example}[theorem]{Example}

\newtheorem{remark}[theorem]{Remark}

\numberwithin{equation}{section}

\newcommand{\nR}{{\mathbb R}}
\newcommand{\nC}{{\mathbb C}}
\newcommand{\nCR}{\nC\setminus\nR}

\newcommand{\kip}{[\,\cdot\, , \cdot\,]}
\newcommand{\hip}{(\,\cdot\, , \cdot\,)}
\newcommand{\ahip}{\langle \,\cdot\, , \cdot\,\rangle}

\newcommand{\la}{\langle}\newcommand{\ra}{\rangle}

\newcommand{\cA}{{\mathcal A}}

\newcommand{\cC}{{\mathcal C}}

\newcommand{\cH}{{\mathcal H}}
\newcommand{\cI}{{\mathcal I}}

\newcommand{\cK}{{\mathcal K}}
\newcommand{\cL}{{\mathcal L}}
\newcommand{\cM}{{\mathcal M}}

\newcommand{\cS}{{\mathcal S}}

\def\sN{{\mathfrak N}}

\def\wt#1{{{\widetilde #1} }}
\newcommand{\wh}[1]{\widehat{#1}}
\newcommand{\V}[1]{\mathsf{#1}}
\newcommand{\mb}[1]{\mathbf{#1}}

\newcommand{\co}[1]{{#1}^{*}}

\newcommand{\SM}{\cS_{\cM}}

 \DeclareMathOperator{\sgn}{sgn}
  \DeclareMathOperator{\Arg}{Arg}
 \DeclareMathOperator{\ran}{ran}
 \DeclareMathOperator{\dom}{dom}
 \DeclareMathOperator{\re}{Re}
 \DeclareMathOperator{\im}{Im}
 \DeclareMathOperator{\gr}{gr}

 \DeclareMathOperator{\lspan}{span}

 \DeclareMathOperator{\cspan}{\overline{span}}

\allowdisplaybreaks

\begin{document}

%
%
%
%
%
%
%
%
%

\title[Partially fundamentally reducible operators]%
{Partially fundamentally reducible \\ operators in Kre\u{\i}n spaces}

\author[Branko \'{C}urgus]{Branko \'{C}urgus}
\address{Department of Mathematics,
Western Washington University,
516 High Street \\
Bellingham, Washington 98226, USA}
 \email{curgus@wwu.edu}

\author[Vladimir Derkach]{Vladimir Derkach}
 \address{Department of Mathematics, Donetsk National University \\
600-richchya Str 21 \\ Vinnytsya, 21021, Ukraine}

 \address{\rule{0pt}{12pt}Department of Mathematics, National Pedagogical University  \\
Pirogova Str 9 \\ Kiev, 01601, Ukraine}
 \email{derkach.v@gmail.com}
\thanks{The research of the second author was supported by the Fulbright Fund}



 \date{\today}

\begin{abstract}
A self-adjoint operator $A$ in a Kre\u{\i}n space $(\cK,\kip )$ is called partially fundamentally reducible if there exist a fundamental decomposition $\cK = \cK_+[\dot{+}]\cK_-$ (which does not reduce $A$) and densely defined symmetric operators $S_+$ and $S_-$ in the Hilbert spaces $(\cK_+, \kip)$ and $(\cK_-, -\kip)$, respectively, such that each $S_+$ and $S_-$ has defect numbers $(1,1)$ and the operator $A$ is a self-adjoint extension of $S=S_+ \oplus (-S_-)$ in the Kre\u{\i}n space $(\cK,\kip)$.  The operator $A$ is interpreted as a coupling of operators $S_+$ and $-S_-$ relative to some boundary triples $\bigl(\nC,\Gamma_0^+,\Gamma_1^+\bigr)$ and $\bigl(\nC,\Gamma_0^-,\Gamma_1^-\bigr)$.
Sufficient conditions for a nonnegative partially fundamentally reducible operator $A$ to be similar to a self-adjoint operator in a Hilbert space are given in terms of the Weyl functions $m_+$ and $m_-$ of $S_+$ and $S_-$ relative to the boundary triples $\bigl(\nC,\Gamma_0^+,\Gamma_1^+\bigr)$ and $\bigl(\nC,\Gamma_0^-,\Gamma_1^-\bigr)$.
Moreover, it is shown that under some asymptotic assumptions on $m_+$ and $m_-$ all positive self-adjoint extensions of the operator $S$ are similar to self-adjoint operators in a Hilbert space.
\end{abstract}

\maketitle

 \tableofcontents

\section{Introduction}

In this paper we study a class of self-adjoint operators in a Kre\u{\i}n space which turns out to be similar to self-adjoint operators in a Hilbert space. Recall that a Kre\u{\i}n space is a complex vector space $\cK$ with a sesquilinear form $\kip$ such that there exist subspaces $\cK_+$ and $\cK_-$ of $\cK$ with the following three properties: $\bigl[\cK_+,\cK_-\bigr] = \{0\}$, $\cK = \cK_+ \dot{+}  \cK_-$, a direct sum, and the spaces $(\cK_+, \kip)$ and $(\cK_-, -\kip)$ are Hilbert spaces.  A pair of subspaces $\cK_+$ and $\cK_-$ with the preceding three properties is called a {\em fundamental decomposition} of a Kre\u{\i}n space $\cK$. The projections $P_+$ and $P_-$ associated with the direct sum $\cK = \cK_+ \dot{+} \cK_-$ are called {\em fundamental projections} and the operator $J = P_+ - P_-$ is called a {\em fundamental symmetry} of a Kre\u{\i}n space.  The space $\cK$ with the inner product  $\la x, y \ra = [Jx,y]$, $x,y \in \cK$, is a Hilbert space.  Neither a fundamental decomposition nor a fundamental symmetry of a Kre\u{\i}n space is unique. However, the Hilbert space norms generated by different fundamental decompositions via the corresponding fundamental symmetries are equivalent. All topological notions in a Kre\u{\i}n space refer to the topology of the Hilbert space $(\cK, \ahip)$.  For the general theory of Kre\u{\i}n spaces and operators acting in them we refer to the monographs \cite{AI, B74}.

Unlike the spectrum of a self-adjoint operator in a Hilbert space, the spectrum of a general self-adjoint operator in a Kre\u{\i}n space can be quite arbitrary.  Therefore it is of interest to look for conditions that would guarantee good spectral properties of a self-adjoint operator in a Kre\u{\i}n space.

The ultimate task in this direction is to provide sufficient conditions for a self-adjoint operator in a Kre\u{\i}n space to be similar to a self-adjoint operator in a Hilbert space.  A simple characterization of similarity is as follows.  A self-adjoint operator $A$ in a Kre\u{\i}n space $(\cK,\kip)$ is similar to a self-adjoint operator in the Hilbert space $(\cK,\ahip)$ if and only if $A$ is fundamentally reducible in $(\cK,\kip)$; where {\em fundamentally reducible} means that there exists a fundamental decomposition $\cK = \cK_+ \dot{+}  \cK_-$ of $(\cK,\kip)$ such that $A$ is the direct sum of its restrictions to $\cK_+\!\cap (\dom A)$ and $\cK_- \cap (\dom A)$.  Equivalently, an operator $A$ is fundamentally reducible in a Kre\u{\i}n space $(\cK,\kip)$ if and only if there exists a fundamental symmetry $J$ on $(\cK,\kip)$ such that $J \dom(A) \subseteq \dom(A)$ and $JAx = AJx$ for all $x \in \dom(A)$.  The above characterization of similarity for bounded operators appears in  \cite[Theorem~1, Section~2]{McEn82} and in \cite[Proposition~2.2]{KuTr11}, where an equivalent terminology of $\cC$-symmetry is used.

Another kind of a self-adjoint operator in a Kre\u{\i}n space whose spectral properties resemble those of a self-adjoint operator in a Hilbert space is a nonnegative self-adjoint operator with a nonempty resolvent set.  The spectrum of such an operator is real and, excluding arbitrary neighborhoods of $0$ and $\infty$, the operator $A$ has a projector valued spectral function whose properties resemble the properties of the spectral function of a self-adjoint operator in a Hilbert space; for details see \cite{La} and Subsection~\ref{ss:non-Ks} below for a short review. If the spectrum of $A$ accumulates on both sides of $0$ ($\infty$), then $0$ ($\infty$, respectively) is called a {\em critical point} of $A$.  If the spectral function of $A$ is bounded in a neighborhood of a critical point, then that critical point is said to be {\em regular}.  Otherwise, it is said to be a {\em singular} critical point.  The set of all singular critical points of $A$ is denoted by $c_s(A)$.  Here, by definition, $c_s(A)\subseteq \{0,\infty\}$.  These concepts are closely related to the similarity of $A$ to a self-adjoint operator in a Hilbert space: A nonnegative self-adjoint operator $A$ in a Kre\u{\i}n space such that $\rho(A) \neq \emptyset$ is similar to a self-adjoint operator in a Hilbert space if and only if $\ker(A) = \ker(A^2)$ and $0, \infty \not\in c_s(A)$.

Our first step in studying the similarity question is to introduce a new concept related to the fundamental reducibility in Kre\u{\i}n spaces. For the definition of defect numbers of symmetric operators which we use in the next definition and for general Hilbert space theory see \cite{AG}.

\begin{definition}\label{def:PartFR}
    A self-adjoint operator $A$ in a Kre\u{\i}n space $(\cK,\kip )$ is called {\em partially fundamentally reducible} if there exists a fundamental decomposition $\cK = \cK_+[\dot{+}]\cK_-$ of $(\cK,\kip)$ such that the subspaces
\begin{equation*}
 {\mathcal D}_+\!\!=\!\bigl\{ f \in \cK_+\!\cap\!(\dom A)\! : \! Af \in \cK_+ \!\bigr\}
 \ \ \text{and} \ \
{\mathcal D}_-\!\!=\!\bigl\{ f \in \cK_-\!\cap\!(\dom A)\! : \! Af \in \cK_-\!\bigr\}
\end{equation*}
are dense in $\cK_+$ and $\cK_-$ and the restrictions
\begin{equation}\label{eq:Spm}
    S_+ = A|_{{\mathcal D}_+}\quad\text{and} \quad  S_- = - A|_{{\mathcal D}_-}
    \end{equation}
are symmetric operators with defect numbers $(1,1)$
in the Hilbert spaces $(\cK_+, \kip)$ and $(\cK_-, -\kip)$, respectively.
\end{definition}

The idea behind the above definition is to introduce a desirable restriction of an operator $A$ and then utilize those desirable properties of the restriction to study $A$.  This method resembles the Glazman decomposition method from \cite{Gl}.  A similar idea was also used to study definitizability of differential operators with indefinite weights in \cite[Section~3]{Beh07} and definitizability of self-adjoint operators in Kre\u{\i}n spaces in \cite[Theorem~3.5]{BehPhi10}.

Our objective in this paper is to give sufficient conditions under which a nonnegative partially fundamentally reducible operator in a Kre\u{\i}n space is similar to a self-adjoint operator in a Hilbert space.

To this end we will use a boundary triple approach to extension theory developed in \cite{Koch79, GG, DM91}; see Subsection~\ref{SecW} below for a brief review.  We will apply this theory to the symmetric operators $S_+$ and $S_-$ associated via~\eqref{eq:Spm} with a partially fundamentally reducible operator $A$.
Specifically, let $\bigl(\nC,\Gamma^+_0,\Gamma^+_1\bigr)$ be a boundary triple of the operator $S_+^*$, the adjoint of $S_+$ in the Hilbert space $(\cK_+, \kip)$, and let $m_+$ be the corresponding Weyl function.  Then there exists a unique boundary triple $\bigl(\nC,\Gamma^-_0,\Gamma^-_1\bigr)$ for $S_-^*$ such that the operator $A$ is a coupling of $S_+$ and $S_-$ relative to the boundary triples $\bigl(\nC,\Gamma^+_0,\Gamma^+_1\bigr)$ and $\bigl(\nC,\Gamma^-_0,\Gamma^-_1\bigr)$. That is,  $f \in \dom(A)$ if and only if there exist $f_+ \in \dom\bigl(S_+^*\bigr)$ and $f_- \in \dom\bigl(S_-^*\bigr)$ such that
\[
f = f_+ + f_- \quad \text{and} \quad \Gamma_0^+ f_+ = \Gamma_0^- f_-, \quad \Gamma_1^+ f_+ = - \Gamma_1^- f_-
\]
(see Theorems~\ref{tHsc} and~\ref{tKsc} below). Let $m_-$ be the Weyl function of $S_-$ corresponding to the boundary triple $\bigl(\nC,\Gamma^-_0,\Gamma^-_1\bigr)$. The Weyl functions $m_+$ and $m_-$ belong to the class of Nevanlinna functions, for the definition and basic properties see Subsection~\ref{sect:NS} below. These functions completely characterize the simple (non-self-adjoint) parts of the symmetric operators $S_+$ and $S_-$ acting in the Hilbert spaces $\cK_+$ and $\cK_-$.  Therefore, it is natural to look for conditions for the fundamental reducibility of $A$ in terms of the local behavior of the associated Weyl functions  $m_+$ and $m_-$ at $0$ and $\infty$.

The coupling method used here is a combination of the Glazman decomposition method and the boundary triple approach to the extension theory.  The coupling method was worked out for operators in Hilbert spaces in~\cite{DHMS} and it was used in \cite{Ka}, and also in \cite{KaMal07, KaKost08, KaKoMal09, Ka09, Kost13}, to study the problem of similarity of differential operators with indefinite weights to self-adjoint operators in Hilbert spaces.

Using this method, see \cite{KaMal07, KaKost08}, it was proved that the boundedness of the function
\begin{equation} \label{eq:D0-Dinf}
y \mapsto \frac{\im m_+(iy)+\im m_-(iy)}{m_+(iy)+m_-(-iy)}, \quad y > 0,
\end{equation}
on $(0,1)$ (on $(1,\infty)$, respectively) is necessary for $0 \not\in c_s(A)$ ($\infty \not\in c_s(A)$, respectively).  Since we use these necessary conditions in an essential way, we introduce the following terminology. A pair of functions $(m_+,m_-)$ is said to have $D_0$-{\em property} ($D_\infty$-{\em property}) if the function in \eqref{eq:D0-Dinf} is bounded on $(0,1)$ (on $(1,\infty)$, respectively).

Next we introduce a different kind of local behavior of a Nevanlinna function. A Nevanlinna function $m$ with the integral representation \eqref{eq:IntRep+} in Subsection~\ref{SubSec-B-prop-def} is said to have $B_0$-property ($B_\infty$-property, respectively) if the mapping
\begin{equation*}
 f \mapsto  \int_{0}^{+\infty}\! \frac{f(x)}{x+y}\,d\sigma(x),
\end{equation*}
is a bounded mapping from $L_{\sigma}^2({\mathbb R}_+)$ into $L_{w_m}^2\!(0,1)$ ($L_{w_m}^2\!(1,\infty)$, respectively). Here $w_m(y) = \bigl(\im m(iy)\bigr)^{-1}$, the reciprocal of the imaginary part of $m$.

With $D$- and $B$-properties our main results are as follows. If $A$ is a nonnegative partially fundamentally reducible operator and if associated Weyl functions $m_+$ and $m_-$ have $B_\infty$-property, then $\infty \not\in c_s(A)$ if and only if the pair $(m_+, m_-)$ has $D_\infty$-property (Theorem~\ref{Main}).   Analogously, if $A$ is a nonnegative partially fundamentally reducible operator and if the associated Weyl functions $m_+$ and $m_-$ have $B_0$-property, then $0 \not\in c_s(A)$ if and only if the pair $(m_+, m_-)$ has $D_0$-property and $\ker(A) = \ker(A^2)$ (Theorem~\ref{Main0}). Together these two results give sufficient conditions for a nonnegative partially fundamentally reducible operator in a Kre\u{\i}n space $(\cK,\kip)$ to be similar to a self-adjoint operator in a Hilbert space.

Furthermore, in Subsection~\ref{sSecAsy} we define the asymptotic class $\cA_\infty$ of Nevanlinna functions which all satisfy $B_\infty$-property and pairs of which have $D_\infty$-property and the analogous asymptotic class $\cA_0$ for $B_0$-property and $D_0$-property.  In Theorem~\ref{t:AllPos} we prove that if the Weyl functions $m_+$ and $m_-$  associated  with the partially fundamentally reducible operator $A$ both belong to $\cA_\infty \cap \cA_0$, then, not only $A$, but all the nonnegative self-adjoint extensions of $S_+\oplus(-S_-)$ are similar to a self-adjoint operators in a Hilbert space.

Finally, in Section~\ref{SecExe} we apply our results to indefinite Sturm-Liouville differential operators. In some cases our results lead to a new point of view at some results from \cite{CL, CN, F, FN98, Ka, KaMal07, KaKost08, KaKoMal09}. We also get some new results for the case of nonsymmetric coefficients and the case when $A$ is a coupling of two differential operators of different order.

Since a self-adjoint operator in a Kre\u{\i}n space is similar to a self-adjoint operator in a Hilbert space if and only if it is fundamentally reducible, throughout the rest of the paper we will use the shorter phrase fundamentally reducible to relate to this property.

The main results of this paper were presented at the 21st International Workshop on Operator Theory and Applications held in July of 2010 in Berlin, Germany.

The authors would like to thank Victor Katsnelson and Mark Malamud for useful discussions and relevant literature suggestions.

\subsection{Notation.} We use the standard notation $\nC$ for the set of complex numbers and $\nR$ for the set of real numbers. By $\nC_+$ we denote the set of all $z \in \nC$ with positive imaginary part. Similarly,  $\nR_+$ ($\nR_-$) stands for the set of all positive (negative, respectively) reals. For $z \in \nC$, $\co{z}$, $\re z$, $\im z$ and $\Arg z$ denote the complex conjugate, real, imaginary part of $z$ and the principal value of the argument of $z$ with $\Arg z \in (-\pi, \pi]$, respectively.

For a noninteger real $\alpha$ and $z \in \nC\setminus[0,+\infty)$ we designate the principal branch of $z^\alpha$ to be $|z|^\alpha \exp\bigl(i \alpha \Arg(z)\bigr)$, where $\Arg(z) \in (-\pi, \pi)$.

All operators in this paper are closed densely defined linear operators. For such an operator $T$, we use the common notation $\rho(T)$, $\dom(T)$, $\ran(T)$ and $\ker(T)$ for the resolvent set, the domain, the range and the null-space, respectively, of $T$.

We use the symbols $\pm$ and $\mp$ in a very specific way. Each sentence in which one or both of these symbols appear should be read twice, the first time with the top symbols and the second time with the bottom symbols.  A good example of this is the sentence containing \eqref{CritTpm} and \eqref{CritTpm-+} in Subsection~\ref{ss:RegCPs}. Reading this sentence the first time defines $T_{++}$ and $T_{+-}$; reading it the second time defines $T_{--}$ and $T_{-+}$. The only exception to this rule is the symbol $\operatorname{SL}^\pm(2,\nR)$ which is the common symbol for one object:  the group of all real matrices with determinant equal to $-1$ or $1$.

\section{Preliminaries}
\subsection{Weyl functions of symmetric operators} \label{SecW}
Let $S$ be a closed densely defined symmetric operator in the
Hilbert space $\bigl(\cH,\ahip_\cH\bigr)$ and let $\wh\rho(S)$ denotes the set of points of regular type of $S$, see~\cite{AG}. The subspace
\begin{equation*}
\sN_z=\cH\ominus\ran(S-\bar z)=\ker(S^*-z), \quad z\in\wh\rho(S)
\end{equation*}
is called the defect subspace of the operator $S$. The dimension $\dim(\sN_z)$ is constant on each of the open half-planes $\nC^+$ and $\nC^-$ and is denoted by $d_+,$ for $z \in \nC^+$ and $d_-,$ for $z \in \nC^-$. The numbers $d_+$ and $d_-$ are called the upper and lower defect numbers of $S$. In this paper we
assume that $d_+ = d_- = 1$.

Since the space $(\dom S^*) /(\dom S)$ is two dimensional there exist
(unbounded) non-zero linear functionals $\Gamma_j: S^* \to \nC, \ i,
j \in \{0,1\}$, such that
\[
\ker(\Gamma_0)\cap \ker(\Gamma_1) = \dom(S)
\]
and the abstract Green's identity
\begin{equation} \label{eGreen}
\la S^* f , g \ra_\cH - \la f , S^* g \ra_\cH = \Gamma_0(g)^* \,
\Gamma_1(f) - \Gamma_1(g)^* \, \Gamma_0(f),
\end{equation}
holds for all $f,g\in \dom(S^*)$, see \cite[Section~3.1.4]{GG} for much
more general setting. In \cite{DM95} the triple $\bigl(\nC,
\Gamma_0, \Gamma_1\bigr)$ is called the {\em boundary triple} of a
symmetric operator $S$.

It follows from \eqref{eGreen} that the extensions $S_0$, $S_1$ of
$S$ defined as restrictions of $S^*$ to the domains
 \begin{equation*}
\dom(S_0) :=  \ker(\Gamma_0) \qquad \text{and} \qquad  \dom(S_1) :=
\ker(\Gamma_1),
 \end{equation*}
are self-adjoint linear operators in $\bigl(\cH,\ahip_\cH\bigr)$.
For all $z \in \rho(S_0)$ we have
 \begin{equation} \label{eS*ds}
\dom(S^*) = (\dom S_0) \dotplus \sN_z \quad \text{ direct sum in }
\quad \cH.
 \end{equation}
Since $\sN_z$ is one-dimensional and $\Gamma_0 \neq 0$, it follows
that the restriction $\Gamma_0|_{\sN_z}$ of $\Gamma_0$ to $\sN_z$ is
a bijection between $\sN_z$ and $\nC$. Then
 \[
\bigl( \Gamma_0|_{\sN_z} \bigr)^{-1}: \nC \longrightarrow \sN_z
\subset \cH, \quad z \in \rho(S_0),
 \]
and we define the following two functions:
 \[
\psi: \nCR\longrightarrow \cH \qquad \text{and} \qquad m: \nCR
\longrightarrow \nC
 \]
by
\begin{equation} \label{eWeyl}
\psi(z) := \bigl( \Gamma_0|_{\sN_z} \bigr)^{-1}(1)\quad \text{and}
\quad    m(z) := \Gamma_1 \bigl( \Gamma_0|_{\sN_z} \bigr)^{-1}(1) ,
\quad z \in \rho(S_0).
\end{equation}
Clearly,
\begin{equation}\label{esubg1}
    m(z) = \Gamma_1\psi(z) .
\end{equation}
The functions $\psi$ and $m$ are called the {\em Weyl solution} and
the {\em Weyl function}, respectively,  of the symmetric operator
$S$ relative to the boundary triple $\bigl(\nC, \Gamma_0,
\Gamma_1\bigr)$. For a fixed $z \in \rho(S_0)$ the vector
$f=\psi(z) $ is the solution of the boundary value problem
\[
S^*f = z f, \quad \Gamma_0 f = 1, \quad f \in \dom(S^*).
\]

With $z,w \in \rho(S_0)$, substituting
 \begin{equation*}
f = \psi(z), \quad S^*f = z\psi(z), \quad  g = \psi(w), \quad S^* g = w\psi(w)
 \end{equation*}
in \eqref{eGreen} and using $\Gamma_0 f = \Gamma_0 g = 1$, $m(z) = \Gamma_1 f$, $m(w) = \Gamma_1 g$ we get
\begin{equation} \label{eQfun}
m(z) - \co{m(w)} = (z - \co{w})\,\la \psi(z),\psi(w) \ra_\cH  \quad \text{for all} \quad
z,w \in \rho(S_0).
\end{equation}
With $w = \co{z}$
the identity
\eqref{eQfun} yields that the Weyl function $m$ satisfies the
symmetry condition
\begin{equation} \label{esym}
m(\co{z}) = \co{m(z)}  \quad \text{for all} \quad z \in \rho(S_0).
\end{equation}

The identity \eqref{eQfun} was used in \cite{KL1} as a definition
of the $Q$-function.  It follows from \eqref{eQfun} and  \eqref{esym} that $m$
is a Nevanlinna function; for the definition and basic properties see Subsection~\ref{sect:NS} below.
\begin{remark}
It turns out (see \cite[page 8]{DM91}) that the Weyl solution can be
used to evaluate $\Gamma_1(S_0-z)^{-1}h$ for arbitrary $h \in \cH$
and $z \in \rho(S_0)$:
 \begin{equation} \label{eG1A0}
\Gamma_1(S_0-z)^{-1}h = \bigl\la h,\psi(\co{z}) \bigr\ra_\cH.
 \end{equation}
Indeed, substituting in~\eqref{eGreen}  $f=(S_0-z)^{-1}h$,
$g=\psi(\co{z})$, and using  $\Gamma_0f = 0$, $\Gamma_0\psi(\co{z})=1$,
we obtain the equality
 \[
\bigl\la (S_0-z)f,\psi(\co{z}) \bigr\ra_\cH
 = (\Gamma_1 f)(\Gamma_0 \psi(\co{z}))^*- (\Gamma_0 f)(\Gamma_1 \psi(\co{z}))^* = \Gamma_1f,
 \]
which proves \eqref{eG1A0}.
\end{remark}
\begin{proposition}\label{prop:S_01}
{\rm\cite{AG,DM91}}
For every $z\in\rho(S_0)$ the following equivalence hold:
\begin{equation*}
    z\in\rho(S_1)\quad \Longleftrightarrow \quad m(z)\ne 0
\end{equation*}
and the resolvent of $S_1$ can be found by the formula
\begin{equation*}
    (S_1-z)^{-1} h =(S_0-z)^{-1} h - \frac{\bigl\langle h,\psi(\co{z})\bigr\rangle_{{\cH}}}{m(z)} \,
 \psi(z)
\end{equation*}
for all $h \in \cH$ and all $z \in \rho(S_0) \cap\rho(S_1)$.
\end{proposition}
\begin{remark} \label{rem:transp}
Any two boundary triples $\bigl(\nC,
\Gamma_0, \Gamma_1\bigr)$ and $\bigl(\nC,
\wh{\Gamma}_0, \wh{\Gamma}_1\bigr)$ of the same symmetric operator $S$ are related by
\begin{equation*}
    \begin{pmatrix}
\wh{\Gamma}_0\\
\wh{\Gamma}_1
\end{pmatrix} = W
\begin{pmatrix}
\Gamma_0\\
\Gamma_1
\end{pmatrix},
\end{equation*}
where $W$ is a complex $2\!\times\!2$ matrix satisfying $W^* \V{J} W= \V{J}$, with
$\V{J}=\begin{pmatrix}
0 & -i\\
i & 0
\end{pmatrix}$. The condition $W^* \V{J} W= \V{J}$ is equivalent to $W$ being a unitary matrix in the Kre\u{\i}n space $\bigl(\nC^2, \kip_2\bigr)$, where $[\mb{x},\mb{y}]_2 = \mb{y}^*\! \V{J} \mb{x}$ for $\mb{x},\mb{y} \in \nC^2$.  A direct calculation shows that a unitary matrix $W$ in $\bigl(\nC^2, \kip_2\bigr)$ allows a factorization
\[
W = e^{i \theta} \begin{pmatrix}
a & b\\
c & d
\end{pmatrix} \quad \text{where} \quad  \begin{pmatrix}
a & b\\
c & d
\end{pmatrix} \in \operatorname{SL}(2,\nR) \quad \text{and} \quad \theta \in (-\pi, \pi].
\]
Here $\operatorname{SL}(2,\nR)$ is the special linear group of all real $2\!\times\!2$ matrices with determinant one.

The Weyl function $m$ relative to the boundary triple $\bigl(\nC,
\Gamma_0, \Gamma_1\bigr)$ and the Weyl function $\wh{m}$ relative to the boundary triple $\bigl(\nC, \wh{\Gamma}_0, \wh{\Gamma}_1\bigr)$ are related by
\[
\wh{m}(z) = \frac{d\rule{1.2pt}{0pt}m(z)+c}{b\rule{1.2pt}{0pt}m(z)+a}, \qquad z \in  \rho(S_0) \cap\rho(S_1).
\]

In particular, the boundary triple $\bigl(\nC,-\Gamma_1, \Gamma_0\bigr)$ is said to be a {\it transpose} of the boundary triple $\bigl(\nC, \Gamma_0, \Gamma_1\bigr)$.  The Weyl function $m^{\!\top}$ of $S^*$ relative to $\bigl(\nC,-\Gamma_1, \Gamma_0\bigr)$ is given by
$m^{\!\top}(z)=-{1}/{m(z)}$, $z \in \nCR$.
\end{remark}

\subsection{Nonnegative operators in Hilbert spaces} \label{sect:NonnegSO}
Recall that a symmetric operator $S$ in a
Hilbert space $\bigl(\cH,\ahip_\cH\bigr)$ is called nonnegative if $\langle Sf,f\rangle_{\cH}\ge 0$ for all $f\in\dom(S)$.  By a result of Friedrichs, every nonnegative symmetric operator $S$ admits a nonnegative self-adjoint extension.  Moreover, as was shown by Kre\u{\i}n \cite{Kr47}, in the set $\operatorname{Ext}_+(S)$ of all nonnegative self-adjoint extensions of a nonnegative symmetric operator $S$ there are two extremal extensions $S_F$ and $S_K$.  The extensions $S_F$ and $S_K$, which are called the Friedrichs extension and the Kre\u{\i}n  extension, respectively, are maximal and minimal in the following sense: for all $\wt S \in \operatorname{Ext}_+(S)$ and all $a > 0$ we have,
\[
(S_F+a)^{-1}\le(\wt S+a)^{-1}\le(S_K+a)^{-1}.
\]

The extensions  $S_F$ and $S_K$ can be characterized in terms of boundary triples and Weyl functions. If $S$ is a  symmetric nonnegative operator and $\bigl( \nC,\Gamma_0,\Gamma_1\bigr)$ is a boundary triple for $S^*$, then the extension $S_0$ of $S$ is also nonnegative if and only if  the corresponding Weyl function $m$ is holomorphic on $\nR_-$.
Moreover (see~\cite[Proposition~10]{DM91}), the following two equivalences hold:
\begin{align} \nonumber
   S_0=S_F \quad &\Longleftrightarrow \quad \lim_{x\downarrow -\infty}  m(x) =-\infty,\\
\label{eq:m_K0}
   S_0=S_K \quad &\Longleftrightarrow \quad \rule{2pt}{0pt}\lim_{x\uparrow 0}  m(x) = +\infty.
\end{align}


\subsection{Nevanlinna, Stieltjes and inverse Stieltjes functions} \label{sect:NS}
A complex function $m$ is called a {\em
Nevanlinna function} if $m$ is holomorphic at least on $\nCR$ and satisfies the following two conditions
\begin{equation*}
m(\co{z}) = \co{m(z)} \quad \text{and} \quad \im m(z) \ge 0, \qquad \text{for all} \qquad z \in \nC_+.
\end{equation*}

Equivalently, $m$ is a Nevanlinna function if and only if there exist $a, b \in \nR$, $b \geq 0$,  and a nondecreasing function $\sigma:\nR \to \nR$ such that
\begin{equation*}
m(z) = a + b z +\int_{\nR} \left( \frac{1}{t-z} -
\frac{t}{1+t^2} \right) d\sigma(t)
\end{equation*}
and
\begin{equation}\label{eq:sigmaCond}
    \int_{\nR} \dfrac{d\sigma(t)}{1+t^2} < + \infty.
\end{equation}
If, additionally, $\sigma$ is normalized by
\begin{equation} \label{eq:sigma-nor}
\sigma(t) = \frac{\sigma(t-0)+\sigma(t+0)}{2} \quad \text{and} \quad \sigma(0) = 0,
\end{equation}
then it is uniquely determined by $m$. For these and other facts on Nevanlinna functions see \cite{KaKr74} and \cite[Chapter~II]{Do}.

We consider that a Nevanlinna function is defined on its {\em domain of holomorphy}. That is, the domain of a Nevanlinna function $m$ coincides with the union of $\nCR$ and the set of all those real points to which $m$ admits a holomorphic continuation.

A Nevanlinna function $m$ is called a {\em Stieltjes} (an {\em inverse Stieltjes}) function if it is holomorphic on $\nC\setminus[0,+\infty)$ and it takes nonnegative (nonpositive, respectively) values on $\nR_-$.

The following proposition is a direct consequence of the definitions.
\begin{proposition}\label{p:SiS}
Let $m$ be a Nevanlinna function. The following statements are equivalent.
\begin{enumerate}
\renewcommand*\theenumi{\roman{enumi}}
\renewcommand*\labelenumi{{\rm (\theenumi)}}
\item
$m$ is an inverse Stieltjes function.
\item
The function $z \mapsto -m(1/z)$, $z \in \nCR$, is a Stieltjes function.
\item
The function $z \mapsto -1/m(z)$, $z \in \nCR$, is a Stieltjes function.
\end{enumerate}
\end{proposition}

A Stieltjes function $m$ admits the integral representation
\begin{equation}\label{eq:int_S}
m(z) = \gamma+\int_{0}^{+\infty}  \frac{d\sigma(t)}{t-z}
\end{equation}
with $\gamma \ge 0$ and with a non-decreasing function $\sigma(t)$, such that
\begin{equation} \label{eq:int_S_gc}
    \int_{0}^{+\infty} \dfrac{d\sigma(t)}{1+t} < + \infty.
\end{equation}
Clearly, for every Stieltjes function function $m$
\begin{equation*}
    \re m(iy)\ge 0 \quad\mbox{for all}\quad y\in\nR_+.
\end{equation*}

\subsection{Reproducing kernel Hilbert spaces} \label{sect:RKHS}
In this section we review the basic properties of reproducing
kernel Hilbert spaces associated with Nevanlinna functions.
A function $m$ defined on
$\nCR$, is a Nevanlinna function if and only if
the kernel
\begin{equation} \label{eKm}
K_m(z,w) : = \dfrac{m(z) - \co{m(w)}}{z - \co{w}}, \quad z \neq
\co{w}, \quad z,w \in \nCR,
\end{equation}
is non-negative.

With the kernel $K_{m}(z,w)$ in \eqref{eKm} we associate the
reproducing kernel Hilbert space $\bigl(
\cH(m),\ahip_{\cH(m)} \bigr)$ defined as follows:
\begin{enumerate}
\renewcommand*\theenumi{\alph{enumi}}
\renewcommand*\labelenumi{{\rm (\theenumi)}}
\item
the elements of $\cH(m)$ are holomorphic functions defined on
$\nCR$,
\item \label{rkfs}
for every $w \in \nCR$ the function $z \mapsto K_m(z,w), \
z\in\nCR$, belongs to $\cH(m)$,
\item
the set of all linear combinations of the functions in (\ref{rkfs})
is dense in the reproducing kernel Hilbert space $\bigl(
\cH(m),\ahip_{\cH(m)} \bigr)$,
\item
for every $f \in \cH(m)$ and every $w \in \nCR$ we have
\begin{equation*}
\bigl\la f(\cdot), K_m(\cdot,w) \bigr\ra_{\cH(m)} = f(w).
\end{equation*}
\end{enumerate}

The following theorem, presented in \cite[Proposition 5.3]{DM95} (see also~\cite[Theorem 2.2]{LT77}) assures
that each Nevanlinna function is a Weyl function of a closed simple
symmetric operator.

\begin{theorem} \label{mz}
Let $m$ be a Nevanlinna function, $m(z)\not\equiv 0$.  The operator $S_{m}$ of
multiplication by the independent variable in the reproducing kernel
Hilbert space $\cH(m)$  is a closed simple symmetric operator
with defect numbers $(1,1)$.  The operator $S_m$ is densely defined if and only if
\begin{equation*}
    \lim_{y\uparrow +\infty}\frac{\im m(iy)}{y}=0 \quad \text{and} \quad
    \lim_{y\uparrow +\infty}y\im m(iy)=+\infty.
\end{equation*}
In the graph notation, the adjoint of $S_m$ is given by
\begin{align*}
S_{m}^* &= \cspan\bigl\{ \{ K_{m}(\cdot,w),  \co{w} K_{m}(\cdot,w)
\} : w\in\nCR \bigr\}  \\
  &= \bigl\{ \{f,g\} \in \cH(m)^2 :
             g(z) - z f(z) = c_0 - c_1 m(z) \ \, \text{for some} \, \ c_0, c_1
\in \nC \bigr\}.
\end{align*}
The numbers $c_0,c_1 \in \nC$ in the last displayed equality are uniquely
determined by $f \in \dom(S_{m}^{*})$ and $\bigl(\nC,\Gamma_{m,0}, \Gamma_{m,1}\bigr)$ defined by
\begin{equation*}
\Gamma_{m,0}(f) := c_0, \qquad \Gamma_{m,1}(f) := c_1, \qquad f \in \dom(S_{m}^{*}),
\end{equation*}
is a boundary triple for $S_{m}^*$. The Weyl function of $S_m$
relative to the boundary triple $\bigl(\nC,\Gamma_{m,0}, \Gamma_{m,1}\bigr)$ of $S_{m}^*$ is
$m$.
\end{theorem}

The boundary triple $\bigl(\nC,\Gamma_{m,0}, \Gamma_{m,1}\bigr)$ is
called the {\em canonical boundary triple} for the operator $S_m^*$.

The following theorem is in some sense a converse of Theorem \ref{mz}.

\begin{theorem} \label{tFt}
Let $S$ be a closed simple symmetric operator in a Hilbert space
$\bigl(\cH,\ahip_{\cH}\bigr)$ with defect numbers $(1,1)$, let
$\bigl(\nC,\Gamma_0, \Gamma_1\bigr)$ be a boundary triple for $S$
and let $m$ be the corresponding Weyl function. Then the operator
$F_m: \cH \to \cH(m)$ defined by
\begin{equation}\label{eq:F_trans}
    (F_m f)(z) = \bigl\la f, \psi(z^*) \bigl\ra_{\cH}, \quad f \in \cH,
\ \ \ z \in \nC\setminus\nR,
\end{equation}
is an isometry between $\bigl(\cH,\ahip_{\cH}\bigr)$ and $\bigl(
\cH(m),\ahip_{\cH(m)} \bigr)$ and
\begin{equation*}
F_m S = S_m F_m.
\end{equation*}
\end{theorem}

\subsection{Nevanlinna functions and M\"{o}bius transformations}

It is clear that a composition of two Nevanlinna functions is a Nevanlinna function. However, the reproducing kernel space corresponding to the composition can be significantly different from the reproducing kernel Hilbert space of the composed functions.  In this subsection we show that the situation is different when a Nevanlinna function is composed with a linear fractional (or M\"{o}bius)  transformations, which are itself Nevanlinna functions. Then the corresponding reproducing kernel Hilbert spaces are isomorphic.

It is straightforward to verify that the mapping
\begin{equation} \label{eq:MobHom}
\operatorname{GL}(2,\nC) \ni \left(\!\!\! \begin{array}{cc}
a & b \\ c & d \end{array} \!\!\!\right) \mapsto \mu(z) = \frac{a\rule{1pt}{0pt}z +b}{c\rule{1pt}{0pt}z+d}
\end{equation}
is a group homomorphism from the matrix group $\operatorname{GL}(2,\nC)$ of all $2\!\times\!2$ complex matrices with nonzero determinant onto the M\"{o}bius group.
The kernel of this homomorphism is the subgroup of all nonzero multiples of the identity matrix $I_2$.

We already encountered the subgroup $\operatorname{SL}(2,\nR)$ of $\operatorname{GL}(2,\nC)$. Below we will encounter its another subgroup $\operatorname{SL}^{\pm}(2,\nR)$ of all $2\!\times\!2$ real matrices with determinant $1$ or $-1$.  The following lemma is standard. A short proof based on the homomorphism in \eqref{eq:MobHom} is included for completeness.

\begin{lemma} \label{lMobNev}
For a  M\"{o}bius transformation $\mu$ the condition $\mu(z^*) = \mu(z)^*$ is equivalent to
\begin{equation} \label{eq:MobNev}
\mu(z) = \frac{a\rule{1pt}{0pt}z +b}{c\rule{1pt}{0pt}z+d} \quad \text{for some}
\quad \left(\!\!\! \begin{array}{cc}
a & b \\ c & d \end{array} \!\!\!\right) \in \operatorname{SL}^\pm(2,\nR).
\end{equation}
A M\"{o}bius transformation $\mu$  in \eqref{eq:MobNev} is a Nevanlinna function whenever $ad-bc >0$;  its opposite $-\mu$ is a Nevanlinna function whenever $ad-bc <0$.
\end{lemma}
\begin{proof}
Let
\begin{equation*}
\mu(z) = \frac{a_1z + b_1}{c_1z+d_1} \qquad \text{where} \qquad \left(\!\!\! \begin{array}{cc}
a_1 & b_1 \\ c_1 & d_1 \end{array} \!\!\!\right) \in \operatorname{GL}(2,\nC).
\end{equation*}
Assume that $\mu(z^*) = \mu(z)^*$, that is $\mu(z^*)^* = \mu(z)$. Then the homomorphism in \eqref{eq:MobHom} implies
\begin{equation} \label{eq:conmult}
 \delta  \left(\!\!\! \begin{array}{cc}
a_1^* & b_1^* \\ c_1^* & d_1^* \end{array} \!\!\!\right) =\left(\!\!\! \begin{array}{cc}
a_1 & b_1 \\ c_1 & d_1 \end{array} \!\!\!\right) \quad \text{for some} \quad \delta \in \nC\setminus\{0\}.
\end{equation}
Calculating determinants of both sides yields,
\begin{equation*}
\delta^2 = e^{2 i \theta} \qquad \text{where} \qquad\theta = \operatorname{Arg}(a_1d_1-b_1c_1).
\end{equation*}
Since $\delta$ in \eqref{eq:conmult} is uniquely determined, we have
\begin{equation} \label{eq:choose-pm}
 \delta = e^{i \theta} \qquad \text{or} \qquad \delta = -e^{i \theta}.
\end{equation}
It follows from \eqref{eq:conmult} and \eqref{eq:choose-pm} that
\[
\frac{\sqrt{\delta}}{\sqrt{|a_1d_1 - b_1c_1|}} \left(\!\!\! \begin{array}{cc}
a_1 & b_1 \\ c_1 & d_1 \end{array} \!\!\!\right) =  \left(\!\!\!\begin{array}{cc}
a & b \\ c & d \end{array} \!\!\!\right) \in \operatorname{SL}^\pm(2,\nR)
\]
and whether $ad-bc = 1$ or $ad-bc = -1$ depends on the choice of the root in \eqref{eq:choose-pm}. The converse is straightforward.

The last claim follows from the identity
\begin{equation*}
\frac{\im \mu(z)}{\im z} = \frac{ad - bc}{|c\rule{1pt}{0pt}z+d|^2}. \qedhere
\end{equation*}
\end{proof}

For a M\"{o}bius transformation $\mu$ given in \eqref{eq:MobNev} a direct calculation confirms  the following identity
\begin{equation} \label{eq:MobId}
\mu(z) - \mu(w^*) = \frac{ad-bc}{(c\rule{1pt}{0pt}z+d)(c\rule{1pt}{0pt}w^*+d)} (z - w^*) \quad \text{for all} \quad z, w \in \nCR,
\end{equation}
which will be used in the next proof.

\begin{theorem} \label{tV}
Let $\mu_1$ and $\mu_2$ be M\"{o}bius transformations given by
\begin{equation} \label{eq:MobNev12}
\mu_j(z) = \frac{a_jz +b_j}{c_j z+d_j} \quad \text{where} \quad \left(\!\!\! \begin{array}{cc}
a_j & b_j \\ c_j & d_j \end{array} \!\!\!\right) \in \operatorname{SL}^\pm(2,\nR), \quad j \in \{1,2\}.
 \end{equation}
For $j \in \{1,2\}$ set $\epsilon_j := a_j d_j-b_j c_j \in \{-1,1\}$.
Let $m$ be a complex function defined on $\nCR$ and set
\begin{equation} \label{eq:MobNev-whm}
\wh{m}(z) = \epsilon_1 \epsilon_2 \, \mu_2\bigl(m(\mu_1(z))\bigr), \qquad z \in \nC\setminus\nR.
\end{equation}
Then $m$ is a Nevanlinna function if and only if $\wh{m}$ is a Nevanlinna function.
If $m$ is a Nevanlinna function the mapping $V: \cH(\wh{m}) \to \cH\bigl(m\bigr)$ defined by
\begin{equation*}
(Vf)(z) : = \frac{c_1 \mu_1(z) - a_1}{c_2 m(\mu_1(z))+d_2} \, f\bigl(\mu_1(z)\bigr), \ \ \ f \in  \cH(m), \ \ \ z \in \nCR,
\end{equation*}
is an isomorphism between the Hilbert spaces $\cH(m)$ and $\cH(\wh{m})$.
\end{theorem}

\begin{proof}
First notice that the M\"{o}bius transformations in
\eqref{eq:MobNev12} have inverses which are also M\"{o}bius transformations of the same kind. For example
\[
\mu_1^{-1}(z) = \frac{d_1 z - b_1}{-c_1 z + a_1} \quad \text{and} \quad  \left(\!\!\! \begin{array}{cc}
d_1 & -b_1 \\ -c_1 & a_1 \end{array} \!\!\!\right) \in \operatorname{SL}^\pm(2,\nR).
\]
Now the first statement follows from Lemma~\ref{lMobNev} and the fact that a composition of Nevanlinna functions is a Nevanlinna function.

Assume that $m$ is a Nevanilnna function. In the following identities we assume that $z, w \in \nCR$ and we use \eqref{eq:MobId}, first for $\mu_2$, then for $\mu_1^{-1}$. For convenience we set $z_1 = \mu_1(z)$ and $w_1 = \mu_1(w)$ and calculate:
\begin{align*}
K_{\wh{m}}(z,w) & = \epsilon_1 \epsilon_2  \,\frac{\mu_2\bigl(m(\mu_1(z))\bigr) - \mu_2\bigl(m(\mu_1(w^*))\bigr)}{z-w^*} \\
 & = \frac{\epsilon_1 \epsilon_2\epsilon_2}{(c_2 m(z_1)+d_2)(c_2m(w_1)^*+d_2)} \ \frac{m(z_1) - m(w_1^*)}{\mu_1^{-1}(z_1)-\mu_1^{-1}(w_1^*)} \\
 & = \frac{(-c_1 z_1 + a_1)(-c_1 w_1^*+ a_1)}{(c_2 m(z_1)+d_2)(c_2m(w_1)^*+d_2)} \
 \frac{m(z_1) - m(w_1^*)}{z_1 -w_1^*} \\
 & = \frac{(c_1 \mu_1(z) - a_1)(c_1 \mu_1(w)^* - a_1)}{(c_2 m(\mu_1(z))+d_2)(c_2 m(\mu_1(w^*))+d_2)} \
K_m\bigl(\mu_1(z),\mu_1(w)\bigr).
\end{align*}
Consequently, substituting $w$ with $\mu_1^{-1}(w)$, we get
\begin{equation*}
\frac{c_1 \mu_1(z) - a_1}{c_2 m(\mu_1(z))+d_2} \
K_m\bigl(\mu_1(z),w\bigr) = \frac{c_2 m(w^*)+d_2}{c_1 w^* - a_1} K_{\wh{m}}(z,\mu_1^{-1}(w)),
\end{equation*}
that is
\begin{equation} \label{eqUl}
V\bigl(K_m(\cdot,w)\bigr) = \frac{c_2 m(w^*)+d_2}{c_1 w^* - a_1}  K_{\wh{m}}(\cdot,\mu_1^{-1}(w)) \quad \text{for all} \quad w \in \nCR.
\end{equation}

Therefore, for arbitrary $v, w \in \nCR$ we have
\begin{align*}
\bigl\la V K_{m}(\cdot,v), V K_{m}(\cdot,w) \bigr\ra_{\cH(\wh{m})} & \\
& \hspace*{-4cm}
 = \frac{c_2 m(v^*)+d_2}{c_1 v^* - a_1}   \
 \frac{c_2 m(w)+d_2}{c_1 w - a_1} \
  \Bigl\la K_{\wh{m}}(\cdot,\mu_1^{-1}(v)), K_{\wh{m}}(\cdot,\mu_1^{-1}(w)) \Bigr\ra_{\cH(\wh{m})} \\
 &\hspace*{-4cm} = \frac{c_2 m(v^*)+d_2}{c_1 v^* - a_1}   \
 \frac{c_2 m(w)+d_2}{c_1 w - a_1}  \
 K_{\wh{m}} \bigl(\mu_1^{-1}(w),\mu_1^{-1}(v) \bigr) \\
 & \hspace*{-4cm} = K_m(w,v) \\
 &\hspace*{-4cm} = \bigl\la K_m(\cdot,v), K_m(\cdot,w) \bigr\ra_{\cH(m)}.
\end{align*}
Thus, for arbitrary $v, w \in \nCR$ we have
\begin{equation} \label{eqUlu}
\bigl\la V K_{m}(\cdot,v), V K_{m}(\cdot,w) \bigr\ra_{\cH(\wh{m})} = \bigl\la K_m(\cdot,v), K_m(\cdot,w) \bigr\ra_{\cH(m)}.
\end{equation}
Set
\begin{align*}
\cL(m) & := \lspan \bigl\{ K_m(\cdot,w) \, : \, w \in \nCR \bigr\},  \\
\cL(\wh{m}) & := \lspan \bigl\{ K_{\wh{m}}(\cdot,w) \, : \, w \in
\nCR \bigr\}.
\end{align*}
As both functions $m$ and $\wh{m}$ are Nevanlinna functions, $\cL(m)$ is dense in the Hilbert space $\cH(m)$ and $\cL(\wh{m})$ is dense in the Hilbert space $\cH(\wh{m})$.

Since $V$ is linear, it follows from \eqref{eqUl} and \eqref{eqUlu}
that the restriction $V\bigl|_{\cL(m)}\bigr.$ is a bijection and an
isomorphism of pre-Hilbert spaces  $\cL(m)$ and $\cL(\wh{m})$.
Denote by $V_1$ the extension by continuity of $V\bigl|_{\cL(m)}
\bigr.$ to $\cH(m)$. Then $V_1$ is an isomorphism between
$\cH(m)$ and $\cH(\wh{m})$. Let $g \in \cH(\wh{m})$ be arbitrary
and let $f \in \cH(m)$ be such that $V_1 f = g$. For $z \in
\nCR$ we calculate:
\begin{align*}
g(z) & = \bigl\la g(\cdot), K_{\wh{m}}(\cdot,z) \bigr\ra_{\cH(\wh{m})} \\
 & = \Bigl\la
  (V_1f)(\cdot),  \frac{c_1 \mu_1(z^*) - a_1}{c_2 m(\mu_1(z^*))+d_2} V K_m(\cdot,\mu_1(z)) \Bigr\ra_{\cH(\wh{m})} \\
 & = \frac{c_1 \mu_1(z) - a_1}{c_2 m(\mu_1(z))+d_2} \, \Bigl\la f(\cdot),  K_m(\cdot,\mu_1(z)) \Bigr\ra_{\cH(m)} \\
  & = \frac{c_1 \mu_1(z) - a_1}{c_2 m(\mu_1(z))+d_2} \, f\bigl( \mu_1(z) \bigr).
\end{align*}
Thus $V = V_1$.
\end{proof}

\subsection{Stieltjes functions and M\"{o}bius transformations}
Assigning to a meromorphic function $m$ the value $\infty$ at the poles we can consider $m$ to be defined on its whole domain of meromorphy.  In particular, if the domain of meromorphy of a Nevanlinna function $m$ includes $\nR_-$, then the image $m(\nR_-)$ of $\nR_-$ under $m$ is contained in the one point compactification $\overline{\nR} = \nR\cup \{\infty\}$ of $\nR$. In the next lemma by $\overline{m(\nR_-)}$ we denote the closure of $m(\nR_-)$ in $\overline{\nR}$.

\begin{lemma} \label{l:Ms}
Let $m$ be a Nevanlinna function. The following statements are equivalent.
\begin{enumerate}
\renewcommand*\theenumi{\roman{enumi}}
\renewcommand*\labelenumi{{\rm (\theenumi)}}
\item \label{i:Ms-1}
There exists a
M\"{o}bius transformation $\mu$ of the form~\eqref{eq:MobNev} such that $\mu\circ m$ is a Stieltjes function.
\item \label{i:Ms-2}
There exists a M\"{o}bius transformation $\mu$ of the form~\eqref{eq:MobNev} such that $\mu\circ m$ is an inverse Stieltjes function.
\item \label{i:Ms-3}
The function $m$ is meromorphic on $\nC\setminus[0,+\infty)$ and
$\overline{m(\nR_-)}$ is a proper subset of $\overline{\nR}$.
\end{enumerate}
\end{lemma}
\begin{proof}
The equivalence of (\ref{i:Ms-1}) and (\ref{i:Ms-2}) follows from Proposition~\ref{p:SiS}. The implication  (\ref{i:Ms-2}) $\Rightarrow$ (\ref{i:Ms-3}) is clear. To complete the proof, assume (\ref{i:Ms-3}).  Then the complement of $\overline{m(\nR_-)}$ in $\overline{\nR}$  contains a finite open interval. By shifting $m$ we can assume that that interval is $(0,2/c)$ with $c > 0$. Set $\mu(z) = -z/(c z-1)$. Then $\mu\circ m$ is a Nevanlinna function which is holomorphic on $\nC\setminus[0,+\infty)$.  Since $\mu(x) < 0$ whenever $x \in \nR \setminus[0,2/c]$ and $m(\nR_-) \subset \nR \setminus[0,2/c]$ we conclude that $(\mu\circ m)(x) < 0$ for all $x \in \nR_-$. Thus $\mu\circ m$ is an inverse Stieltjes function.  This proves (\ref{i:Ms-3}) $\Rightarrow$ (\ref{i:Ms-2}).
\end{proof}

Denote by $\SM$ the set of Nevanlinna functions that satisfy the equivalent conditions in Lemma~\ref{l:Ms}. The set $\SM$ appears in \cite{BHSWW13} as the union of Nevanlinna functions of type I, II, III and V$^{\prime}$, see \cite[Definition~2.4]{BHSWW13} and also \cite[Corollary~5.6]{BHSWW13}.

Notice that a function $m \in \SM$ can have at most one pole in $\nR_-$ and that the following limits exist
\[
m(-\infty):= \!\!\lim_{x\downarrow -\infty} m(x) \in \{-\infty\}\cup\nR \quad \text{and} \quad  m(0-):=\lim_{x\uparrow 0} m(x) \in \nR \cup\{+\infty\},
\]
with at least one of them being finite. In addition, $m(-\infty) \leq m(0-)$ if and only if $m$ is holomorphic on $\nC\setminus[0,+\infty)$ and $m(-\infty) >  m(0-)$ if and only if $m$ has a pole in $\nR_-$.

\begin{remark} \label{r:MsInfZer}
Let $m \in \SM$. It follows from Lemma~S1.3.1 and Remark~5.1 in \cite{KaKr74} that $m(-\infty) \in \nR$ if and only if $\lim_{y\uparrow +\infty} m(i y)$ exists as a real number; in this case $\lim_{y\uparrow +\infty} m(i y)=m(-\infty)$.  Also, $m(-\infty) = -\infty$ if and only if  $\lim_{y\uparrow +\infty} m(i y) = \infty$.

Since $m \in\SM$ if and only if the function $z \mapsto -m(1/z)$, $z\in \nCR$, is in $\SM$, we similarly have that $m(0-) \in \nR$ if and only if $\lim_{y\downarrow 0} m(i y)$ exists as a real number and in this case $\lim_{y\downarrow 0} m(i y)=m(0-)$.  Also, $m(0-) = +\infty$ if and only if  $\lim_{y\downarrow 0} m(i y) = \infty$.
\end{remark}


\subsection{Two asymptotic classes of Nevanlinna functions} \label{sSecAsy}

For functions $f$ and $g$ defined on $\nR_+$ the expression
\begin{alignat*}{2}
f(t) \sim g(t) \quad  &\text{as} \quad t \to +\infty  & \qquad
\bigl( f(t) \sim & g(t) \quad  \text{as} \quad t \to 0+  \bigr)
\intertext{
means
}
\lim_{t\uparrow +\infty} \frac{f(t)}{g(t)} &= 1  & \qquad
\Bigl( \lim_{t\downarrow 0+} \frac{f(t)}{g(t)} &= 1, \quad \text{respectively} \Bigr).
\end{alignat*}


We define the set of functions $\cA_\infty$ as follows: a Nevanlinna function $m$ belongs to  $\cA_\infty$ if and only if there exist $\alpha\in(0,1)$, $C > 0$ and a M\"{o}bius transformation $\mu$ of the form~\eqref{eq:MobNev} such that $\mu \circ m$ is a Stieltjes function and for all $z \in \nCR$ we have
\begin{equation}\label{eq:mAsymInf-defAInf}
   (\mu \circ m)(r z) \sim \frac{C}{ (- r z)^{\alpha} } \qquad \text{as} \qquad r \to +\infty.
     \end{equation}

Similarly, we define the set of functions $\cA_0$ as follows: a Nevanlinna function $m$ belongs to  $\cA_0$ if and only if there exist $\alpha \in (0,1)$, $C > 0$ and a M\"{o}bius transformation $\mu$ of the form~\eqref{eq:MobNev} such that $\mu \circ m$ is an inverse Stieltjes function and for all $z \in \nCR$ we have
\begin{equation} \label{eq:mAsymZer-defAZer}
   (\mu\circ m)( r z) \sim - C (- r z)^{\alpha} \qquad \text{as} \qquad r \to 0\!+.
\end{equation}

In the notation of  Theorem~{\rm\ref{tV}}, the following proposition holds.

\begin{proposition} \label{prop:AInf-AZer}
\begin{enumerate}
\renewcommand*\theenumi{\roman{enumi}}
\renewcommand*\labelenumi{{\rm (\theenumi)}}
\item  \label{i:AInf-AZer-1}
Assume that $\mu_1(z) = z$ and let $\mu_2$ be a M\"{o}bius transformation as in \eqref{eq:MobNev12}.  Then, $m \in \cA_\infty$ {\rm (}$m \in \cA_0${\rm )} if and only if $\wh{m} \in \cA_\infty$ {\rm (}$\wh{m} \in \cA_0$, respectively{\rm )}.

\item  \label{i:AInf-AZer-2}
Let $\mu_1(z) = 1/z$ and let $\mu_2$ be a M\"{o}bius transformation as in \eqref{eq:MobNev12}. Then, $m \in \cA_\infty$ {\rm (}$m \in \cA_0${\rm )} if and only if $\wh{m} \in \cA_0$  {\rm (}$\wh{m} \in \cA_\infty$, respectively{\rm )}.
\end{enumerate}
\end{proposition}
\begin{proof}
We will prove one implication from each (\ref{i:AInf-AZer-1}) and (\ref{i:AInf-AZer-2}). The other implications are proved similarly.

Let $m \in \cA_\infty$.  Then there exists a M\"{o}bius transformation $\mu$ of the form~\eqref{eq:MobNev} such that $\mu \circ m$ is a Stieltjes function and \eqref{eq:mAsymInf-defAInf} holds for all $z \in \nCR$.

Recall that, according to \eqref{eq:MobNev-whm}, in (\ref{i:AInf-AZer-1}) we have $\wh{m} = \epsilon_2 \mu_2\circ m$. Set $\nu(z) = \epsilon_2 z$, $z \in \nC$.  Since $\mu\circ\mu_2^{-1}\circ \nu \circ \wh{m} = \mu\circ m$ and $\mu\circ\mu_2^{-1}\circ \nu$ is a a M\"{o}bius transformation of the form~\eqref{eq:MobNev}, we have $\wh{m} \in \cA_\infty$. This proves the implication $m \in \cA_\infty \Rightarrow \wh{m} \in \cA_\infty$ from (\ref{i:AInf-AZer-1}).

Again, according to \eqref{eq:MobNev-whm}, in (\ref{i:AInf-AZer-2}) we have $\wh{m} = -\epsilon_2 \mu_2\circ m \circ \mu_1$. Set $\nu(z) = -\epsilon_2 z$, $z \in \nC$. Clearly,
\[
-\mu \circ \mu_2^{-1} \circ \nu \circ \wh{m} = - \mu \circ m \circ \mu_1. \]
Since $\mu \circ m$ is a Stieltjes function that satisfies  \eqref{eq:mAsymInf-defAInf} for all $z \in \nCR$, by Proposition~\ref{p:SiS} the function
$- \mu \circ m \circ \mu_1$ is an inverse Stieltjes function that satisfies \eqref{eq:mAsymZer-defAZer} for all $z \in \nCR$.
As $-\mu \circ \mu_2^{-1} \circ \nu$ is a M\"{o}bius transformation of the form~\eqref{eq:MobNev},  we have $\wh{m} \in \cA_0$. This proves the implication $m \in \cA_\infty \Rightarrow \wh{m} \in \cA_0$ from (\ref{i:AInf-AZer-2}).
\end{proof}

In the next proposition we characterize the functions in $\cA_\infty$ and $\cA_0$ by their asymptotic behavior.

\begin{proposition} \label{prop:0-inf}
The following equivalences hold.
\begin{enumerate}
\renewcommand*\theenumi{\roman{enumi}}
\renewcommand*\labelenumi{{\rm (\theenumi)}}
  \item \label{i:P-0-inf-1}
$m \in \cA_\infty$ if and only if $m \in \SM$ and there exist $C_0 \in \nR\!\setminus\!\{0\}$ and $\alpha_0 \in (-1,0)\cup(0,1)$ such that for all $z \in \nCR$ we have
\begin{equation}\label{eq:mAsymInf-noC}
    m( r z) \sim C_0 (- r z)^{\alpha_0} \quad \text{as} \quad r \to +\infty,
     \end{equation}
or, corresponding to $\alpha_0 =0$, there exist $C_0\in\nR\!\setminus\!\{0\}$, $C_1 > 0$ and $\alpha_1 \in (-1,0)$ such that for all $z \in \nCR$ we have
\begin{equation}\label{eq:mAsymInf-C}
    m( r z) - C_0 \sim C_1( - r z )^{\alpha_1} \quad \text{as} \quad r \to +\infty.
     \end{equation}
\item \label{i:P-0-inf-2}
$m\in \cA_0$ if and only if $m \in \SM$ and there exist $C_0 \in \nR\!\setminus\!\{0\}$ and $\alpha_0 \in (-1,0)\cup(0,1)$ such that for all $z \in \nCR$ we have
\begin{equation*}
    m( r z) \sim C_0 (- r z)^{\alpha_0} \quad \text{as} \quad
    r \to 0\!+,
\end{equation*}
or, corresponding to $\alpha_0 =0$, there exist $C_0, \in\nR\!\setminus\!\{0\}$, $C_1 > 0$ and $\alpha_1 \in (0,1)$ such that for all $z \in \nCR$ we have
\begin{equation*}
    m( r z) - C_0 \sim -C_1( - r z )^{\alpha_1} \quad \text{as} \quad r \to 0\!+.
\end{equation*}
\end{enumerate}
\end{proposition}
\begin{proof}
We first prove the direct implication in  (\ref{i:P-0-inf-1}). Clearly $\cA_\infty \subset \SM$.  Assume that $m \in \cA_\infty$ is such that \eqref{eq:mAsymInf-defAInf} holds and let $\mu^{-1}(z) = (az+b)/(cz+d)$ with $ad - bc = 1$. We distinguish two cases.

If $d = 0$, then $\mu^{-1}(z) = a/c + b/(cz)$, and therefore, for all $z \in \nCR$,
\[
m( r z)  \sim  \frac{b}{c C} (-r z)^{\alpha} \quad \text{as} \quad r \to +\infty.
\]
Thus \eqref{eq:mAsymInf-noC} holds.

If $d \neq 0$, then $\mu^{-1}(z) = b/d + z/d^2 + O(z^2)$ as $z\to 0$, and therefore, for all $z \in \nCR$
\[
m( r z) - \frac{b}{d}  \sim  \frac{C}{d^2}(-r z)^{-\alpha}  \quad \text{as} \quad r \to +\infty.
\]
Consequently, \eqref{eq:mAsymInf-noC} holds if $b=0$ and \eqref{eq:mAsymInf-C} holds if $b \neq 0$. This proves the direct implication in (\ref{i:P-0-inf-1}).

To prove the converse let $m \in \SM$. First assume that $m$ satisfies  \eqref{eq:mAsymInf-noC} with $\alpha_0 \in (-1,0)$ and $C_0 > 0$. By Remark~\ref{r:MsInfZer} we have $m(-\infty) = 0$. If $m$ is holomorphic on $\nC\setminus[0,+\infty)$, then it is a Stieltjes function; so $m \in \cA_\infty$ in this case.  If $m$ has a pole in $\nR_-$, then, since $m \in \SM$, $m(0-) < 0$ and $m(x) \not\in\bigl(m(0-),0\bigr)$ for all $x\in\nR_-$. Setting $c = -m(0-)/2$ and $\mu(z) = z/(c z+1)$, as in the proof of Lemma~\ref{l:Ms} we have that $\mu\circ m$ is a Stieltjes function. Since $\mu(z) = z + O(z^2)$ as $z \to 0$, we have that
\[
 (\mu\circ m)(r z) \sim C_0 (- r z)^{\alpha_0} \qquad \text{as} \qquad r \to 0.
\]
Hence $m \in \cA_\infty$.

Second, assume that $m$ satisfies \eqref{eq:mAsymInf-noC} with $\alpha_0 \in (0,1)$ and $C_0 < 0$. Then, the already proven part implies  $m^{\!\top} \in \cA_\infty$ and Proposition~\ref{prop:AInf-AZer} yields $m \in \cA_\infty$.

Third, assume that $m$ satisfies \eqref{eq:mAsymInf-C}. Then, with $\mu(z) = z - C_0$, we have $\mu\circ m \in \cA_\infty$, by the second part of this proof. Now Proposition~\ref{prop:AInf-AZer} implies that $m \in \cA_\infty$.

The proof of (\ref{i:P-0-inf-2}) is similar.
\end{proof}


\subsection{Nonnegative operators in Kre\u{\i}n spaces} \label{ss:non-Ks}

Let $\bigl(\cK,\kip\bigr)$ be a Kre\u{\i}n space and let $A$ be a densely defined operator in $\cK$. The {\em adjoint} $A^{[*]}$ of $A$ with respect to $\kip$ is defined analogously as in a Hilbert space. In fact, if $J$ is a fundamental symmetry on $\bigl(\cK,\kip\bigr)$ and $\ahip$ is the corresponding Hilbert space inner product, then $A^{[*]} = JA^{\la*\ra}J$.  In analogy with definitions in a Hilbert space, $A$ is {\em symmetric in} $\bigl(\cK,\kip\bigr)$ if $A^{[*]}$ is an extension of $A$ and $A$ is {\em self-adjoint in} $\bigl(\cK,\kip\bigr)$ if $A = A^{[*]}$.

Since our main interest in this paper is similarity of a self-adjoint operator in a Kre\u{\i}n space to a self-adjoint operator in a Hilbert space, in the next proposition we recall a known characterization of similarity. This characterization is proved in \cite[Theorem~1, Section~2]{McEn82} for bounded self-adjoint operators in a Kre\u{\i}n space. In \cite[Proposition~2.2]{KuTr11} an equivalent statement in terms of $\cC$-stable symmetries is given and Phillips theorem \cite[Chapter 2, Corollary 5.20]{AI} is cited for a proof. Below we give a simple direct proof.

\begin{proposition} \label{p:Sim-iff-fd}
A self-adjoint operator $A$ in a Kre\u{\i}n space $(\cK,\kip)$ is similar to a self-adjoint operator in the Hilbert space $(\cK,\ahip)$ if and only if $A$ is fundamentally reducible in $(\cK,\kip)$.
\end{proposition}
\begin{proof}
Let $A$ be a self-adjoint operator in a Kre\u{\i}n space $(\cK,\kip)$ and assume that $A$ is similar to a self-adjoint operator in the Hilbert space $(\cK,\ahip)$.  That is, $A$ is self-adjoint in a Hilbert space  $(\cK,\hip)$ whose norm is equivalent to that of $(\cK,\ahip)$.  The equivalence of the norms of $(\cK,\ahip)$ and $(\cK,\hip)$ and the identity $\la x,y \ra = [Jx,y]$, for all $x,y \in \cK$, imply the existence of a bounded and boundedly invertible operator $G:\cK \to \cK$ such that $(G x,y)= [x,y]$ for all $x,y \in \cK$.  The last identity yields that  $G$ is self-adjoint in $\hip$ and the operator $\sgn G$ is a fundamental symmetry on $(\cK,\kip)$.  As $A$ is self-adjoint in both $\kip$ and $\hip$,  $A$ commutes with $G$.  Hence, $A$ commutes with the fundamental symmetry $\sgn G$ on $(\cK,\kip)$; proving that $A$ is fundamentally reducible. The converse is clear.
\end{proof}

A densely defined operator $A$ is {\em nonnegative in} $\bigl(\cK,\kip\bigr)$ if $[Af,f] \geq 0$ for all $f \in \dom(A)$.  A nonnegative self-adjoint operator in a Kre\u{\i}n space can have an empty resolvent set; a specific example is given in \cite[1.2]{La71} and \cite[Example~VII.1.5]{B74}. A modification of this example leads to a positive self-adjoint operator in a Kre\u{\i}n space with an empty resolvent set, see \cite[Example~3.4]{Ko14}.

However, if $A$ is a nonnegative self-adjoint operator in $\bigl(\cK,\kip\bigr)$ and $\rho(A) \neq \emptyset$, then the spectrum of $A$ is real and $A$ has a spectral function $E$, see \cite[Theorem~II.3.1]{Le}.  The domain of $E$ is the ring $\cI$ which consists of finite unions of (bounded) intervals whose endpoints are nonzero real numbers and their complements in $\overline{\nR} = \nR\cup\{\infty\}$.  The values of $E$ are bounded operators on $\cK$ with the following properties. For all $\Delta, \Delta_1, \Delta_2 \in \cI$ we have
\begin{enumerate}
\item[(E1)]
$E(\emptyset) = 0$, $E(\overline{\nR}) = I$,
\item[(E2)]
$E(\Delta) = E(\Delta)^{[*]}$
\item[(E3)]
$E(\Delta_1 \cap \Delta_2) = E(\Delta_1) E(\Delta_2)$,
\item[(E4)]
$E(\Delta_1 \cup \Delta_2) = E(\Delta_1) + E(\Delta_2)$  whenever $\Delta_1 \cap \Delta_2 = \emptyset$,
\item[(E5)]
the space $\bigl(E(\Delta)\cK, \pm\kip \bigr)$ is a Hilbert space whenever $\Delta \subset \nR_\pm$,
\item[(E6)]
$E(\Delta)$ is in the double commutant of the resolvent of $A$,
\item[(E7)]
for a bounded $\Delta$ we have $E(\Delta)\cK \subseteq \dom(A)$ and the restriction of $A$ to $E(\Delta)\cK$ is a bounded operator whose spectrum is contained in $\overline{\Delta}$.
\end{enumerate}

For $\lambda \in \nR\setminus\{0\}$, it follows from (E5) that in a neighborhood of $\lambda$ the spectral function $E$ behaves as a spectral function of a self-adjoint operator in a Hilbert space. In particular, with $\lambda_1, \lambda_2 \in \nR\setminus\{0\}$ such that $\lambda_1 < \lambda < \lambda_2$ the limits
\begin{equation} \label{eq:sotl}
 \lim_{t\uparrow \lambda}E\bigl([\lambda_1,t]\bigr) \qquad \text{and} \qquad \lim_{t\downarrow \lambda}E\bigl([t,\lambda_2]\bigr)
\end{equation}
exist in the strong operator topology. The existence of these limits is a consequence of the following property of nonzero real numbers: For $\lambda \in \nR\setminus\{0\}$ there exists a $\Delta \in \cI$  such that $\lambda$ is in the interior of $\Delta$ and  $\bigl(E(\Delta)\cK, (\sgn \lambda)\kip  \bigr)$ is a Hilbert space.

A possible  absence of the just mentioned property for $0$ or $\infty$ is a motivation for the following definition.  The point $0$ ($\infty$) is said to be a {\em critical point of} $A$ if $\kip$ is indefinite on $E(\Delta)\cK$ for every $\Delta \in \cI$ such that $0 \in \Delta$ ($\infty \in \Delta$, respectively).

However, even if $0$ or $\infty$ is a critical point the limits analogous to \eqref{eq:sotl} may exist. If $0$ is a critical point of $A$ and the limits in \eqref{eq:sotl} exist with $\lambda = 0$, then $0$ is called a {\em regular critical point of} $A$. If $\infty$ is a critical point of $A$ and the limits
\begin{equation*}
 \lim_{t\uparrow +\infty}E\bigl([\lambda_1,t]\bigr) \qquad \text{and} \qquad \lim_{t\downarrow -\infty}E\bigl([t,\lambda_2]\bigr)
\end{equation*}
exist in the strong operator topology for some $\lambda_1, \lambda_2 \in \nR\setminus\{0\}$, then $\infty$ is called {\em regular critical point of} $A$. A critical point of $A$ which is not regular is called {\em singular critical point of} $A$. The set of all singular critical points of $A$ is denoted by $c_s(A)$.

The following proposition is a part of folklore in this setting. For a bounded $A$ it was proved in  \cite{Bay78}, and, in a more general form in \cite{McEn82}. For unbounded $A$ with a bounded inverse it appears in \cite{Cur} where \cite{Bay78} was cited for a proof by taking an inverse. For completeness we include a simple proof.
\begin{proposition} \label{p:fd-iff-rcp}
Let $A$ be a nonnegative operator in a Kre\u{\i}n space $\bigl(\cK,\kip\bigr)$.  Then $A$ is fundamentally reducible in $\bigl(\cK,\kip\bigr)$ if and only if  $\rho(A) \neq \emptyset$, $\ker(A) = \ker(A^2)$ and  $\infty, 0 \not\in c_s(A)$.
\end{proposition}
\begin{proof}
Since a fundamentally reducible operator is similar to a self-adjoint operator in a Hilbert space, the direct implication follows. For the converse, notice that $\infty, 0 \not\in c_s(A)$ implies the existence of the operators $E(\nR_-):=\lim_{t\downarrow-\infty} E([t,1/t])$ and $E(\nR_+):=\lim_{t\uparrow +\infty} E([1/t,t])$. It follows from (E2), (E3), (E5) and (E6) that $E(\nR_\pm) = E(\nR_\pm)^2 = E(\nR_\pm)^{[*]}$, $\bigl(E(\nR_\pm)\cK, \pm\kip \bigr)$ is a Hilbert space and $E(\nR_\pm)$ is in the double commutant of the resolvent of $A$. Since $\ker(A) = \ker(A^2)$, Propositions~5.1 and~5.6 in~\cite{La} yield
\[
\cK = E(\nR_+)\cK [\dot{+}] \ker(A) [\dot{+}] E(\nR_-)\cK,
\]
direct and orthogonal sum in $\bigl(\cK,\kip\bigr)$.  Since $\bigl(E(\nR_-)\cK [\dot{+}]E(\nR_+)\cK,\kip\bigr)$ is a Kre\u{\i}n space, \cite[Theorem~5.2]{La} implies that $\bigl(\ker(A),\kip\bigr)$ is a Kre\u{\i}n space. Let $\ker(A) = \cK_{+}^0[\dot{+}]\cK_{-}^0$ be its fundamental decomposition. Then
\[
\cK = \bigl(E(\nR_+)\cK [\dot{+}] \cK_{+}^0\bigr)[\dot{+}]\bigl(\cK_{-}^0 [\dot{+}] E(\nR_-)\cK\bigr)
\]
is a fundamental decomposition of $\cK$ which reduces $A$.
\end{proof}

In Subsection~\ref{ss:RegCPs}  we essentially use the following resolvent criterion of  K.~Veseli\'{c} \cite{Ves} for $\infty\not\in c_s(A)$. We state a special case of this criterion as it has appeared in~\cite[Corollary 1.6]{Jo}.

\begin{theorem} \label{Veselic}
Let $A$ be a nonnegative self-adjoint operator with a nonempty resolvent set in a Kre\u{\i}n space $({\cH}, [\,\cdot\,,\,\cdot\, ])$. Then:
\begin{enumerate}
\renewcommand*\theenumi{\alph{enumi}}
\renewcommand*\labelenumi{{\rm (\theenumi)}}
  \item  \label{iVeselic-1}
  $\infty\not\in c_s(A)$ if and only if the operators
 \[
  \int_{-i\eta}^{-i}+ \int_{i}^{i\eta}(A-{z})^{-1}d{z}
  \]
 are uniformly bounded for $\eta\in(1,\infty)$.
  \item \label{iVeselic-2}
   $0\not\in c_s(A)$ if and only if $\ker(A)=\ker(A^2)$ and the operators
 \[
 \int_{-i}^{-i\varepsilon}+ \int_{i\varepsilon}^{i}(A-{z})^{-1}d{z}
 \]
 are uniformly bounded for $\varepsilon\in(0,1)$.
\end{enumerate}
\end{theorem}


\section{$B$-properties of Nevanlinna functions.} \label{Sec-B-prop}

\subsection{The definition and basic properties.} \label{SubSec-B-prop-def}

\begin{definition} \label{defBinf}
Let $m$ be a Nevanlinna function. Set
\[
w_m(y) := \frac{1}{\im m(i\,\! y)}, \qquad y > 0.
\]
Let $G_{m,\infty}$ be the mapping defined for all $h \in \cH(m)$
by
\[
\bigl(G_{m,\infty} h\bigr)(y) := h(i\,\!y),  \ \ \ y > 1.
\]
A Nevanlinna function $m$ is said to have a $B_\infty$-{\em
property} if $\ran\bigl(G_{m,\infty}\bigr) \subset L_{w_m}^2\!(1,\infty)$ and the operator $G_{m,\infty}: \cH(m) \to  L_{w_m}^2\!(1,\infty)$ is bounded.

Let $G_{m,0}$ be the mapping defined for all $h \in
\cH(m)$ by
\[
\bigl(G_{m,0} h\bigr)(y) := h(i\,\!y),  \ \ \ 0  < y < 1.
\]
A Nevanlinna function $m$ is said to have a $B_0$-{\em property}
if $\ran\bigl(G_{m,0}\bigr) \subset L_{w_m}^2\!(0,1)$ and the operator $G_{m,0}: \cH(m) \to L_{w_m}^2\!(0,1)$ is bounded.
\end{definition}
The following proposition is a straightforward consequence of
Theorem~\ref{tFt} and the above definition.
\begin{proposition} \label{pGF}
Let $S$ be a closed simple symmetric operator with defect numbers $(1,1)$ in a Hilbert space $\bigl(\cH,\ahip_{\cH}\bigr)$. Let
$\bigl(\nC,\Gamma_0, \Gamma_1\bigr)$ be a boundary triple for $S^*$
and let $m$ be the corresponding Weyl function.  Then $m$ has
$B_\infty$-property {\rm (}$B_0$-property{\rm )} if and only if $G_{m,\infty} F_m$ is a bounded mapping from $\cH$ into $L_{w_m}^2\!(1,\infty)$ {\rm (}$L_{w_m}^2\!(0,1)$, respectively{\rm )}.  Moreover, if $m$ has a $B_\infty$-property {\rm (}$B_0$-property{\rm )}, then $\|G_{m,\infty}\| = \|G_{m,\infty} F_m\|$ {\rm (}$\|G_{m,0}\| = \|G_{m,0} F_m\|$, respectively{\rm )}.
\end{proposition}

In the next lemma, for a Nevanlinna function $m$, we introduce the operator $G_{m,\infty}^-$.

\begin{lemma}\label{prop:4.5}
Let $m$ be a Nevanlinna function. Then:
\begin{enumerate}
\renewcommand*\theenumi{\roman{enumi}}
\renewcommand*\labelenumi{{\rm (\theenumi)}}
  \item \label{p45i}
  $m$ has $B_\infty$-property if and only if
 \[
\bigl(G_{m,\infty}^- h\bigr)(y) := h(-i\,\!y),  \ \ \ y > 1.
\]
is a bounded operator from $\cH(m)$ into $L_{w_m}^2\!(1,\infty)$.
  \item \label{p45ii}
   $m$ has $B_0$-property if and only if
      \[
\bigl(G_{m,0}^- h\bigr)(y) := h(-i\,\!y),  \ \ \ y \in(0,1).
\]
is a bounded operator from $\cH(m)$ into $L_{w_m}^2\!(0,1)$.
\end{enumerate}
\end{lemma}
\begin{proof}
For a Nevanlinna function $m$ define the anti-linear mapping
$W: \cH({m}) \to \cH({m})$ by
\[
(Wf)(z) = f(z^*)^*, \quad z \in \nC\setminus\nR, \quad f \in \cH({m}).
\]
Then $W$ is a bijection and
\begin{equation} \label{eq:Wh}
\bigl\la Wf,Wg \bigr\ra = \la f,g \ra^*  \qquad \text{for all} \qquad f,g \in \cH({m}).
\end{equation}
Further, for every $\omega\in\nCR$ and $y > 0$ we have
\begin{align*}
\bigl(G_{m,\infty}^- K_\omega\bigr)(y)&=K_\omega(-iy)=\frac{m(-iy)-m(\omega)^*}{-iy-\omega^*}\\
&=K_{\omega^*}(iy)^*=\left(\bigl(G_{m,\infty} K_{\omega^*}\bigr)(y)\right)^*=\left(\bigl(G_{m,\infty}W K_\omega\bigr)(y)\right)^*.
\end{align*}
Together with~\eqref{eq:Wh}, this implies statements~(\ref{p45i}) and~(\ref{p45ii}).
\end{proof}

M\"{o}bius transformations of Nevanlinna functions preserve $B_\infty$-property and $B_0$-property. In the notation of  Theorem~{\rm\ref{tV}} we have the following proposition.

\begin{proposition} \label{thm:Binv}
\begin{enumerate}
\renewcommand*\theenumi{\alph{enumi}}
\renewcommand*\labelenumi{{\rm (\theenumi)}}
\item  \label{iBinv-1}
Assume that $\mu_1(z) = z$ or $\mu_1(z) = -z$ and let $\mu_2$ be a M\"{o}bius transformation as in \eqref{eq:MobNev12}. A Nevanlinna function $m$ has $B_\infty$-property {\rm (}$B_0$-property{\rm )} if and only if the Nevanlinna function $\wh{m}$ has $B_\infty$-property {\rm (}$B_0$-property, respectively{\rm )}.

\item \label{iBinv-2}
Let $\mu_1(z) = 1/z$ or $\mu_1(z) = -1/z$ and let $\mu_2$ be a  M\"{o}bius transformation as in \eqref{eq:MobNev12}.  A Nevanlinna function $m$ has $B_\infty$-property {\rm (}$B_0$-property{\rm )} if and only if the Nevanlinna function $\wh{m}$ has $B_0$-property {\rm (}$B_\infty$-property, respectively{\rm )}.
\end{enumerate}
\end{proposition}
\begin{proof}
To prove (\ref{iBinv-1}), let $\mu_1(z) = z$ and calculate
\[
w_{\wh{m}}(y) = |c_2 m(iy)+ d_2|^2 \, w_m(y), \quad y > 0.
\]
For $f \in \cH(m)$ we have
\begin{align*}
\| G_{\wh{m},\infty} Vf \|^2_{L_{w_{\wh{m}}}^2\!(1,\infty)}
 & = \int_{1}^{\infty} \frac{1}{|c_2m(iy)+d_2|^2}
  \left| f(iy)\right|^2  \, w_{\wh{m}}(y) dy \\
 &=\int_{1}^{\infty}  \left| f(iy)\right|^2  \, w_{m}(y) dy \\
 & = \| G_{m,\infty} f \|^2_{L_{w_m}^2\!(1,\infty)}.
\end{align*}
The above equality, Theorem~\ref{tV}, and elementary arguments
yield the proposition. The statement for $\mu_1(z) = -z$ follows from Lemma~\ref{prop:4.5}.

For (\ref{iBinv-2}), let $\mu_1(z) = -1/z$ and calculate
\[
w_{\wh{m}}(y) = |c_2 m(i/y)+ d_2|^2 \, w_m( 1/y), \quad y > 0.
\]
For $f \in \cH(m)$ we have
\begin{align*}
\| G_{\wh{m},\infty} V f \|^2_{L_{w_{\wh{m}}}^2\!(1,\infty)}
 & = \int_{1}^{\infty} \frac{1}{y^2} \frac{1}{|c_2m(i/y)+d_2|^2}
  \left| f(i/y)\right|^2  \, w_{\wh{m}}(y) dy \\
  & = \int_{1}^{\infty} \frac{1}{y^2} | f(i/y)|^2 w_m(1/y) dy  \\
 & = \int_{0}^{1} | f(i s )|^2 w_m(s) ds  \\
 & = \| G_{m,0}f \|^2_{L_{w_m}^2\!(0,1)}.
\end{align*}
This identity, Theorem~\ref{tV} and elementary arguments yield claim (\ref{iBinv-2}). The statement for $\mu_1(z) = 1/z$ follows from Lemma~\ref{prop:4.5}.
\end{proof}

Next, we characterize $B$-properties in terms of a bounded operator between weighted $L^2$ spaces.

\begin{proposition} \label{prop:B0prop}
Let $m$ be a Nevanlinna function with the integral representation
\begin{equation} \label{eq:IntRep+}
m(z) = a +\int_{0}^{+\infty} \left( \frac{1}{x-z} - \frac{x}{1+x^2}
\right) d\sigma(x)
\end{equation}
where $a$ is a real number and $\sigma(x)$ is a non-decreasing
function such that~\eqref{eq:sigmaCond} holds.
The following statements hold.
\begin{enumerate}
\renewcommand*\theenumi{\alph{enumi}}
\renewcommand*\labelenumi{{\rm (\theenumi)}}
\item \label{B0prop-i}
The function $m$ has $B_\infty$-property if and only if the mapping  $H_{m,\infty}$ defined by
\begin{equation*}
\bigl(H_{m,\infty} f \bigr)(y) : =  \int_{0}^{+\infty}\frac{f(x)}{x+y}d\sigma(x),
 \quad f \in L_{\sigma}^2({\mathbb R}_+), \quad y > 1,
\end{equation*}
is a bounded operator from $L_{\sigma}^2({\mathbb R}_+)$ into
$L_{w_m}^2\!(1,\infty)$.

\item \label{B0prop-ii}
The function $m$ has $B_0$-property if and only if the mapping  $H_{m,0}$ defined by
\begin{equation*}
\bigl(H_{m,0} f \bigr)(y) : =  \int_{0}^{+\infty}\frac{f(x)}{x+y}d\sigma(x),
 \quad f \in L_{\sigma}^2({\mathbb R}_+), \quad 0 <  y < 1,
\end{equation*}
is a bounded operator from $L_{\sigma}^2({\mathbb R}_+)$ into
$L_{w_m}^2\!(0,1)$.
\end{enumerate}
\end{proposition}

\begin{proof}
It follows from~\cite[Theorem~3.2]{AG04} that the mapping
\begin{equation}\label{eq:Lm}
f \mapsto \int_{0}^{+\infty}\frac{f(x)}{x-z}d\sigma(x)
\end{equation}
is an isomorphism between $L^2_{\sigma}({\mathbb R}_+)$ and $\cH(m)$.  Therefore, by Definition~\ref{defBinf},  $m$ has $B_\infty$-property if and only if the composition of \eqref{eq:Lm} and $G_{m,\infty}$, that is, \begin{equation}\label{eq:Tm}
    f \mapsto \int_{0}^{+\infty}\frac{f(x)}{x-iy}d\sigma(x)
\end{equation}
is a bounded operator from $L^2_{\sigma}({\mathbb R}^+)$ into
$L^2_{w_m}(1,\infty)$. The inequalities
\[
\frac{1}{x+y}\le \frac{1}{|x-iy|}\le \frac{\sqrt{2}}{x+y}, \quad
(x,y\in{\mathbb R}_+)
\]
imply that the operator in \eqref{eq:Tm} is a bounded operator from
$L^2_{\sigma}({\mathbb R}^+)$ into $L^2_{w_m}(1,\infty)$ if and only if
$H_{m,\infty}$ is a bounded operator from
$L^2_{\sigma}({\mathbb R}^+)$ into $L^2_{w_m}(1,\infty)$. This proves (\ref{B0prop-i}).

The second statement is proved similarly.
\end{proof}

\subsection{Sufficient conditions for $B$-properties of Nevanlinna functions}

We will need the following Schur test, see \cite[Theorem~5.9.2]{Gar}, for boundedness of integral operators.
\begin{lemma}\label{lem:Ap1}
Let $(X,\Sigma_1,\sigma_1)$ and $(Y,\Sigma_2,\sigma_2)$ be $\sigma$-finite measure spaces and let $K(x,y)$ be a nonnegative measurable function on a product space $(X,\Sigma_1,\sigma_1)\times(Y,\Sigma_2,\sigma_2)$. Suppose that there exist strictly positive measurable functions $q_1$ on $(X,\Sigma_1,\sigma_1)$ and $q_2$ on $(Y,\Sigma_2,\sigma_2)$ such that the function
\begin{equation*}
y \mapsto \frac{1}{q_2(y)}\int_X K(x,y)q_1(x)d\sigma_1(x)
\end{equation*}
is essentially bounded on $(Y,\Sigma_2,\sigma_2)$ and the function
\begin{equation*} 
x \mapsto \frac{1}{q_1(x)} \int_Y K(x,y)q_2(x)d\sigma_2(y)
\end{equation*}
is essentially bounded on $(X,\Sigma_1,\sigma_1)$.
Then for all $f \in L_{\sigma_1}^2\!(X)$ we have that
\[
(Af)(y) = \int_X K(x,y)f(x)d\sigma_1(x)
\]
exists for $\sigma_2$-almost all $y \in Y$, $Af \in L_{\sigma_2}^2\!(Y)$ and $A$ is a bounded operator from $L_{\sigma_1}^2\!(X)$ to $L_{\sigma_2}^2\!(Y)$.
\end{lemma}

Combining Proposition~\ref{prop:B0prop} and Lemma~\ref{lem:Ap1} one
obtains the following sufficient conditions for $B_\infty$ and
$B_0$-properties.
\begin{corollary} \label{cor:Binfprop}
Let $m$ have the integral representation \eqref{eq:IntRep+} and
assume that there are strictly positive measurable functions $q_1$
on ${\mathbb R}_+$ and $q_2$ on $(1,\infty)$ such that the function
\begin{equation}\label{eq:CorSchur1}
 y \mapsto  \frac{1}{q_2(y)}  \int_{0}^{+\infty}\frac{q_1(x)}{x+y}d\sigma(x)
\end{equation}
is essentially bounded on $y\in(1,\infty)$ and the function
\begin{equation}\label{eq:CorSchur2}
 x \mapsto \frac{1}{q_1(x)}   \int_{1}^{+\infty}\frac{q_2(y)}{x+y}w_m(y)dy
\end{equation}
is $\sigma$-essentially bounded on $\nR_+$.  Then the function $m$ has $B_\infty$-property.
\end{corollary}
\begin{proof}
The assumptions of the corollary and Lemma~\ref{lem:Ap1} imply that the operator $H_{m,\infty}$ is bounded, so, by Proposition~\ref{prop:B0prop}, $m$ has $B_\infty$-property.
\end{proof}

In Theorem~\ref{th:Binfprop} and its corollaries below we give sufficient conditions for $B_\infty$-property and \mbox{$B_0$-property} in terms of the asymptotics of a Nevanlinna function $m$.  The key step in the proof of this theorem is the following Abelian and Tauberian theorem for the Stieltjes transform at the point $\infty$ which is essentially contained in \cite{Ben89} and~\cite{CarHay81}, see also~\cite{Car11} for a more general theorem, or \cite{Ka31} for the classical Karamata Tauberian theorem.

\begin{theorem} \label{thm:TaubFinal}
Let $m$ be a Stieltjes function with the integral representation
\begin{equation} \label{eq:IntRep0+}
m({z}) = \int\limits_{0}^{+\infty} \frac{d\sigma(t)}{t-{z}}
\end{equation}
where $\sigma(t)$ is a non-decreasing function satisfying \eqref{eq:sigma-nor} and \eqref{eq:int_S_gc}.  Let $C > 0$ and $\alpha \in (0,1)$.  Then
\begin{equation}\label{eq:mAsympInf-A}
\text{for all} \ \ z \in \nCR \quad \text{we have} \quad m(r z) \sim \frac{C}{(-r z)^\alpha} \quad \text{as} \quad r \to +\infty,
\end{equation}
if and only if
\begin{equation}\label{eq:SpAsympInf-A}
\sigma(t)\sim \frac{C\sin(\pi\alpha)}{\pi}\frac{t^{1-\alpha}}{1-\alpha} \quad \text{as} \quad t \to +\infty.
\end{equation}
\end{theorem}
\begin{proof}
The direct part of Theorem~\ref{thm:TaubFinal} is a
special case of the Tauberian Theorem~7.5 from~\cite{Ben89}.

Assume now that $\sigma$ satisfies~\eqref{eq:SpAsympInf-A}. Integrating by parts
in~\eqref{eq:IntRep0+} and using~\eqref{eq:SpAsympInf-A} one obtains
\[
m(z) = \int\limits_{0}^{+\infty} \frac{\sigma(t)dt}{(t-z)^2}.
\]
Now the asymptotic~\eqref{eq:mAsympInf-A} is implied by the Abelian
Theorem~3.2 from~\cite{CarHay81}.
\end{proof}

\begin{theorem} \label{th:Binfprop}
Let $m$ be a Stieltjes function. If there exist $\alpha\in(0,1)$ and $C > 0$ such that for all  $z \in\nCR$ we have
\begin{equation}\label{eq:AsymInf}
 m(rz) \sim \frac{C}{(-r z)^\alpha}  \quad \text{as} \quad r \to +\infty,
\end{equation}
then $m$ has $B_\infty$-property.
\end{theorem}
\begin{proof}

Let $m$ be a Stieltjes function with the integral representation \eqref{eq:int_S} for which  \eqref{eq:int_S_gc} and \eqref{eq:sigma-nor} hold.  Assume \eqref{eq:AsymInf}. Since by \cite[Lemma~S1.3.1]{KaKr74} and \eqref{eq:AsymInf} for $\gamma$ in \eqref{eq:int_S} we have
\[
\gamma = \lim\limits_{y\uparrow+\infty} m(iy) = 0,
\]
Theorem~\ref{thm:TaubFinal} applies.

Consequently, there exists $C_1 = \frac{C\sin(\pi\alpha)}{\pi(1-\alpha)} > 0$ such that
\begin{equation} \label{eqasigma}
\sigma(t) \sim C_1 \, t^{1-\alpha }  \quad \text{as} \quad  t\to +\infty.
\end{equation}
Let $\beta > 0$ be such that $\alpha + \beta < 1$. Set
\begin{align*}
q_1(x) &:= \frac{1}{(1+x)^\beta} \\
\intertext{and}
q_2(y) &:= \int_{0}^{\infty}\frac{d\sigma_1(x)}{x+y}, \quad \text{where} \quad
\sigma_1(x)=\int_{0}^{x}\frac{d\sigma(t)}{(1+t)^\beta}.
\end{align*}
Notice that by definitions of $q_1$ and $q_2$ we have
\begin{equation*}  
\int_0^\infty \frac{q_1(x)}{x+y} d\sigma(x) = q_2(y) \quad \text{for all} \quad  y > 1.
\end{equation*}
Hence, $q_1$ and $q_2$ satisfy \eqref{eq:CorSchur1} in Corollary~\ref{cor:Binfprop}. The rest of the proof is a verification of \eqref{eq:CorSchur2}.

Integration by parts in the formula for $\sigma_1$ yields
\begin{equation} \label{eq2ndsig}
 \sigma_1(x)=\frac{\sigma(x)}{(1+x)^\beta} + \beta \int_{0}^{x} \frac{\sigma(t)}{(1+t)^{\beta+1}} dt .
\end{equation}
It follows from \eqref{eqasigma} that
\[
\frac{\sigma(t)}{(1+t)^{\beta+1}} \sim \frac{C_1}{t^{\alpha + \beta}}
 \quad \text{as} \quad t\to +\infty,
\]
and, by l'H\^{o}pital's rule,
\[
\int_{0}^{x} \frac{\sigma(t)}{(1+x)^{\beta+1}} dt \sim
  C_2 \, x^{1- \alpha - \beta} \quad \text{as} \quad x\to +\infty.
\]
Since also
\[
\frac{\sigma(t)}{(1+t)^{\beta}} \sim  C_1  \, t^{1-\alpha - \beta}
 \quad \text{as} \quad t\to +\infty,
\]
the equality \eqref{eq2ndsig} implies
\begin{equation} \label{eqsig1as}
\sigma_1(x) \sim C_3 \, x^{1-\alpha - \beta} \quad \text{as} \quad  x\to +\infty,
\end{equation}
Now \eqref{eqsig1as} and Theorem~\ref{thm:TaubFinal} yield
\begin{equation}\label{eq:4.12}
q_2(y) \sim \frac{C_4}{y^{\alpha+\beta}} \quad \text{as} \quad  y\to\infty.
\end{equation}
for some $C_4 > 0$. Next consider the function
\[
x \mapsto \int_{1}^{\infty}\frac{q_2(y)}{x+y}w_m(y)dy, \quad x > 0.
\]
Since by \eqref{eq:4.12} and assumption \eqref{eq:AsymInf} we have
\[
q_2(y)w_m(y) \sim \frac{C_5}{y^{\beta}} \quad \text{as} \quad y\to\infty,
\]
for some $C_5 > 0$, Theorem~\ref{thm:TaubFinal} yields
\[
\int_{1}^{\infty}\frac{q_2(y)}{x+y}w_m(y)dy \sim \frac{C_6}{x^{\beta}} \quad \text{as} \quad x\to\infty.
\]
The last displayed relationship implies that the function
\[
x \mapsto \frac{1}{q_1(x)} \int_{1}^{\infty}\frac{q_2(y)}{x+y}w_m(y)dy, \quad x > 0,
\]
is bounded; that is \eqref{eq:CorSchur2} holds. Now Corollary~\ref{cor:Binfprop} implies that $m$ has $B_\infty$- property.
\end{proof}

\begin{corollary} \label{co:Binfprop}
Every function in $\cA_\infty$ has $B_\infty$-property.
\end{corollary}
\begin{proof}
Let $m \in \cA_\infty$ be arbitrary. Then there exists a M\"{o}bius transformation $\mu$ such that $\mu\circ m$ is a Stieltjes function and  \eqref{eq:mAsymInf-defAInf} holds. By Theorem~\ref{th:Binfprop} the Nevanlinna function $\mu\circ m$ has $B_\infty$-property. In turn,  Proposition~\ref{thm:Binv}(\ref{iBinv-1}) yields that $m =  \mu^{-1}\circ\mu\circ m$ also has $B_\infty$-property.
\end{proof}

\begin{corollary} \label{th:B0prop}
Every function in $\cA_0$ has $B_0$-property.
\end{corollary}
\begin{proof}
Let $m \in \cA_0$ be arbitrary. Set $\mu_1(z) = 1/z$ and $\mu_2(z) = -z$.  By Proposition~\ref{prop:AInf-AZer}(\ref{i:AInf-AZer-2}) the function $\wh{m} = \mu_2\circ m\circ \mu_1$ belongs to $\cA_\infty$, so $\wh{m}$ has $B_\infty$-property by Corollary~\ref{co:Binfprop}. Now, Proposition~\ref{thm:Binv}(\ref{iBinv-2}) implies that $m$ has $B_0$-property.
\end{proof}

\section{Coupling of symmetric operators} \label{SecCoup}

\subsection{Coupling of symmetric operators in a Hilbert space.} \label{SubS:CoupleHS}
In this section we consider Hilbert spaces $\bigl( \cH_+, \ahip_{\cH_+} \bigr)$ and $\bigl( \cH_-, \ahip_{\cH_-} \bigr)$ and their (external) direct sum ${\cH} = \cH_+\oplus \cH_-$ with the natural inner product $\ahip_{\cH}$.  If  $T_+$ is an operator in $\cH_+$ and  $T_-$ is an operator in $\cH_-$, then $T_+ \oplus T_-$ denotes their direct sum, that is
\begin{equation*} \label{edefwAHC}
\bigl(T_+ \oplus T_-\bigr) \begin{pmatrix} f_+ \\ f_-
\end{pmatrix}:=
\begin{pmatrix} T_+ f_+ \\ T_- f_- \end{pmatrix}, \quad  \ f_+ \in \dom(T_+), \  f_- \in \dom(T_-).
\end{equation*}

The following assumptions apply to all four statements in this subsection. We assume that $S_\pm$ is a closed symmetric densely defined operator with defect numbers $(1,1)$ in the Hilbert space $\bigl( \cH_\pm, \ahip_{\cH_\pm} \bigr)$.  We let $\bigl(\nC,\Gamma_{0}^\pm, \Gamma_{1}^\pm\bigr)$ be a boundary triple for $S_\pm^*$ and $m_\pm$  and $\psi_\pm$ are the corresponding Weyl function and the Weyl solution. By $S_0^\pm$ we denote the self-adjoint extension of $S_\pm$ which is defined on $\dom(S_0^\pm) = \ker(\Gamma_{0}^\pm)$. That is
$S_0^\pm=S_\pm^*|_{\ker(\Gamma_{0}^\pm)}$.

In the following theorem we reformulate results from~\cite{DHMS}, (see also~\cite{CDR}) in a form which is convenient for further use in this paper.

\begin{theorem} \label{tHsc}
Under the general assumptions of this subsection we have:
\begin{enumerate}
\renewcommand*\theenumi{\alph{enumi}}
\renewcommand*\labelenumi{{\rm (\theenumi)}}
\item
The linear operator $S$ defined as the restriction of $S_+^* \oplus S_-^*$ to the domain
\begin{equation*}
\dom(S) = \left\{
\begin{pmatrix}\!f_+ \\ f_-\!\end{pmatrix}: \!
\begin{array}{l}
 \Gamma_{0}^+(f_+) = \Gamma_{0}^- (f_-)=0,  \\
 \Gamma_{1}^+(f_+) + \Gamma_{1}^- (f_-)=0,
\end{array} \ \begin{array}{l}  f_+ \in \dom(S_+^*), \\ f_-\in \dom(S_-^*)
 \end{array}
 \right\}
\end{equation*}
is, closed, densely defined and symmetric with defect
numbers $(1,1)$ in the Hilbert space ${\cH}$.
\item \label{H2b}
The adjoint ${S}^*$ of ${S}$ is the restriction of $S_+^* \oplus S_-^*$ to the domain
\begin{equation*}
\dom(S^*) = \left\{\!
\begin{pmatrix}\! f_+ \\ f_- \!\end{pmatrix} :
 \Gamma_{0}^+(f_+) - \Gamma_{0}^- (f_-) = 0, \ \begin{array}{l}  f_+ \in \dom(S_+^*), \\ f_-\in \dom(S_-^*)
 \end{array} \right\}.
\end{equation*}
\item \label{ibatr}
A boundary triple $\bigl(\nC,\Gamma_0,\Gamma_1\bigr)$ for $S^*$ is
given by
\begin{equation} \label{Htripl}
\Gamma_0 f = \Gamma_0^+ f_+, \quad
 \Gamma_1 f = \Gamma_1^+ f_+ + \Gamma_1^-f_-, \quad
  f = \begin{pmatrix} f_+ \\ f_- \end{pmatrix}  \in \dom(S^*).
\end{equation}
\item \label{H2d}
The Weyl function of $S$ relative to the boundary triple $\bigl(\nC,\Gamma_0,\Gamma_1\bigr)$ is
\begin{equation*} \label{mH2}
 m(z) = m_+(z) + m_-(z), \quad z \in \nC \setminus \nR.
\end{equation*}
\item \label{ietHsc}
The self-adjoint extension $S_1$ of $S$ such that $\dom(S_1) = \ker(\Gamma_1)$ coincides with the restriction of $S_+^* \oplus S_-^*$
to the domain
\begin{equation} \label{eq:domwAHC}
\dom(S_1) = \left\{\begin{pmatrix} f_+ \\ f_- \end{pmatrix}:
  \begin{array}{l}
 \Gamma_{0}^+(f_+) - \Gamma_{0}^- (f_-)=0,  \\
 \Gamma_{1}^+(f_+) + \Gamma_{1}^- (f_-)=0,
 \end{array} \  \begin{array}{l}  f_+ \in \dom(S_+^*), \\ f_-\in \dom(S_-^*)
 \end{array}
 \right\}.
\end{equation}
\item \label{ietHsc0}
The self-adjoint extension $S_0$ of $S$ such that $\dom(S_0) = \ker(\Gamma_0)$ coincides with the restriction of $S_+^* \oplus S_-^*$
to the domain
\begin{equation*} 
\dom(S_0) = \left\{\begin{pmatrix} f_+ \\ f_- \end{pmatrix}:\,
 \Gamma_{0}^+(f_+) = \Gamma_{0}^- (f_-)=0, \ \begin{array}{l}  f_+ \in \dom(S_+^*), \\ f_-\in \dom(S_-^*)
 \end{array} \right\}
\end{equation*}
and thus $S_0=S_0^+\oplus S_0^-$.
\end{enumerate}
\end{theorem}

The operator $S_1$ defined in Theorem~\ref{tHsc}(\ref{ietHsc}) is called the {\it coupling} of the operators $S_+$ and $S_-$ in the Hilbert space $\bigl( {\cH}, \ahip_{{\cH}} \bigr)$ relative to the triples $\bigl(\nC,\Gamma_{0}^+, \Gamma_{1}^+\bigr)$ and $\bigl(\nC,\Gamma_{0}^-, \Gamma_{1}^-\bigr)$.

In the next proposition we give a criterion for nonnegativity of the coupling $S_1$ of two nonnegative operators $S_+$ and $S_-$ in the Hilbert space $\bigl( {\cH}, \ahip_{{\cH}} \bigr)$.

\begin{proposition}\label{prop:nonneg}
In addition to the general assumptions of this subsection assume that $S_0^\pm$ is a nonnegative self-adjoint extension of $S_\pm$ in $\cH_\pm$. Then the coupling ${S_1}$ is a nonnegative self-adjoint operator in  $\bigl(\cH, \ahip_{\cH}\bigr)$  if and only if
\begin{equation} \label{eqN}
m_+(x) + m_-(x) \neq 0 \qquad \text{for all} \qquad x \in \nR_-.
\end{equation}
\end{proposition}
\begin{proof}
Since the self-adjoint extensions $S_0^+$  and $S_0^-$ are nonnegative, then the Weyl functions $m_+$ and $m_-$ are holomorphic on $\nR_-$ and $(-\infty,0)\subset\rho(S_0)$, where $S_0=S_0^+\oplus S_0^-$.  As, by  Theorem~\ref{tHsc}(\ref{H2d}), $m(z)=m_+(z)+m_-(z)$ is the Weyl function of $S$  relative to the boundary triple~\eqref{Htripl}, it follows from Proposition~\ref{prop:S_01} that~\eqref{eqN} holds if and only if $(-\infty,0)\subset\rho(S_1)$, or, equivalently, if and only if $S_1$ is nonnegative.
\end{proof}

\begin{corollary} \label{col:nonneg}
In addition to the general assumptions of this subsection assume that  the operator $S_0^+$ is the Kre\u{\i}n extension of $S_+$. Then the coupling ${S_1}$ is a nonnegative self-adjoint operator  in the Hilbert space $\bigl( {\cH}, \ahip_{{\cH}} \bigr)$ if and only if the following two conditions are satisfied:
\begin{enumerate}
\renewcommand*\theenumi{\roman{enumi}}
\renewcommand*\labelenumi{{\rm (\theenumi)}}
  \item
   $S_0^-:=S_-^*|_{\ker(\Gamma_0^-) }$ is a non-negative operator  in the Hilbert space $\cH_-$.
  \item \label{eq:m_+-inf}
  $\lim_{x\downarrow -\infty} \bigl( m_+(x) + m_-(x) \bigr) \geq 0$.
\end{enumerate}
If the coupling ${S_1}$ is a nonnegative, then there exists a constant $c\in\nR$ such that $m_+-c$ and $m_-+c$ are Stieltjes functions.
\end{corollary}
\begin{proof} Since $S_0^+$ is the Kre\u{\i}n extension of $S_+$, then by~\eqref{eq:m_K0}
  \begin{equation}\label{eq:m_+inf}
     \lim_{x\uparrow 0}m_+(x)=+\infty.
\end{equation}
If the coupling ${S_1}$ is nonnegative, then the operator $S_-$ is also nonnegative and hence its extension $S_0^-$ has at most one negative eigenvalue (see~\cite{Kr47}). Then the Weyl function $m_-$ has at most one pole on $\nR_-$ and by~\eqref{eq:m_+inf}
\begin{equation}\label{eq:m+-K0}
     \lim_{x\uparrow 0}\bigl(m_+(x)+m_-(x)\bigr)=+\infty.
\end{equation}

Assume that $m_-$ has a pole at $x_0\in(-\infty,0)$. Then
\[
    \lim_{x\downarrow x_0}\bigl(m_+(x)+m_-(x)\bigr)=-\infty,\quad
\]
and, hence $m_+ + m_-$ has a zero in the interval $(x_0,0)$.
By Proposition~\ref{prop:nonneg} the operator $S_1$ has a negative eigenvalue which is impossible, if the operator $S_1$ is nonnegative.
Therefore, the function $m_-$ is holomorphic on $\nR_-$ and $S_0^-$ is a non-negative operator  in the Hilbert space $\cH_-$.

Since the restrictions of $m_+$ and $m_-$ to $\nR_-$ are continuous and monotonically increasing functions which satisfy the condition~\eqref{eq:m+-K0}, the condition~\eqref{eqN} of nonnegativity of $S_1$ can be rewritten as~\eqref{eq:m_+-inf}.

Conversely, if (i) and (ii) hold, then the operators $S_0^+$ and $S_0^-$ are nonnegative  in $\cH_+$ and $\cH_-$, respectively, $m_-$ is holomorphic on $\nR_-$ and~\eqref{eqN} holds. Therefore, the coupling $S_1$ is nonnegative by Proposition~\ref{prop:nonneg}.

To prove the last statement,  set
$
c = \lim_{x\downarrow -\infty}m_+(x).
$
Then $m_+(x)-c \ge 0$ and by (ii)
\[
m_-(x)+ c \ge \lim_{x\downarrow -\infty}m_-(x)+\lim_{x\downarrow -\infty}m_+(x)\ge 0
\]
for all $x\in\nR_-$. Therefore $m_+-c$ and $m_-+c$ are Stieltjes functions.
\end{proof}
The next corollary is proved similarly.
\begin{corollary} \label{col:nonnegF}
In addition to the general assumptions of this subsection assume that the operator $S_0^+$ is the Friedrichs extension of $S_+$.
 Then the coupling ${S_1}$ is a nonnegative self-adjoint operator  in the Hilbert space $\bigl( {\cH}, \ahip_{{\cH}} \bigr)$ if and only if the following two conditions are satisfied:
\begin{enumerate}
\renewcommand*\theenumi{\roman{enumi}}
\renewcommand*\labelenumi{{\rm (\theenumi)}}
  \item
   $S_0^-:=S_-^*|_{\ker(\Gamma_0^-) }$ is a non-negative operator  in $\cH_-$.
  \item
  $m_+(0) + m_-(0) =\lim_{x\uparrow 0-}  \bigl(m_+(x)+m_-(x) \bigr) \leq 0.$
  \end{enumerate}
If the coupling $S_1$ is nonnegative, then there exists a constant $c\in\nR$ such that $m_+-c$ and $m_-+c$ are inverse Stieltjes functions.
\end{corollary}

\subsection{Non-separated extensions as couplings.}

In Theorem~\ref{tHsc}(\ref{ietHsc}) we defined the coupling $S_1$ of symmetric operators $S_+$ and $S_-$ relative to the boundary triples $\bigl(\nC,\Gamma_0^+,\Gamma_1^+\bigr)$ for $S_+^*$ and $\bigl(\nC,\Gamma_0^-,\Gamma_1^-\bigr)$  for $S_-^*$.  Using different boundary triples would result in a  different coupling operator.  In this subsection we characterize all self-adjoint extensions of $S_+\oplus S_-$ which can be obtained as couplings of $S_+$ and $S_-$.

Let, as before, $\bigl(\nC,\Gamma_0^\pm,\Gamma_1^\pm\bigr)$ be a boundary triple for $S_\pm^*$. It is not difficult to see (cf.~\cite[Lemma~3.4]{CDR}, \cite{DHMS} for a more general setting or \cite[Appendix~II,~125]{AG78} for the case of differential operators) that a restriction $\wh{S}$ of $S_+^*\!\oplus\!S_-^*$ is a self-adjoint extension of $S_+\oplus S_-$ in $\bigl(\cH, \ahip_{\cH}\bigr)$ if and only if
\[
\dom(\wh{S}) = \left\{ \left(\!\!\!\begin{array}{c} f_+ \\ f_-
\end{array}\!\!\!\right) \in \bigl(\dom S_+^*\bigr) \oplus \bigl(\dom S_-^*\bigr)  :
M\! \left(\!\!\!\begin{array}{c} \Gamma_1^+ f_+  \\ \Gamma_1^- f_-
\end{array}\!\!\!\!\right) = N\! \left(\!\!\!\begin{array}{c} \Gamma_0^+ f_+  \\ \Gamma_0^- f_-
\end{array}\!\!\!\!\right) \right\},
\]
where $M$ and $N$ are $2\times 2$ matrices with complex entries such that the block matrix $\bigl( M \ N \bigr)$ has rank $2$ and  the matrix $MN^*$ is self-adjoint.

If $\wh{S}$ is a self-adjoint extension of $S_+\oplus S_-$ in $\bigl(\cH, \ahip_{\cH}\bigr)$ then either
\begin{equation} \label{eq:nssa}
(\gr \wh{S}) \cap   \bigl(\cH_+ \oplus \{0\}\bigr)^2 = \gr(S_+) \quad \text{and}
 \quad (\gr \wh{S}) \cap \bigl(\{0\}\oplus \cH_-\bigr)^2 = \gr(S_-)
\end{equation}
or
\begin{equation*} 
(\gr \wh{S}) \cap   \bigl(\cH_+ \oplus \{0\}\bigr)^2  \quad \text{and}
\quad (\gr \wh{S}) \cap \bigl(\{0\}\oplus \cH_-\bigr)^2
\end{equation*}
are graphs of self-adjoint operators in $\bigl(\cH_+, \ahip_{\cH_+}\bigr)$ and $\bigl(\cH_-, \ahip_{\cH_-}\bigr)$, respectively.

In the later case we say that the self-adjoint operator $\wh{S}$ is a {\em separated} extension of $S_+\oplus S_-$ and the corresponding boundary conditions are called {\em separated} boundary conditions. If \eqref{eq:nssa} holds, then we say that $\wh{S}$ is a {\em non-separated} extension of $S_+\oplus S_-$ and the corresponding boundary conditions are called {\em non-separated} boundary conditions.

Notice that multiplying matrices $M$ and $N$ from the left by the same invertible matrix does not change the domain of $\wh{S}$. Therefore, we can assume that the block matrix $\bigl( M \ N \bigr)$ is in reduced row echelon form, see \cite[Subsection~0.3.4]{HJ}.  As the rank of $\bigl( M \ N \bigr)$ is $2$, the reduced row echelon form of $\bigl( M \ N \bigr)$ takes one of the following six forms:
\begin{alignat*}{3}
&
\left(\!\!\!\begin{array}{cccc}
0 & 0 & 1 & 0 \\
0 & 0 & 0 & 1 \end{array}\!\!\!\right)  & \qquad
&
\left(\!\!\!\begin{array}{cccc}
0 & 1 & * & 0 \\
0 & 0 & 0 & 1 \end{array}\!\!\!\right) & \qquad
&
\left(\!\!\!\begin{array}{cccc}
1 & * & * & 0 \\
0 & 0 & 0 & 1 \end{array}\!\!\!\right)  \\
&
\left(\!\!\!\begin{array}{cccc}
0 & 1 & 0 & * \\
0 & 0 & 1 & * \end{array}\!\!\!\right) & \qquad
&
\left(\!\!\!\begin{array}{cccc}
1 & * & 0 & * \\
0 & 0 & 1 & * \end{array}\!\!\!\right) & \qquad
&
\left(\!\!\!\begin{array}{cccc}
1 & 0 & * & * \\
0 & 1 & * & * \end{array}\!\!\!\right),
\end{alignat*}
where $*$ stands for an arbitrary complex number.

A straightforward calculations show that the only matrices of the above 6 types for which the corresponding matrix $MN^*$ is self-adjoint are of the following four types:
\begin{alignat}{2} \label{eq:types-a-b}
&
\left(\!\!\!\begin{array}{cccc}
0 & 0 & 1 & 0 \\
0 & 0 & 0 & 1 \end{array}\!\!\!\right)  & \qquad
&
\left(\!\!\!\begin{array}{cccc}
1 & 0 & \alpha & 0 \\
0 & 0 & 0 & 1 \end{array}\!\!\!\right)  \\ \label{eq:types-c-d}
&
\left(\!\!\!\begin{array}{cccc}
1 & \rho e^{i \theta} & 0 & \sigma e^{i \theta} \\
0 & 0 & 1 & -e^{i \theta}/\rho \end{array}\!\!\!\right) & \qquad
&
\left(\!\!\!\begin{array}{cccc}
1 & 0 & \alpha & \omega \\
0 & 1 & \omega^* & \beta \end{array}\!\!\!\right),
\end{alignat}
where $\omega \in \nC$, $\alpha, \beta, \sigma, \theta \in \nR$, and $\rho > 0$. Clearly, the matrices in \eqref{eq:types-a-b} give rise to separated boundary conditions, while the first matrix  in \eqref{eq:types-c-d} leads to non-separated boundary conditions. The second matrix in \eqref{eq:types-c-d} leads to separated boundary conditions if and only if $\omega = 0$. A similar classification of boundary conditions for a regular Sturm-Liouville problem was established in \cite{CFK13}.

\begin{theorem}\label{Thm:5.5}
Let $S_\pm$ be a closed symmetric densely defined operator with defect numbers $(1,1)$ in the Hilbert space $\bigl(\cH_\pm,\ahip_{\cH_\pm}\bigr)$. A self-adjoint extension $\wh{S}$ of $S_+\oplus S_-$ in $\bigl(\cH, \ahip_{\cH}\bigr)$ is a coupling of $S_+$ and $S_-$ relative to some boundary triples of $S_+^*$ and $S_-^*$ if and only if $\wh{S}$ is a non-separated extension of $S_+\oplus S_-$.
\end{theorem}
\begin{proof}
Assume that $\wh{S}$ is a coupling of $S_+$ and $S_-$ relative to  boundary triples $\bigl(\nC,\Gamma_0^+,\Gamma_1^+\bigr)$ and $\bigl(\nC,\Gamma_0^-,\Gamma_1^-\bigr)$. Let
\[
\left\{\left(\!\!\!\begin{array}{c}
f_+ \\ 0
\end{array}\!\!\!\right), \left(\!\!\!\begin{array}{c}
S_+^* f_+ \\ 0
\end{array}\!\!\!\right) \right\} \in \gr \wh{S}.
\]
Then, by definition of $\wh{S}$ as a coupling, we have $\Gamma_0^+ f_+ - \Gamma_0^- 0 = 0$ and $\Gamma_1^+ f_+ + \Gamma_1^- 0 = 0$. That is, $\Gamma_0^+ f_+ = 0$ and $\Gamma_1^+ f_+ = 0$, proving that $(\gr \wh{S}) \cap   \bigl(\cH_+ \oplus \{0\}\bigr)^2 = \gr S_+$. Similarly, $(\gr \wh{S}) \cap   \bigl(\{0\}\oplus\cH_- \bigr)^2 = \gr S_-$. Hence, $\wh{S}$ is a non-separated extension.

To prove the converse, assume that $\wh{S}$ is a non-separated self-adjoint extension of $S_+\oplus S_-$ in $\bigl(\cH, \ahip_{\cH}\bigr)$. Let $\bigl(\nC,\Gamma_0^\pm,\Gamma_1^\pm\bigr)$ be a boundary triple for $S_\pm^*$.  Then the block matrix $\bigl(M \ N\bigr)$ corresponding to the boundary conditions that determine the domain of $\wh{S}$ are of two forms in \eqref{eq:types-c-d} with $\omega \neq 0$.

First we consider the boundary conditions corresponding to the first matrix in \eqref{eq:types-c-d}
\begin{alignat*}{3}
&\Gamma_0^+ f_+  & &- \bigl(e^{i \theta}/\rho\bigr) \Gamma_0^- f_- & & = 0 \\
&\Gamma_1^+ f_+  & &+ \bigl(-\sigma e^{i \theta}\bigr) \Gamma_0^- f_- + \rho e^{i \theta} \Gamma_1^- f_- & & =  0.
\end{alignat*}
These boundary conditions can be rewritten as:
\begin{equation}\label{eq:Coupl_Cond}
     \wh{\Gamma}_{0}^+(f_+) - \wh\Gamma_{0}^- (f_-)=0,  \quad
 \wh{\Gamma}_{1}^+(f_+) + \wh\Gamma_{1}^- (f_-)=0,
\end{equation}
where
\begin{equation*} 
\begin{pmatrix}
\wh\Gamma_{0}^+ (f_+) \\ \wh\Gamma_{1}^+ (f_+)
\end{pmatrix}=
\begin{pmatrix}
\Gamma_{0}^+ (f_+) \\ \Gamma_{1}^+ (f_+)
\end{pmatrix}
,\quad
\begin{pmatrix}
\wh\Gamma_{0}^- (f_-) \\ \wh\Gamma_{1}^- (f_-)
\end{pmatrix}=
e^{i\theta}
\begin{pmatrix}
1/\rho & 0\\
-\sigma & \rho
\end{pmatrix}\begin{pmatrix}
\Gamma_{0}^- (f_-) \\ \Gamma_{1}^- (f_-)
\end{pmatrix}
\end{equation*}
By Remark~\ref{rem:transp} the triple $\bigl(\nC,\wh{\Gamma}_0^-,\wh{\Gamma}_1^- \bigr)$ is a boundary triple for $S_-^*$. Therefore, the coupling of $S_+$ and $S_-$ relative to the boundary triples $\bigl(\nC,\wh{\Gamma}_0^+,\wh{\Gamma}_1^+ \bigr)$ and $\bigl(\nC,\wh{\Gamma}_0^-,\wh{\Gamma}_1^- \bigr)$ coincides with the operator $\wh{S}$.

Next we consider the boundary conditions corresponding to the second matrix in \eqref{eq:types-c-d} with $\omega \neq 0$. These boundary conditions are
\begin{alignat}{3} \label{eq:B_Cond21}
&\omega^* \Gamma_0^+ f_+ & &- \bigl( \Gamma_1^- f_- - \beta \Gamma_0^- f_- \bigr)   & & =  0 \\
\label{eq:B_Cond22}
&\frac{1}{\omega}\Gamma_1^+ f_+ - \frac{\alpha}{\omega} \Gamma_0^+ f_+
  & & + (- \Gamma_0^- f_-) & & = 0
\end{alignat}
Setting
\begin{equation*} 
\begin{pmatrix}
\wh\Gamma_{0}^+ (f_+) \\ \wh\Gamma_{1}^+ (f_+)
\end{pmatrix}=
\begin{pmatrix}
\omega^* & 0\\
-\frac{\alpha}{\omega} & \frac{1}{\omega}
\end{pmatrix}
\begin{pmatrix}
\Gamma_{0}^+ (f_+) \\ \Gamma_{1}^+ (f_+)
\end{pmatrix},\
\begin{pmatrix}
\wh\Gamma_{0}^- (f_-) \\ \wh\Gamma_{1}^- (f_-)
\end{pmatrix}=
\begin{pmatrix}
-\beta & 1\\
-1 & 0
\end{pmatrix}\begin{pmatrix}
\Gamma_{0}^- (f_-) \\ \Gamma_{1}^- (f_-)
\end{pmatrix}
\end{equation*}
one can rewrite boundary conditions~\eqref{eq:B_Cond21},\eqref{eq:B_Cond22} in the form~\eqref{eq:Coupl_Cond}.  By Remark~\ref{rem:transp} the triple $\bigl(\nC,\wh{\Gamma}_0^+,\wh{\Gamma}_1^+ \bigr)$ is a boundary triple for $S_+^*$ and the triple $\bigl(\nC,\wh{\Gamma}_0^-,\wh{\Gamma}_1^- \bigr)$ is a boundary triple for $S_-^*$. Therefore the  coupling of $S_+$ and $S_-$ relative to the boundary triple $\bigl(\nC,\wh{\Gamma}_0^+,\wh{\Gamma}_1^+ \bigr)$ for $S_+^*$ and the boundary triple $\bigl(\nC,\wh{\Gamma}_0^-,\wh{\Gamma}_1^- \bigr)$ for $S_-^*$ coincides with the operator $\wh{S}$.
This completes the proof.
\end{proof}

\subsection{A partially fundamentally reducible operator as a coupling} \label{SubS:CoupleKS}

Given two symmetric operators, $S_+$ in the Hilbert space $\cH_+$ and $S_-$ in the Hilbert space $\cH_-$, in Subsection~\ref{SubS:CoupleHS}  we constructed a coupling $S_1$ of these two operators which is a self-adjoint operator in the Hilbert space $\cH=\cH_+\oplus\cH_-$.

Next we introduce a Kre\u{\i}n space structure on $\cH$.  Let  $J$ be a self-adjoint involution on $\bigl( \cH, \ahip_{\cH} \bigr)$ defined by
\begin{equation} \label{eq:defJ}
J \begin{pmatrix} f_+ \\ f_- \end{pmatrix} := \begin{pmatrix}\! f_+\! \\ \!- f_- \!\end{pmatrix}, \quad f_+ \in \cH_+, \ f_- \in \cH_-.
\end{equation}
This involution induces an indefinite inner product on $\cH$:
\[
[f,g]_{{\cH}} := \bigl\langle {J} f,g\bigr\rangle_{{\cH}}, \quad f,g \in {\cH},
\]
and with this inner product $\bigl( {\cH}, \kip_{{\cH}} \bigr)$ is a Kre\u{\i}n space.

With $S_0$ and $S_1$ from Theorem~\ref{tHsc} we define $A_0 := J S_0$ and $A_1 := J S_1$. Since $S_0$ and $S_1$ are self-adjoint in the Hilbert space $({\cH},\ahip_\cH)$, $A_0$ and $A_1$ are self-adjoint in the Kre\u{\i}n space $(\cH,\kip_\cH)$. In fact,  $A_0$ and $A_1$  are given by
\begin{equation}\label{eq:A0}
    A_0 \begin{pmatrix} f_+ \\ f_- \end{pmatrix}=\begin{pmatrix} S_0^+ f_+ \\ -S_0^-f_-\end{pmatrix}, \quad
\begin{pmatrix} f_+ \\ f_- \end{pmatrix}\in \dom(A_0)=\dom(S_0),
\end{equation}
\begin{equation}\label{eq:A1}
    A_1 \begin{pmatrix} f_+ \\ f_- \end{pmatrix}=\begin{pmatrix} S_+^* f_+ \\ -S_-^*f_-\end{pmatrix}, \quad
\begin{pmatrix} f_+ \\ f_- \end{pmatrix}\in \dom(A_1)=\dom(S_1).
\end{equation}
The operator $A_1$  will be called the {\it coupling} of the operators $S_+$ and $-S_-$ in the Kre\u{\i}n space $\bigl( {\cH}, \kip_{{\cH}} \bigr)$.

In the rest of this subsection we will proceed in the opposite direction: Given a partially fundamentally reducible self-adjoint operator $A$ in the Kre\u{\i}n space $({\cK}, \kip_{\cK})$ we will prove that $A$ is a coupling of two operators.

Recall that Definition~\ref{def:PartFR} associates a fundamental decomposition
$\cK = \cK_+ [\dot{+}]\cK_-$ of $(\cK, \kip_{\cK})$ and symmetric operators $S_+$ and $S_-$
with a partially fundamentally reducible operator $A$ in a Kre\u{\i}n space $(\cK, \kip_{\cK})$.
As before, $S_\pm^*$ denotes the adjoint of $S_\pm$ in the Hilbert space $(\cK_\pm, \pm\kip)$. By   $P_+$ and $P_-$ we denote the orthogonal projections and by $J$ the fundamental symmetry corresponding to the fundamental decomposition $\cK = \cK_+[\dot{+}]\cK_-$, while $\ahip_\cK$ denotes the corresponding Hilbert space inner product and $\|\cdot\|_\cK$ denotes the norm induced by $\ahip_\cK$. The notation related to a partially fundamentally reducible operator $A$ introduced in this paragraph is used throughout the rest of the paper.

\begin{theorem} \label{prop:AasCoupl}
Let $A$ be a partially fundamentally reducible operator in a Kre\u{\i}n space $({\cK}, \kip_{{\cK}})$. The following statements hold.
\begin{enumerate}
\renewcommand*\theenumi{\alph{enumi}}
\renewcommand*\labelenumi{{\rm (\theenumi)}}
  \item \label{pAasCoupl-1}
We have $\dom(S_\pm^*)=P_\pm\dom(A)$ and
\begin{equation*} 
  \pm S_\pm^*P_\pm f=P_\pm Af\quad\mbox{for all}\quad f\in\dom(A).
  \end{equation*}
  \item \label{pAasCoupl-2}
Let $\bigl(\nC,\Gamma_0^+,\Gamma_1^+\bigr)$ be a boundary triple for $S_+^*$. Then the equalities
  \begin{equation}\label{eq:Gamma-}
    \Gamma_0^- P_-f = \Gamma_0^+ P_+f, \qquad \Gamma_1^- P_-f  = - \Gamma_1^+ P_+f, \qquad f\in\dom(A),
  \end{equation}
define a boundary triple $\bigl(\nC,\Gamma_0^-,\Gamma_1^-\bigr)$ for $S_-$.
  \item \label{pAasCoupl-3}
We have $A = JS_1$ where $S_1$ is the coupling of the operators $S_+$ and $S_-$ in the Hilbert space $(\cK, \ahip_\cK)$ relative to the boundary triples $\bigl(\nC,\Gamma_0^+,\Gamma_1^+\bigr)$ and $\bigl(\nC,\Gamma_0^-,\Gamma_1^-\bigr)$.
 \item
Let $A$ be a nonnegative operator in $(\cK, \kip_{\cK})$ and assume that the operator $S_0^+:=S_+^*|_{\ker(\Gamma_0^+) }$ coincides with Kre\u{\i}n's or Friedrichs' extension of $S_+$. Then $S_0^-:=S_-^*|_{\ker(\Gamma_0^-) }$ is a nonnegative operator in the Hilbert space $(\cK_-, -\kip_{{\cK}})$.
\end{enumerate}

\end{theorem}
\begin{proof}
The statements (\ref{pAasCoupl-1}) and (\ref{pAasCoupl-2}) are proved in~\cite[Lemma 5.1]{DHMS}.

(\ref{pAasCoupl-3}) It follows from~\eqref{eq:domwAHC} and \eqref{eq:Gamma-} that $\dom(A)=\dom(S_1)$. The equality $A=J\!A_1$ follows from~\eqref{eq:A1}.

The last statement is implied by Corollaries~\ref{col:nonneg} and~\ref{col:nonnegF}.
\end{proof}


The results of the next theorem can be derived from~\cite{Der95}. However, we prefer to present a direct proof for this elementary case when defect numbers of $S_\pm$ are $(1,1)$. In the case of indefinite Sturm-Liouville operator similar construction has been used in
\cite[Proposition~2.5]{KaMal07}.

\begin{theorem} \label{tKsc}
Let $A$ be a partially fundamentally reducible operator in a Kre\u{\i}n space $({\cK}, \kip_{{\cK}})$ and let $\bigl(\nC,\Gamma_0^+, \Gamma_1^+\bigr)$ and $\bigl(\nC,\Gamma_0^-, \Gamma_1^-\bigr)$ be the boundary triples from  Theorem~{\rm\ref{prop:AasCoupl}~\!\!(\ref{pAasCoupl-2})}. Let $m_\pm$ be the Weyl function and let $\psi_\pm$ be the Weyl solution of $S_\pm$ relative to the boundary triple $\bigl(\nC,\Gamma_0^\pm, \Gamma_1^\pm\bigr)$. Let the operator $A_0$ be given by~\eqref{eq:A0}. Then
\begin{equation} \label{eRwA}
\rho(A) \setminus\nR
= \bigl\{ z \in \nCR :  m_+(z) + m_-(-z) \neq 0 \bigr\}.
\end{equation}
If $\rho(A) \setminus\nR $ is a nonempty set, then for every $\ z \in \rho(A) \setminus\nR$ and every $h \in \cK$ the resolvent of $A$ is given by
\begin{equation} \label{ereswA}
\bigl(A - z \bigr)^{-1} h = \bigl(A_0 - z \bigr)^{-1} h
  - \frac{\bigl[h,\psi(\co{z})\bigr]_{{\cK}}}{m_+(z)+m_-(-z)} \,
 \psi(z),
 \end{equation}
where
\begin{equation*} 
    \psi(z)= \psi_+(z) + \psi_-(-z).
\end{equation*}
For $h\in\cK_+$ the vector $f=P_{+}\bigl(A - z \bigr)^{-1} h$ is the solution of the $z$-dependent boundary value problem
\begin{equation}\label{eq:z_depend}
    (S_+^*-z)f=h,\quad \Gamma_1^+f+m_-(-z)\Gamma_0^+f=0.
\end{equation}
\end{theorem}
\begin{proof}
Recall that the operator $A$ coincides with the coupling $A_1$ defined by~\eqref{eq:A1}.

Let $z$ be an arbitrary non-real number. By the definition of $A_1$ the equation $(A_1-z)f=h$ with $f \in \dom(S_1)$ and $h \in \cK$ is equivalent to the system
\begin{equation}\label{eq:syst}
(S_+^*-z)f_+ = h_+, \ \ -(S_-^*+z)f_-=h_-,
\end{equation}
where $f_+ = P_+ f$, $f_- = P_- f$, $h_+ = P_+ h$ and $h_- = P_- h$.
It follows from \eqref{eS*ds} and \eqref{eq:syst} that  $f_+ \in
\dom(S_+^*)$ can be expressed as
\begin{equation} \label{efdecom+}
f_+ = \bigl(S_0^+ - z\bigr)^{-1}h_+ + c_+ \psi_+(z) \quad
\text{with some}  \quad c_+ \in\nC.
\end{equation}
Similarly, $f_- \in \dom(S_-^*)$ can be expressed as
\begin{equation} \label{efdecom-}
f_- = -\bigl(S_0^- + z\bigr)^{-1}h_- + c_- \psi_-(-z) \quad
\text{with some}  \quad c_- \in\nC.
\end{equation}
By~\eqref{esubg1} and~\eqref{eG1A0} we have
\begin{align*}
\Gamma_0^+f_+ & = c_+, \\
\Gamma_0^-f_- & = c_-, \\
\Gamma_1^+f_+ & = \Gamma_1^+\bigl(S_0^+ - z\bigr)^{-1}h_+ + c_+ m_+(z) \\
  & = \bigl[ h_+, \psi_+(\co{z}) \bigr]_{\cK} + c_+ m_+(z), \\
\Gamma_1^-f_- & = -\Gamma_1^-\bigl(S_0^- + z\bigr)^{-1}h_- + c_- m_-(-z) \\
 & = \bigl[ h_-,\psi_-(-z^*) \bigr]_{\cK} + c_- m_-(-z).
\end{align*}
As $f_+$ and $f_-$ satisfy the boundary conditions~\eqref{eq:Gamma-} in the definition of $\dom (A_1) = \dom (S_1)$ we get
\begin{align*}
c_+ - c_- & = 0, \\
\Bigl(
 \bigl[ h_+, \psi_+(\co{z}) \bigr]_{\cK} + c_+ m_+(z) \Bigr) +
 \Bigl(
 \bigl[ h_-,\psi_-(-\co{z}) \bigr]_{\cK} + c_- m_-(-z)\Bigr)& = 0.
\end{align*}
For each $z$ such that $m_+(z)+m_-(-z)\ne 0$ the above system has a unique solution for $c_+, c_-$:
\begin{equation*} 
 c_+ = c_-
 =-\frac{%
 \bigl[ h_+, \psi_+(\co{z})\bigr]_{\cK} +
 \bigl[ h_-, \psi_-(-\co{z})\bigr]_{\cK} }%
 {m_+(z)+m_-(-z)} = \frac{-\bigl[ h, \psi(\co{z})\bigr]_\cK}{m_+(z)+m_-(-z)}.
\end{equation*}
Now, the preceding equations, \eqref{efdecom+} and \eqref{efdecom-} imply that whenever $m_+(z)+m_-(-z)\ne 0$, for arbitrary  $h\in {\cK}$, the system \eqref{eq:syst} has a unique solution $f\in \dom(A_1)$  and \eqref{ereswA} holds.  This also proves that the right hand side of~\eqref{eRwA} is a subset of the left hand side.

To prove the converse inclusion in~\eqref{eRwA} assume that $m_+(z)+m_-(-z)=0$  for some $z\in\nC\setminus\nR$.  Then~\eqref{eWeyl} and \eqref{esubg1} imply that
\[
\begin{split}
\Gamma_0^+\psi_+(z)-\Gamma_0^-\psi_-(-z)&=1-1=0\\
\Gamma_1^+\psi_+(z)+\Gamma_1^-\psi_-(-z)&=m_+(z)+m_-(-z)=0.
\end{split}
\]
In view of~\eqref{eq:domwAHC} this means that $\psi_+(z) + \psi_-(-z) \in\dom(A_1)=\dom(S_1)$. Since
\[
(A_1-z)\psi(z) = (S_+^*-z)\psi_+(z) + -(S_-^*+z)\psi_-(-z) = 0,
\]
$z$ is an eigenvalue of $A_1$. That is $z \notin \rho(A_1)\setminus\nR$.

To prove~\eqref{eq:z_depend}, set $h=h_+\in\cK_+$. Then the first equality in ~\eqref{eq:z_depend} is implied by~\eqref{eq:syst}. Since~\eqref{efdecom-} takes the form $f_-=c_-\psi_-(-z)\in\ker(S_-^*+z)$, then
\[
\Gamma_1^-f_-=m_-(-z)\Gamma_0^-f_-=m_-(-z)\Gamma_0^+f_+.
\]
In view of~\eqref{eq:Gamma-} this proves the second equality in~\eqref{eq:z_depend}.
\end{proof}

\begin{corollary}\label{c:m+-}
Let the assumptions of Theorem~{\rm\ref{tKsc}} hold. Then
\begin{equation}\label{eq:m+-}
    m_+(z)+m_-(-z) = 0 \quad \text{for all} \quad z \in \nC \setminus \nR
\end{equation}
if and only if the resolvent set of $A$ is empty. In this case
$\nC\setminus\nR\subset\sigma_p(A)$.
\end{corollary}

As we pointed out in Subsection~\ref{ss:non-Ks}, a nonnegative self-adjoint operator in Kre\u{\i}n space can have an empty resolvent set. Next we show that this cannot happen for a nonnegative partially fundamentally reducible operator.

\begin{corollary} \label{c:non-none}
The spectrum of a nonnegative partially fundamentally reducible operator $A$ in a Kre\u{\i}n space $({\cK}, \kip_{{\cK}})$ is real.
\end{corollary}
\begin{proof}
Assume that $\rho(A)=\emptyset$. Since $A$ is nonnegative in $({\cK}, \kip_{{\cK}})$, then $S_+$ and $S_-$ are nonnegative in the Hilbert spaces $(\cK_+, \kip_{{\cK}})$ and $(\cK_-, -\kip_{{\cK}})$, respectively. Hence each of the Weyl functions $m_+$ and $m_-$ has at most 1 pole on $\nR_-$. By Corollary~\ref{c:m+-} the equality~\eqref{eq:m+-} holds, which implies that $m_+$ is a  rational function with at most three poles on $\nR$. But this contradicts the assumption that $S_+$ is a densely defined operator.
\end{proof}

The claim in the last corollary can also be deduced from \cite[Corollary~2.5]{ABT}.  We notice that for a class of indefinite Sturm-Liouville operators the nonemptiness of the resolvent set was proved in~\cite[Proposition~4.5]{KaMal07}, \cite[Proposition~3.1]{KaKost08} and \cite[Theorem~5.3]{KuTr11} using the known asymptotic behavior of the Titchmarsh-Weyl coefficients.

\section{Regularity of critical points} \label{sect:3}

Throughout this section we use the notation introduced in the paragraph preceding Theorem~\ref{prop:AasCoupl} and in Theorem~\ref{tKsc}. By Corollary~\ref{c:non-none} a nonnegative partially fundamentally reducible operator $A$ in a Kre\u{\i}n space has a nonempty resolvent set.  Therefore $A$ has a spectral function with critical points $0$ and $\infty$.  In this section we study these critical points in terms of the Weyl functions $m_+$ and $m_-$.

\subsection{$D$-properties of pairs of Nevanlinna functions}
The following necessary conditions for the regularity
of the critical points $0$ and $\infty$ of an indefinite
Sturm-Liouville operator in terms of the Titchmarsh-Weyl
coefficients were proved in~\cite[Theorem 3.4]{KaKost08} as an extension of the corresponding result for the real line established in~\cite[Corollary 5.3]{{KaMal07}}.

\begin{theorem}\label{thm:NesReg}
{\rm (\cite{KaMal07}, \cite{KaKost08})}  Let  $A$ be a nonnegative partially fundamentally reducible operator
in a Kre\u{\i}n space $(\cK, \kip_{\cK})$ and let $m_+$ and $m_-$ be Weyl functions introduced in Theorem~{\rm~\ref{tKsc}}.  Then the following two statements hold.
\begin{enumerate}
\renewcommand*\theenumi{\alph{enumi}}
\renewcommand*\labelenumi{{\rm (\theenumi)}}
    \item \label{thm:NesReg-1}
If $\infty\not\in c_s(A)$, then the functions
\begin{equation}\label{eq:ratio2}
 z \mapsto   \frac{\im m_+({z})}{|m_+({z})+m_-(-{z})|}, \qquad z \mapsto
    \frac{\im m_-({z})}{|m_+({z})+m_-(-{z})|}
\end{equation}
are defined on $\nC_+$ and are bounded on
$\Omega_R^\infty = \{{z}\in\nC_+:\,|{z}|>R\}$ for each $R>0$;
    \item \label{thm:NesReg-2}
If $0\not\in c_s(A)$, then the functions in~\eqref{eq:ratio2} are defined on $\nC_+$ and are  bounded on $\Omega_R^0 = \{{z}\in\nC_+:\,|{z}|<R\}$ for each $R>0$.
\end{enumerate}
\end{theorem}

Let us sketch the proof of (\ref{thm:NesReg-1}) from~\cite{KaKost08} using the notation of Theorem~\ref{tKsc}. Since $\infty\not\in c_s(A_1)$
the functions
$z  \mapsto (\im{z})\|(A-{z})^{-1}\|$,
$z \mapsto (\im{z})\|(S_+ -{z})^{-1}\|$ and
$z \mapsto (\im{z})\|(S_- - {z})^{-1}\|$
are bounded on $\Omega_R^\infty$. Then, it follows
from~\eqref{ereswA} and the equalities
\[
\|\psi_+(z)\|_{\cK_+}^2
 = \frac{\im m_+(z)}{\im {z}} \quad \text{and} \quad
\|\psi_-(z)\|_{\cK_-}^2
= \frac{\im m_-(z)}{\im z},
\]
that the ratio in~\eqref{eqDinf} is also bounded on
$\Omega_R^\infty$.

Clearly, this proof works for arbitrary nonnegative coupling $A$ in Theorem~{\rm\ref{prop:AasCoupl}}.

\begin{definition}
We say that a pair $(m_+,m_-)$ of two Nevanlinna functions has
$D_\infty$-property ($D_0$-property), if the function
\begin{equation} \label{eqDinf}
y \mapsto \frac{\im m_+(iy)+\im m_-(iy)}{\bigl|m_+(iy)+m_-(-iy)\bigr|},
\end{equation}
is bounded on the set $(1,\infty)$ ($(0,1)$, respectively). In particular,
a Nevanlinna functions $m_+$ is said to have
$D_\infty$-property ($D_0$-property), if the pair $(m_+,m_+)$ has $D_\infty$-property ($D_0$-property, respectively), that is,  if
the function
\[
y \mapsto \frac{\im m_+(iy)}{\re m_+(iy)},
\]
is bounded on the set $(1,\infty)$ ($(0,1)$, respectively).
\end{definition}

It follows from Theorem~\ref{thm:NesReg} that $D_\infty$-property
($D_0$-property) is necessary for the conditions $\infty\not\in
c_s(A_1)$ ($0\not\in c_s(A_1)$, respectively).

\begin{corollary}\label{Dinf}{\rm (\cite{KaMal07}, \cite{KaKost08})}
Let $A$ be a nonnegative partially fundamentally reducible operator
in a Kre\u{\i}n space $(\cK, \kip_{\cK})$ and let $m_+$ and $m_-$ be Weyl functions introduced in Theorem~{\rm~\ref{tKsc}}.
Then the following statements hold.
\begin{enumerate}
\renewcommand*\theenumi{\roman{enumi}}
\renewcommand*\labelenumi{{\rm (\theenumi)}}
\item
If $\infty\not\in c_s(A)$, then the pair $(m_+, m_-)$ has the $D_\infty$-property.

\item
If $0\not\in c_s(A)$, then the pair  $(m_+, m_-)$ has the $D_0$-property.
\end{enumerate}
\end{corollary}

For a nonnegative densely defined symmetric operator $S$ with defect numbers $(1,1)$ in a Hilbert space we can always choose a boundary triple relative to which the Weyl function of $S$ is a Stieltjes function. Therefore the following proposition is of interest.

\begin{proposition}\label{prop:DInf}
Let $m_+$ be a Stieltjes function which has $D_\infty$-property {\rm (}$D_0$-property{\rm )}. Then for every Stieltjes function $m_-$ the pair $(m_+,m_-)$ has $D_\infty$-property {\rm (}$D_0$-property, respectively{\rm )}.
\end{proposition} \begin{proof}
Since $m_\pm$ is a Stieltjes function, we have $\re m_\pm(iy)\ge 0$ for all $y\in\nR_+$. Therefore,
\begin{equation}\label{eq:Re_miy}
    \re(m_+(iy))  \le \bigl|m_+(iy)+m_-(-iy)\bigr|  \qquad \text{for all} \qquad y\in\nR_+.
\end{equation}
As $m_+$ has $D_\infty$-property there exists $C > 0$ such that
\[
\frac{\im m_+(iy)}{\re m_+(iy)} \le C \qquad \text{for all} \qquad y > 1.
\]
By \eqref{eq:Re_miy} we have
\[
 \frac{\im m_+(iy)}{\bigl|m_+(iy)+m_-(-iy)\bigr|}
  \le \frac{\im m_+(iy)}{\re m_+(iy)} \le C \qquad \text{for all} \qquad y > 1.
\]
Further, the inequality
\begin{equation*}\label{eq:Im_miy}
  \bigl|\im \bigl(m_+(iy)- m_-(iy)\bigr) \bigr| \le   \bigl|m_+(iy)+m_-(-iy)\bigr| \quad \text{for all} \quad y\in\nR_+
\end{equation*}
yields that for all $y > 1$ we have
\[
\begin{split}
 \frac{\im
m_-(iy)}{\bigl|m_+(iy)+m_-(-iy)\bigr|}
 &
 \le \frac{\bigl|\im \bigl( m_-(iy)- m_+(iy) \bigr) \bigr|}%
 {\bigl|m_+(iy)+m_-(-iy)\bigr|}
+ \frac{\im m_+(iy)}{\bigl|m_+(iy)+m_-(-iy)\bigr|}  \\
& \le C+1 .
\end{split}
\]
The proof of the $D_0$-property is similar.
\end{proof}

Asymptotic behavior of Weyl functions has been investigated for many kinds of special symmetric operators. Some specific examples appear in Section~\ref{SecExe} below. The next proposition deduces $D$-properties from the asymptotic behavior of Nevanlinna functions introduced in Subsection~\ref{sSecAsy}.

\begin{proposition}\label{prop:DpsAsym}
\begin{enumerate}
\renewcommand*\theenumi{\alph{enumi}}
\renewcommand*\labelenumi{{\rm (\theenumi)}}
\item  \label{i:DpsAsym-1}
If Nevanlinna functions $m_+$ and $m_-$ belong to $\cA_\infty$, then  the corresponding pair $(m_+, m_-)$ has $D_\infty$-property.

\item \label{i:DpsAsym-2}
If Nevanlinna functions $m_+$ and $m_-$ belong to $\cA_0$, then the corresponding pair  $(m_+, m_-)$ has $D_0$-property.

\end{enumerate}

\end{proposition}
\begin{proof}
Let $m_\pm  \in \cA_\infty$. Proposition~\ref{prop:0-inf} applies. Let  $C_0^\pm \in \nR\setminus\{0\}$, $\alpha_0^\pm \in (-1,1)$, and, as appropriate, $C_1^\pm \in \nR\setminus\{0\}$ and $\alpha_1^\pm \in (-1,0)$ be the corresponding coefficients appearing in \eqref{eq:mAsymInf-noC} and \eqref{eq:mAsymInf-C}.  We first notice that as a consequence of the asymptotics in \eqref{eq:mAsymInf-noC} or \eqref{eq:mAsymInf-C} we have
\begin{align} \label{eq:Limfrm}
 \lim_{y\uparrow +\infty} \frac{m_\pm(iy)}{y^{\alpha_0^\pm}} & = C_0^\pm i^{\alpha_0^\pm}, \\
\intertext{
which, in turn, implies
} \label{eq:LimfrIm}
\lim_{y\uparrow+\infty} \frac{\im m_\pm(iy)}{y^{\alpha_0^\pm}}
 &= C_0^\pm \sin\bigl(\pi\alpha_0^\pm/2\bigr).
\end{align}

Let
\[
\alpha_0 = \begin{cases}
\max\bigl\{  \alpha_0^+ , \alpha_0^- \bigr\} & \ \text{if} \quad  \max\{\alpha_0^+, \alpha_0^-\} \geq 0, \\[1pt]
\min\bigl\{ \alpha_0^+ ,  \alpha_0^- \bigr\} & \  \text{if} \quad \max\{\alpha_0^+, \alpha_0^-\} < 0.
\end{cases}
\]
and, further, set $\epsilon_\pm = 1 - \sgn(|\alpha_0 - \alpha_0^\pm|)$.  By definition $\epsilon_+, \epsilon_- \in \{0,1\}$, and at least one of $\epsilon_+$, $\epsilon_-$ equals $1$.

Next we will calculate the limit
\begin{equation} \label{eq:DinfLimit}
\lim_{y\uparrow+\infty}
\frac{\im m_+(iy)+\im m_-(iy)}{\bigl|m_+(iy)+m_-(-iy)\bigr|}
\end{equation}
and doing so we prove (\ref{i:DpsAsym-1}).  Using \eqref{eq:Limfrm} and \eqref{eq:LimfrIm} this limit can be calculated as follows:
\begin{equation} \label{eq:DinfLimitL}
\lim_{y\uparrow+\infty} \frac{\frac{\im m_+(iy)}{y^{\alpha_0}}
 + \frac{\im m_-(y)}{y^{\alpha_0}}}{\left|\frac{m_+(iy)}{y^{\alpha_0}}
 + \frac{m_-(-iy)}{y^{\alpha_0}}\right|}
 =
 \frac{
 \epsilon_+ C_0^+ \sin\bigl(\pi\alpha_0^+/2\bigr)
 + \epsilon_- C_0^-\sin\bigl(\pi\alpha_0^-/2\bigr)}%
 {\bigl| \epsilon_+ C_0^+i^{\alpha_0^+} + \epsilon_- C_0^-(-i)^{\alpha_0^-}\bigr|}.
\end{equation}
If exactly one of $\epsilon_+, \epsilon_-$ equals $1$, then the denominator in the last fraction is clearly positive. Otherwise, that is, if $\epsilon_+ = \epsilon_- =1$, we have $\alpha_0^+ = \alpha_0^-$ and the denominator is again clearly positive provided that $\alpha_0^+ = \alpha_0^- \neq 0$. If $\alpha_0^+ = \alpha_0^- = 0$, the denominator equals $\bigl|C_0^+ + C_0^-\bigl|$.

Thus, the limit in \eqref{eq:DinfLimit} is calculated to be the right-hand side of \eqref{eq:DinfLimitL} with the exception of the case when $\alpha_0^+ = \alpha_0^- = 0$ and $C_0^+ + C_0^- = 0$.

To calculate the limit in \eqref{eq:DinfLimit} when $\alpha_0^+ = \alpha_0^- = 0$ and $C_0^+ + C_0^- = 0$ we use the second term of the asymptotics in \eqref{eq:mAsymInf-C}.  First we deduce
\begin{align*}
\lim_{y\uparrow+\infty} \frac{m_\pm(iy)-C^\pm}{y^{\alpha_1^\pm}} & = C_1^\pm i^{\alpha_1^\pm}, \\
\intertext{which clearly yields}
\lim_{y\uparrow+\infty} \frac{\im m_\pm(y)}{y^{\alpha_1^\pm}}
& = C_1^\pm \sin\bigl(\pi\alpha_1^\pm/2\bigr).
\end{align*}
As before, we introduce $\alpha_1 = \min\{\alpha_1^+, \alpha_1^-\}$ and $\epsilon_1^\pm = 1 - \sgn(|\alpha_1 - \alpha_1^\pm|) \in \{0,1\}$.

Since we assume that $C_0^+ + C_0^- = 0$, the limit in \eqref{eq:DinfLimit} equals
\begin{equation*}
\lim_{y\uparrow+\infty}\! \frac{\frac{\im m_+(iy)}{y^{\alpha_1}} + \frac{\im m_-(y)}{y^{\alpha_1}}}{\left|\frac{m_+(iy)-C_0^+}{y^{\alpha_1}}
 + \frac{m_-(-iy)-C_0^-}{y^{\alpha_1}}\right|}
 =
 \frac{
 \epsilon_1^+ C_1^+ \sin\bigl(\pi\alpha_1^+/2\bigr) + \epsilon_1^- C_1^-\sin\bigl(\pi\alpha_1^-/2\bigr)}%
 {\bigl| \epsilon_1^+ C_1^+i^{\alpha_1^+} + \epsilon_1^- C_1^-(-i)^{\alpha_1^-}\bigr|}.
\end{equation*}
Now recall that at least one of $\epsilon_1^-$, $\epsilon_1^+$ equals $1$,  $\alpha_1^-,\alpha_1^+ \in (-1,0)$ and $C_1^-,C_1^+ \in \nR\!\setminus\!\{0\}$, to deduce that the denominator of the last fraction is always positive.

Thus, in each case we calculated the limit in \eqref{eq:DinfLimit}. This proves that the pair $(m_+,m_-)$ has the $D_\infty$-property.

The proof of (\ref{i:DpsAsym-2}) is similar.
\end{proof}


\subsection{Regularity of the critical points} \label{ss:RegCPs}
Let  $A$ be a nonnegative partially fundamentally reducible operator
in a Kre\u{\i}n space $(\cK, \kip_{\cK})$. In this subsection we provide sufficient conditions in terms of $m_+$ and $m_-$ for the points $\infty$ or $0$ not to be singular critical points of the operator $A$.

In the next lemma we use Theorem~\ref{Veselic} to obtain criteria for the regularity of the critical points $0$ and $\infty$ for the operator $A_1$ defined by~\eqref{eq:A1} and~\eqref{eq:domwAHC}.  To formulate these results we need the following family of operators. For arbitrary $\varepsilon, \eta$ such that $0< \varepsilon < \eta$ we define the operator
$
T_{\pm\pm}(\epsilon,\eta):\cK_\pm\to\cK_\pm
$
by
\begin{equation}\label{CritTpm}
\bigl( T_{\pm\pm}(\epsilon,\eta) \bigr)f_{\pm} := \int\limits_{\epsilon\le|y|\le\eta}
 \frac{\bigl\langle f_\pm ,\psi_\pm(\mp iy) \bigr\rangle_\cK \,
 \psi_\pm(\pm iy)}{m_+(iy)+m_-(-iy)} \, dy, \quad f_\pm \in \cK_\pm,
\end{equation}
and the operator $T_{\pm\mp}(\epsilon,\eta):\cK_\mp\to\cK_\pm$
by
\begin{equation}\label{CritTpm-+}
\bigl( T_{\pm\mp}(\epsilon,\eta) \bigr)f_\mp  := \int\limits_{\epsilon\le|y|\le\eta}
 \frac{\bigl\langle f_\mp ,\psi_\mp(\pm iy)\bigr \rangle_\cK \,
 \psi_\pm(\pm iy)}{m_+(iy)+m_-(-iy)} \, dy, \quad f_\mp \in \cK_\mp.
\end{equation}

\begin{lemma}\label{CritDinf}
Let $A$ be a nonnegative partially fundamentally reducible operator in the Kre\u{\i}n space $(\cK, \kip_{\cK})$. Then the following statements hold.
\begin{enumerate}
\renewcommand*\theenumi{\roman{enumi}}
\renewcommand*\labelenumi{{\rm (\theenumi)}}
  \item \label{iCritDinf-1}
  $\infty\not\in c_s(A)$ if and only if the operators $T_{++}(1,\eta)$, $T_{--}(1,\eta)$, $T_{+-}(1,\eta)$ and $T_{-+}(1,\eta)$
 are uniformly bounded for $\eta\in(1,\infty)$;
  \item \label{iCritDinf-2}
   $0\not\in c_s(A)$  if and only if $\ker(A)=\ker(A^2)$ and the operators $T_{++}(\varepsilon,1)$, $T_{--}(\varepsilon,1)$, $T_{+-}(\varepsilon,1)$ and $T_{-+}(\varepsilon,1)$
 are uniformly bounded for $\varepsilon\in(0,1)$.
\end{enumerate}
\end{lemma}
\begin{proof}
Notice first that by Corollary~\ref{c:non-none} the operator $A$ has a nonempty resolvent set. Set $A_0 = S_{0}^+\oplus (-S_{0}^-)$, see Theorem~\ref{tHsc}(\ref{ietHsc0}) and \eqref{eq:A0}.  Since $A_0$ is self-adjoint in the Hilbert space $(\cK, \ahip_{\cK})$,  the family of operators
\[
 \int\limits_{1\le|y|\le\eta}  (A_0 - iy)^{-1}dy, \qquad \eta >1,
 \]
is uniformly bounded. Recall from Theorem~\ref{tKsc} that $\psi(z) = \psi_+(z) + \psi_-(-z)$, $z \in \nCR$, and let $f \in \cK$ be arbitrary. By \eqref{ereswA} we have
\begin{equation*} 
\begin{split}
&
\int\limits_{1\le|y|\le\eta}\left(\bigl(A - iy \bigr)^{-1}f  -\bigl(A_0 - iy \bigr)^{-1}f\right)dy \\
 & \rule{4cm}{0pt} = \int\limits_{1\le|y|\le\eta}
  \frac{-\bigl[ f, \psi(-{iy})\bigr]_{\cK} \, \psi(iy)}{m_+(iy)+m_-(-iy)} dy \\
 & \rule{4cm}{0pt} =\begin{pmatrix} T_{++}(1,\eta) & T_{+-}(1,\eta)
 \\ T_{-+}(1,\eta) & T_{--}(1,\eta) \end{pmatrix}
 \begin{pmatrix} -P_+ f \\ P_- f \end{pmatrix}.
 \end{split}
\end{equation*}
Therefore, the uniform boundedness of the family of operators
\[
  \int\limits_{1\le|y|\le\eta}
  ({A_1} - {iy})^{-1}d{y}, \qquad \eta > 1,
\]
is equivalent to the uniform boundedness of the families of
operators
\[
T_{++}(1,\eta), \ \  T_{--}(1,\eta), \ \ T_{+-}(1,\eta),
 \ \ T_{-+}(1,\eta), \quad \eta\in(1,\infty).
\]
Now (\ref{iCritDinf-1}) follows from Theorem~\ref{Veselic}(\ref{iVeselic-1}).

The statement (\ref{iCritDinf-2}) is implied by the  formula~\eqref{ereswA}
and Theorem~\ref{Veselic}(\ref{iVeselic-2}).
\end{proof}

\begin{theorem}\label{Main}
Let $A$ be a nonnegative partially fundamentally reducible operator in a Kre\u{\i}n space $({\cK}, \kip_{{\cK}})$ and assume that Weyl functions $m_+$ and $m_-$ introduced in Theorem~{\rm~\ref{tKsc}} have $B_\infty$-property (see Definition~\ref{defBinf}). Then $\infty\not\in c_s(A)$ if and only if the pair $(m_+,m_-)$ has $D_\infty$-property.
\end{theorem}
\begin{proof}
For the necessity see Corollary~\ref{Dinf}.

To prove the sufficiency consider first $f_+$, $g_+\in\cK_+$ and find an upper bound for
\[
\Bigl|\bigl\la T_{++}(1,\eta)f_+,g_+\bigr\ra_{\cK} \Bigr|.
\]
By the definition of the generalized Fourier transform $F_+=F_{m_+}$ we have
\begin{equation*} \label{eqesT++}
\bigl\la T_{++}(1,\eta)f_+,g_+\bigr\ra_{\cK}
 =  \int\limits_{1\le|y|\le\eta}
 \frac{(F_{+}f_+)(iy) \bigl((F_{+}g_+)(-iy)\bigr)^*}{m_+(iy)+m_-(-iy)}dy.
\end{equation*}
Since the pair $(m_+,m_-)$ has $D_\infty$-property, there exists $C_1 > 0$ such that
\begin{equation}\label{eq:D_infty+}
\bigl| m_+(iy)+m_-(-iy) \bigr| \geq C_1^{-1} \im m_+(iy) \quad \text{for all} \quad  y > 1.
\end{equation}
Set $w_\pm(y)=\bigl(\im m_\pm(iy)\bigr)^{-1}$, $y\in\nR_+$.
Since $m_+$ has $B_\infty$-property, then in view of Definition~\ref{defBinf} and Proposition~\ref{prop:4.5} the mappings
\begin{align*}
G_{m_+,\infty}^+F_{+} &: f_+\mapsto (F_{+}f_+)(iy), \quad f_+\in\cK_+, \\
G_{m_+,\infty}^-F_{+} &: f_+\mapsto (F_{+}f_+)(-iy), \quad f_+\in\cK_+,
\end{align*}
are bounded from $\cK_+$ to $L^2_{w_{+}}(1,\infty)$ and hence there exists $C_2>0$, such that
\begin{equation}\label{eq:B_infty+}
    \bigl\|G_{m_+,\infty}^+F_{+}f_+ \bigr\|_{L^2_{w_{+}}\!(1,\infty)}
     \le C_2 \bigl\|f_+\bigr\|_{\cK}
 \end{equation}
 and
\begin{equation}\label{eq:B_infty-}
 \bigl\|G_{m_+,\infty}^-F_{+}f_+\bigr\|_{L^2_{w_{+}}\!(1,\infty)}
  \le C_2 \bigl\|f_+\bigr\|_{\cK}
 \end{equation}
for all $f_+\in\cK_+$.
Using \eqref{eq:D_infty+}, \eqref{eq:B_infty+}, \eqref{eq:B_infty-} and~\eqref{eq:F_trans} we obtain
\begin{align*}
\Bigl|\bigl\la T_{++}(1,\eta)f_+,g_+\bigr\ra_{\cK} \Bigr|
 & \leq C_1 \int\limits_{1\le|y|\le\eta} \
  \bigl|(F_{+}f_+)(iy) \bigr| \, \bigl| (F_{+}g_+)(-iy) \bigr| w_{+}(y) dy \\
 & \leq C_1
 \left(\int\limits_{1\le|y|\le\eta} \bigl|(F_{+}f_+)(iy)\bigr|^2 w_{+}(y) dy\right)^{\!1/2} \\
  & \rule{26mm}{0pt} \times \left(\int\limits_{1\le|y|\le\eta} \bigl|(F_{+}g_+)(iy)\bigr|^2 w_{+}(y) dy\right)^{\!1/2}  \\
&\le 2C_1C_2^2 \bigl\|f_+\bigr\|_{\cK}\bigl\|g_+\bigr\|_{\cK},
\end{align*}
for all $f_+,g_+ \in \cK_+$. This proves that the family $T_{++}(1,\eta)$ is uniformly bounded for $\eta\in(1,\infty)$.

The proof for the families $T_{--}(1,\eta)$ and $T_{+-}(1,\eta)$ is similar.

Next, we consider $f_+\in\cK_+$ and $g_-\in\cK_-$ and give an upper bound for
\[
\Bigl|\bigl\la T_{-+}(\epsilon,\eta)f_+,g_-\bigr\ra_{\cK} \Bigr|.
\]
Since the function in \eqref{eqDinf} is bounded, then the function
\begin{equation*}
y \mapsto \frac{\sqrt{\im m_+(iy)} \sqrt{\im m_-(iy)}}{\bigl|m_+(iy)+m_-(-iy)\bigr|} \qquad \text{for all} \qquad y > 1,
\end{equation*}
is bounded as well. Therefore, there exists $C_3 >0$ such that
\begin{equation*}
 \bigl|m_+(iy)+m_-(-iy)\bigr| \geq C_2 \sqrt{\im m_+(iy)} \sqrt{\im m_-(iy)} \quad \text{for all} \quad y > 1,
\end{equation*}
Consequently, with $F_- = F_{m_-}$, we have
\begin{align*}
\Bigl|\bigl\la T_{-+}(1,\eta)f_+,g_- \bigr\ra_{\cK} \Bigr|
 & \le \int\limits_{1\le|y|\le\eta} \
 \biggl| \frac{(F_{+}f_+)(iy) \bigl((F_{-}g_-)(-iy)\bigr)^*}%
 {m_+(iy)+m_-(-iy)} \biggr| \ dy  \\
 & \rule{-5mm}{0pt}
 \leq C_3 \!\!\!\!\! \int\limits_{1\le|y|\le\eta} \!\!\!\!
\bigl|(F_{+}f_+)(iy) \bigr| \bigl|(F_{-}g_-)(-iy)\bigr| \sqrt{w_{+}(y)} \sqrt{w_{-}(y)} dy \\
 & \rule{-5mm}{0pt}
 \leq 2C_2C_3 \ \bigl\| G_{m_+} F_{+}f_+ \bigr\|_{L_{w_{+}}^2\!(1,\infty)} \ \bigl\| G_{m_-} F_{-}f_- \bigr\|_{L_{w_{-}}^2\!(1,\infty)} \\
 & \rule{-5mm}{0pt}
 \le 2C_2C_3  \ \bigl\|f_+ \bigr\|_{\cK} \ \bigl\|g_-\bigr\|_{\cK}
\end{align*}
for all $f_+ \in \cK_+$ and all $g_- \in \cK_-$. This proves that the family $T_{-+}(1,\eta)$ is uniformly bounded for $\eta\in(1,\infty)$.

The uniform boundedness of $T_{+-}(1,\eta)$ is proved similarly. Now the statement is implied by Lemma~\ref{CritDinf}.
\end{proof}

\begin{theorem}\label{Main0}
Let $A$ be a nonnegative partially fundamentally reducible operator in a Kre\u{\i}n space $({\cK}, \kip_{{\cK}})$ and assume that Weyl functions $m_+$ and $m_-$ introduced in Theorem~{\rm~\ref{tKsc}} have  $B_0$-property.  Then $0\not\in c_s(A)$ if and only if $\ker(A)=\ker(A^2)$ and the pair $(m_+,m_-)$ has $D_0$-property.
\end{theorem}
\begin{proof}
We prove sufficiency. For necessity see Corollary~\ref{Dinf}.

Since $m_+$ and $m_-$ have $D_0$-property, there exists $C_4 > 0$ such that
\begin{equation}\label{eq:D_infty+A}
\bigl| m_+(iy)+m_-(-iy) \bigr| \geq C_4 \im m_+(iy) \qquad \text{for all} \qquad y\in(0,1).
\end{equation}
Since $m_+$ has $B_0$-property, the mappings
\[
G_{m_+,0}^+F_{+}^+:f_+\mapsto (F_{+}f_+)(iy)\quad (f_+\in\cK_+)
\]
\[
G_{m_+,0}^-F_{+}^-:f_+\mapsto (F_{+}f_+)(-iy)\quad (f_+\in\cK_+)
\]
are bounded from $\cK_+$ to $L_{w_{+}}^2\!(0,1)$ and hence there exists $C_5>0$, such that
\begin{align}\label{eq:B_0+}
    \bigl\|G_{m_+,0}^+F_{+}f_+\bigr\|_{L_{w_{+}}^2\!(0,1)}
     & \le C_5 \bigl\|f_+ \bigr\|_{\cK}, \\
\label{eq:B_0-}
    \bigl\|G_{m_+,0}^-F_{+}f_+\bigr\|_{L_{w_{+}}^2\!(0,1)}
     & \le C_5 \bigl\|f_+\bigr\|_{\cK}
\end{align}
for all $f_+\in\cK_+$.
Using \eqref{eq:D_infty+A}, \eqref{eq:B_0+}, \eqref{eq:B_0-} and~\eqref{eq:F_trans} we obtain
\begin{align*}
\Bigl|\bigl\la T_{++}(\epsilon,1)f_+,g_+\bigr\ra_{\cK} \Bigr|
  & \leq C_4 \int\limits_{\epsilon\le|y|\le 1} \bigl|(F_{+}f_+)(iy)
\bigr| \bigl|(F_{+}g_+)(-iy)\bigr| w_{+}(y) dy \\
 &\le C_4\, \bigl\| G_{m_+,0}^+ F_{+}f_+ \bigr\|_{L^2_{w_{+}}\!(0,1)} \ \bigl\| G_{m_+,0}^- F_{+}g_+ \bigr\|_{L^2_{w_{+}}\!(0,1)}\\
 &\le 2C_4C_5^2 \,  \bigl\|f_+ \bigr\|_{\cK} \,  \bigl\|g_+\bigr\|_{\cK}
\end{align*}
for all $f_+,g_+\in\cK_+$. This proves that the family $T_{++}(\epsilon,1)$ is uniformly bounded for $\epsilon\in(0,1)$.

The proof for the family $T_{--}(\epsilon,1)$ is similar. Now the statement is implied by Lemma~\ref{CritDinf}.
\end{proof}

Proposition~\ref{p:fd-iff-rcp} together with Theorems \ref{Main} and \ref{Main0} yield the following statement.

\begin{corollary}\label{cor:Main0}
Let $A$ be a nonnegative partially fundamentally reducible operator in a Kre\u{\i}n space $({\cK}, \kip_{{\cK}})$ and let $m_+$ and $m_-$ be Weyl functions introduced in Theorem~{\rm~\ref{tKsc}}. If each of the functions $m_+$ and $m_-$ has both $B_\infty$-property and $B_0$-property, then $A$ is fundamentally reducible in the Kre\u{\i}n space $(\cK, \kip_{\cK})$ if and only if the pair $(m_+,m_-)$ has $D_\infty$-property  and $D_0$-property and  $\ker(A)=\ker(A^2)$.
\end{corollary}

\begin{remark}\label{rem:Reflex}
The coupling $A_1$ in~\eqref{eq:A1} will be called {\it
reflexive} if  $m_+(z)=m_-(z)$.
In the reflexive case analogs of Theorems~\ref{Main} and \ref{Main0} for indefinite Sturm-Liouville operator  were proven in~\cite{Kost13}. It was shown there that the $B$-property for $m_+$ is automatically satisfied if the corresponding $D$-property holds.
\end{remark}
This remark leads to the following sufficient conditions of regularity.

\begin{corollary}\label{cor:Main1}
Let $A$ be a nonnegative partially fundamentally reducible operator in a Kre\u{\i}n space $({\cK}, \kip_{{\cK}})$ and assume that Weyl functions $m_+$ and $m_-$ introduced in Theorem~{\rm~\ref{tKsc}} are Stieltjes functions. Then the following statements hold.
\begin{enumerate}
\renewcommand*\theenumi{\roman{enumi}}
\renewcommand*\labelenumi{{\rm (\theenumi)}}
  \item \label{i:cor:Main1-1}
 If each function $m_+$ and $m_-$ has $D_\infty$-property, then $\infty\not\in c_s(A)$.
  \item \label{i:cor:Main1-2}
 If each function $m_+$ and $m_-$ has $D_0$-property and  $\ker(A) = \ker(A^2)$,  then $0\not\in c_s(A)$.
   \item \label{i:cor:Main1-3}
 If each of the functions $m_+$ and $m_-$ has both $D_\infty$-property  and $D_0$-property, and  $\ker(A) = \ker(A^2)$, then $A$ is fundamentally reducible in the Kre\u{\i}n space $({\cK}, \kip_{{\cK}})$.
\end{enumerate}
\end{corollary}
\begin{proof}
(\ref{i:cor:Main1-1}) Since each function $m_+$ and $m_-$ has $D_\infty$-property then   also each $m_+$ and $m_-$ has $B_\infty$-property (see Remark~\ref{rem:Reflex}).  By Proposition~\ref{prop:DInf} the pair $(m_+,m_-)$ has $D_\infty$-property and hence $\infty\not\in c_s(A)$ by Theorem~\ref{Main}.

(\ref{i:cor:Main1-2}) Similarly, if each $m_+$ and $m_-$ has $D_0$-property then also each  $m_+$ and $m_-$ has $B_0$-property (see Remark~\ref{rem:Reflex}). By Proposition~\ref{prop:DInf} the pair $(m_+,m_-)$ has $D_0$-property and thus $0\not\in c_s(A)$ by Theorem~\ref{Main0}.
\end{proof}

In the next theorem we use the notation introduced at the beginnings of Subsections~\ref{SubS:CoupleHS} and \ref{SubS:CoupleKS}.

\begin{theorem} \label{t:AllPos}
Let $S_\pm$ be a closed symmetric densely defined nonnegative operator with defect numbers $(1,1)$ in the Hilbert space $\bigl(\cK_\pm, \ahip_{\cK_\pm}\bigr)$.  Let $\bigl(\nC,\Gamma_0^\pm,\Gamma_1^\pm\bigr)$ be a boundary triple for $S_\pm^*$ and let $m_\pm$ be the corresponding Weyl function.  Let $A$ be a nonnegative self-adjoint extension of $S_+\oplus(-S_-)$ in the Kre\u{\i}n space $\bigl(\cK,\kip_\cK\bigr)$. Then the following statements hold.
\begin{enumerate}
\renewcommand*\theenumi{\alph{enumi}}
\renewcommand*\labelenumi{{\rm (\theenumi)}}
  \item \label{i:AllPos-1}
If $m_+, m_- \in \cA_\infty$, then $\infty\not\in c_s(A)$.

  \item \label{i:AllPos-2}
If $m_+, m_- \in \cA_0$ and $\ker(A)=\ker(A^2)$, then $0\not\in c_s(A)$.

 \item \label{i:AllPos-3}
If $m_+, m_- \in \cA_0 \cap \cA_\infty$ and $\ker(A)=\ker(A^2)$, then $A$ is fundamentally reducible.
\end{enumerate}
\end{theorem}
\begin{proof}
Let $A$ be an arbitrary nonnegative self-adjoint extension of  $S_+\oplus(-S_-)$ in the Kre\u{\i}n space $\bigl(\cK,\kip_\cK\bigr)$. Then $A$ is a nonnegative fundamentally reducible operator in $\bigl(\cK,\kip_\cK\bigr)$, so $\rho(A) \neq \emptyset$ by Corollary~\ref{c:non-none}.

The operator  $S = J\,A$ is a self-adjoint extension of $S_+\oplus S_-$ in the Hilbert space $\bigl(\cK,\ahip_\cK\bigr)$.  Assume first that $S$ is a separated extension of $S_+\oplus S_-$. By the definition of a separated extension, there exists a self-adjoint extension $T_+$ of $S_+$ in $\bigl(\cK_+,\kip_{\cK_+}\bigr)$ and a self-adjoint extension $T_-$ of $S_-$ in $\bigl(\cK_-,\kip_{\cK_-}\bigr)$ such that $S = T_+\oplus T_-$. Then, $S$ commutes with the fundamental symmetry $J$ introduced in \eqref{eq:defJ}.  Therefore, $A = JS$ is fundamentally reducible, so all three claims are trivial in this case.

Next assume that $S$ is a non-separated extension of $S_+\oplus S_-$. By Theorem~\ref{Thm:5.5} there exist a boundary triple $\bigl(\nC,\wh{\Gamma}_0^+, \wh{\Gamma}_1^+\bigr)$ for $S_+^*$ and
a boundary triple $\bigl(\nC,\wh{\Gamma}_0^-, \wh{\Gamma}_1^-\bigr)$ for $S_-^*$, such that $S$ is a coupling of $S_+$ and $S_-$ relative to the triples $\bigl(\nC,\wh{\Gamma}_0^+, \wh{\Gamma}_1^+\bigr)$ and $\bigl(\nC,\wh{\Gamma}_0^-, \wh{\Gamma}_1^-\bigr)$. By Remark~\ref{rem:transp} there exists a M\"{o}bius transformation $\mu_\pm$ such that the Weyl function $\wh{m}_\pm$ of $S_\pm$ relative to $\bigl(\nC,\wh{\Gamma}_0^\pm, \wh{\Gamma}_1^\pm\bigr)$ is given by
\begin{equation*} 
\wh{m}_\pm = \mu_\pm \circ m_\pm.
\end{equation*}

We proceed with a proof of (\ref{i:AllPos-1}). Assume that $m_+, m_- \in \cA_\infty$. Proposition~\ref{prop:AInf-AZer}(\ref{i:AInf-AZer-1}) implies that $\wh{m}_+, \wh{m}_- \in \cA_\infty$ and, in turn, Corollary~\ref{co:Binfprop} yields that $\wh{m}_+$ and $\wh{m}_-$ have $B_\infty$-property. Since by Proposition~\ref{prop:DpsAsym}(\ref{i:DpsAsym-1}) the pair $(\wh{m}_-,\wh{m}_+)$ has $D_\infty$-property, the claim follows from Theorem~\ref{Main}.

Next assume that $m_+, m_- \in \cA_0$. As before,  Proposition~\ref{prop:AInf-AZer}(\ref{i:AInf-AZer-1}) implies that $\wh{m}_+, \wh{m}_- \in \cA_\infty$. Now Corollary~\ref{th:B0prop} yields that $\wh{m}_+$ and $\wh{m}_-$ have $B_0$-property and by Proposition~\ref{prop:DpsAsym}(\ref{i:DpsAsym-2}) the pair $(\wh{m}_-,\wh{m}_+)$ has $D_0$-property. The claim now follows from Theorem~\ref{Main0}.

Statement (\ref{i:AllPos-3}) is a consequence of  Corollary~\ref{cor:Main0}.
\end{proof}


\section{Examples} \label{SecExe}

\begin{example} \label{ex2ndq}
Consider the singular differential expression
\begin{equation} \label{eq2ndq}
\ell(f)(t):= -\frac{\sgn t}{w(t)} \Bigl( \bigl(p(t)\,f'(t)\bigr)' + q(t) f(t)\Bigr) \quad \text{for a.a.} \quad t \in \nR,
 \end{equation}
where the coefficients $p$, $q$ and $w$ are real functions on $\nR$ satisfying the conditions
\begin{enumerate}
    \item[(C1)]
    $1/p, q, w \in L^1_{\rm loc}(\nR)$ and $p , w > 0$ a.e. on $\nR$,
    \item[(C2)]
    the expression $\ell$ is in the limit point case at $-\infty$ and at $+\infty$.
\end{enumerate}

It is natural to consider the expression in \eqref{eq2ndq} in the Kre\u{\i}n  space $\bigl(L_{w}^2(\nR), \kip\bigr)$. Here $\bigl(L_{w}^2(\nR),\ahip\bigr)$ is the standard weighted $L^2$-space with the positive definite inner product $\ahip$ and the indefinite inner product is given by $[f,g] = \la Jf,g \ra, f, g \in L_{w}^2(\nR)$, where
\begin{equation*} 
(Jf)(t)=(\sgn t) f(t), \qquad f \in L_{w}^2(\nR),
\end{equation*}
is a fundamental symmetry  on $\bigl(L_{w}^2(\nR),\kip\bigr)$. Set
\begin{equation} \label{eq:L2pm}
 \cK_\pm = \bigl\{ f \in L_{w}^2(\nR) : f = 0 \ \ \text{a.e. on} \ \ \nR_\mp \bigr\}.
\end{equation}
Then $L_{w}^2(\nR) = \cK_+ [\dot{+}]\cK_-$ is the fundamental decomposition corresponding to $J$.

Let $A$ be the operator associated with the expression in \eqref{eq2ndq} in the Hilbert space $L_{w}^2(\nR)$; that is the operator defined by $Af = \ell(f)$ for all
\[
f \in \dom(A) =
\bigl\{ f \in L_{w}^2(\nR) :  f, pf' \in AC_{\rm loc}(\nR), \ \ell(f) \in L_{w}^2(\nR) \bigr\}.
\]
The differential operator $A$ is partially fundamentally reducible in the Kre\u{\i}n space $\bigl(L_{w}^2(\nR),\kip\bigr)$.  To see this, consider the range restriction $S_\pm$ of $\pm A$ to $\cK_\pm$ which are defined on
\[
\bigl\{ f \in \dom(A) : Af \in \cK_\pm \bigr\} = \cK_\pm \cap \dom(A).
\]
Then $S_\pm$ is the minimal operator associated in $L_{w_\pm}^2\!(\nR_\pm)$ with the restriction of $\pm\ell$ to $\nR_\pm$; here $w_\pm$ denotes the restriction of $w$ to $\nR_\pm$.  In fact we have,
\begin{align} \nonumber 
\dom(S_\pm^*) & =
\bigl\{ f \in L_{w_\pm}^2\!(\nR_\pm): f, pf' \in AC_{\rm loc}[0,\pm\infty), \ \ell(f) \in L_{w_\pm}^2\!(\nR_\pm)  \bigr\}, \\
\nonumber
 \dom(S_\pm) & = \bigl\{ f \in \dom(S_\pm^*) : f(0) = f'(0) = 0\bigr\}, \\
\intertext{and}
\label{edefS}
S_\pm f &:= \pm \ell(f), \quad f \in \dom(S_\pm).
 \end{align}
Since we assume that $\ell$ is in the limit point case at
$\pm\infty$, the operator $S_\pm$ is a densely defined symmetric operator with defect numbers $(1, 1)$ in the Hilbert space $L_{w_\pm}^2\!(\nR_\pm)$.

Let $z \in \nCR$ and denote by $\vartheta(\cdot,z)$ and
$\varphi(\cdot,z)$ the unique solutions of the equation
\begin{equation*} 
   -(p\,f')' + q f =z w f
\end{equation*}
satisfying the boundary conditions
 \[
\vartheta(0,z) = 1, \ (p\,\vartheta')(0,z) = 0, \quad \text{and}
\quad \varphi(0,z) = 0, \ (p\,\varphi')(0,z) = 1,
 \]
respectively. Since we assume that $\ell$ is in the limit point case at
$\pm\infty$, for each $z\in\nCR$ there is a unique solution
\begin{equation} \label{eTWeyl}
\psi_\pm(t,z) = \varphi(t,z) \mp m_\pm(z) \vartheta(t,z), \qquad t \in \nR_\pm,
\end{equation}
of the restriction of $\pm\ell(f) = z f$ to $\nR_\pm$ which belongs to $L_{w_\pm}^2\!(\nR_\pm)$.  Relation
\eqref{eTWeyl} defines the function $m_\pm: \nCR\to \nC$ uniquely. The function $m_\pm$ is called {\em Titchmarsh-Weyl coefficient} of the restriction of the expression $\pm\ell$ to $\nR_\pm$.

A boundary triple for $S_\pm^*$ is $\bigl(\nC,\Gamma_0^\pm,
\Gamma_1^\pm\bigr)$, where
 \begin{equation} \label{ebtS}
\Gamma_0^\pm f := (p\,f')(0\pm), \quad \Gamma_1^\pm(f) = \mp
f(0\pm), \quad f \in \dom(S_\pm^*).
 \end{equation}
It follows from \eqref{eWeyl} and \eqref{ebtS} that the
Titchmarsh-Weyl coefficient defined by \eqref{eTWeyl} coincides with
the Weyl function of the operator $S_\pm$ in \eqref{edefS}
relative to the boundary triple in \eqref{ebtS}.

Assume additionally, that the coefficients $p$, $q$ and $w$ satisfy the
conditions:
\begin{enumerate}
    \item[(C3)] $p> 0$ and $p'\in AC[-b,b]$ for all $b>0$;
    \item[(C4)] $w = 1$.
    \end{enumerate}
The following asymptotic for the Titchmarsh-Weyl coefficient $m_\pm$ has been established  in~\cite{E72} (see also~\cite{Le}, where the asymptotic of the spectral function was found).  For all $z \in \nCR$
\begin{equation}\label{eq_Friedr_Asymp}
 m_\pm(r{z})=\frac{1}{\sqrt{p(0)}}\frac{1}{\sqrt{-{z} r}}+O\left(\frac{1}{r}\right)\qquad \text{as} \qquad  r\to +\infty.
\end{equation}
It is clear that this asymptotic implies \eqref{eq:AsymInf}.

By Theorem~\ref{prop:AasCoupl}, the coupling $A_1$ of $S_+$ and $S_-$ relative to the boundary triples given in \eqref{ebtS} coincides with the differential operator $A$ associated with the expression $\ell$ in $L_{w}^2(\nR)$.

In addition, assume that the operator $S_\pm$ is nonnegative in the Hilbert space $L_{w_\pm}^2\!(\nR_\pm)$.  Then the Friedrichs extension $S_F^\pm$ of $S_\pm$ defined on
\begin{equation*} 
  \dom\bigl(S_F^\pm\bigr) = \bigl\{ f \in \dom\bigl(S_\pm^*\bigr) : f(0) = 0 \bigr\}
\end{equation*}
is also nonnegative in the Hilbert space $L_{w_\pm}^2\!(\nR_\pm)$.
Since $\dom\bigl(S_F^\pm\bigr) = \ker(\Gamma_1^\pm)$, the function $m_\pm^{\!\top} = -1/m_\pm$ is holomorphic on $\nR_-\subset\rho(S_F^\pm)$ and satisfies
\begin{equation*} 
\lim_{x\downarrow-\infty} m_{\pm}^{\!\top}(x) = -\infty.
\end{equation*}
By Corollary~\ref{col:nonneg} the self-adjoint operator $A = A_1$ is nonnegative in the Kre\u{\i}n space $\bigl(L_{w}^2(\nR),\kip\bigr)$ if and only if
\begin{equation}\label{eq:Noneg_A_1}
  \lim_{x\uparrow 0}\bigl(m_+^{\!\top}(x)+m_-^{\!\top}(x)\bigr)\le 0.
\end{equation}

Assuming that $A = A_1$ is nonnegative, \eqref{eq:Noneg_A_1} yields  $m_+^{\!\top}(0-), m_-^{\!\top}(0-) \in \nR$.  Therefore, if $A_1$ is nonnegative, then $m_+^{\!\top}, m_-^{\!\top}  \in \SM$, and consequently, $m_+, m_-  \in \SM$. Now, \eqref{eq_Friedr_Asymp} and Proposition~\ref{prop:0-inf} imply that $m_+, m_-  \in \cA_\infty$.

Further, by Corollary~\ref{co:Binfprop} the functions $m_+$ and $m_-$ have $B_\infty$-property and by Proposition~\ref{prop:DInf} $m_+$ and $m_-$ have $D_\infty$-property.  Therefore, by Theorem~\ref{Main}, we have $\infty\not\in c_s(A)$. This result also follows from \cite[Theorem~3.6]{CL}, which was proved by completely different methods.

Moreover,  Theorem~\ref{t:AllPos} yields that not only $A$, but an arbitrary nonnegative self-adjoint extension of $S_+\oplus(\!-S_-\!)$ in the Kre\u{\i}n space $\bigl(L_{w}^2(\nR),\kip\bigr)$ has a nonempty resolvent set and $\infty$ is not its singular critical point.
\end{example}

\begin{example}\label{ex:5.2}
Consider the differential expression \eqref{eq2ndq} with $p = w = 1$ and assume that the potential $q \in L^1_{\rm loc}(\nR)$ satisfies
\[
\int_\nR (1+|t|)|q(t)|dt < +\infty.
\]
These functions clearly satisfy (C1); (C2) follows from \cite[Problem~9.4]{CoLe} or \cite[Theorem~1]{Re}. Let the operator $S_\pm$ and the corresponding  Titchmarsh-Weyl coefficient  $m_\pm$ be as in Example~\ref{ex2ndq}. Assume that $S_\pm$ is a nonnegative operator in the Hilbert space $L^2(\nR_{\pm})$ and
that~\eqref{eq:Noneg_A_1} holds.  Then the coupling $A_1$, which coincides with the differential operator $A$ associated with the expression \eqref{eq2ndq}, is a nonnegative self-adjoint operator in the Kre\u{\i}n space $\bigl(L^2(\nR),\kip\bigr)$.

The asymptotic behavior of $m_\pm$ at $0$ has been established in~\cite{KaKoMal09} as follows.  It was shown that either there exists $k_\pm > 0$ such that for all $z \in\nC\setminus\nR$ we have
\begin{equation} \label{eq:as1+}
    m_\pm(z) \sim  - k_\pm \sqrt{-r z} \qquad \text{as} \qquad  r \to 0,
\end{equation}
or there exist $a_\pm>0$ and $b_\pm\in\nR$ such that for all $z \in\nC\setminus\nR$ we have
\begin{equation} \label{eq:as2+}
    m_\pm(rz)  \sim  \frac{a_\pm}{b_\pm + \sqrt{- r z}}  \qquad \text{as} \qquad r \to 0.
\end{equation}
Assumption~\eqref{eq:Noneg_A_1} implies that the case \eqref{eq:as1+} is not possible. Thus \eqref{eq:as2+} holds. The asymptotic in  \eqref{eq:as2+} implies that for all $z \in\nC\setminus\nR$ we have
\begin{equation} \label{eq:as3+}
    m_\pm(rz) - \frac{a_\pm}{b_\pm} \sim -\frac{a_\pm}{b_\pm^2} \sqrt{- r z}  \qquad \text{as} \qquad r \to 0.
\end{equation}
Since in Example~\ref{ex2ndq} we proved that $m_+, m_- \in \SM$, \eqref{eq:as3+} and Proposition~\ref{prop:0-inf} yield $m_+, m_- \in \cA_0$. Recall that in Example~\ref{ex2ndq} we proved that $m_+, m_- \in \cA_\infty$ and in \cite[Proposition~4.3]{KaKoMal09} it was proved that $\ker(A) = \{0\}$. Now Theorem~\ref{t:AllPos} implies that $A$ is fundamentally reducible.  This has been proved in~\cite{KaKoMal09} and   Theorem~\ref{t:AllPos} provides an alternative proof of this result. Under a stronger assumption on $q$ the fundamental reducibility of $A$ was proved in \cite{FS00} using a different approach.

However, Theorem~\ref{t:AllPos} implies more. Let $\wt{A}$ be an arbitrary nonnegative self-adjoint extension of $S_+\oplus(-S_-)$ in the Kre\u{\i}n space $\bigl(L^{2}(\nR),\kip\bigr)$.  As in \cite[Proposition~4.3]{KaKoMal09}, it follows that $\ker(\wt{A}) = \{0\}$. Since $m_+, m_- \in \cA_0 \cap\cA_\infty$, Theorem~\ref{t:AllPos} implies that $\wt{A}$ is fundamentally reducible.
\end{example}

\begin{example}\label{ex:5.3}
Consider the differential expression \eqref{eq2ndq} with $q = 0$ and
\[
w(t) = \begin{cases}
(-t)^{\alpha_-}, & \text{if} \ \ t < 0, \\
\phantom{(}\rule{5pt}{0pt}t^{\alpha_+}, & \text{if} \ \ t > 0, \\
\end{cases}
\qquad \text{and} \qquad
p(t) = \begin{cases}
(-t)^{\beta_-}, & \text{if} \ \ t < 0, \\
\phantom{(}\rule{5pt}{0pt} t^{\beta_+}, & \text{if} \ \ t > 0, \\
\end{cases}
\]
with $\alpha_+, \alpha_- > -1$ and $\beta_+, \beta_- < 1$. In this example we use the notation introduced in Example~\ref{ex2ndq}. It is clear that the coefficients $p, q$ and $w$ satisfy condition (C1). In~\cite[Theorem~1]{EvZ78} it was proved that (C2) is also satisfied and that the function $m_\pm$ has the form
\[
 m_\pm({z}) = C_\pm{(-{z})^{-\nu_\pm}},
\]
where
\[
\nu_\pm=\frac{1-\beta_\pm}{\alpha_\pm-\beta_\pm+2},  \quad C_\pm=\frac{(2k_\pm)^{2\nu_\pm}\Gamma(1+\nu_\pm)}{(1-\beta_\pm)\Gamma(1-\nu_\pm)}, \quad
 k_\pm=\frac{\alpha_\pm-\beta_\pm + 2}{2}.
\]
As in Example~\ref{ex2ndq} the coupling $A_1$ of $S_+$ and $S_-$ relative to the boundary triples $\bigl(\nC, \Gamma_0^+,\Gamma_1^+\bigr)$ and $\bigl(\nC, \Gamma_0^-,\Gamma_1^-\bigr)$ given in \eqref{ebtS} coincides with the differential operator associated in $L_{w}^2(\nR)$ with the expression \eqref{eq2ndq} with $q=0$ and above $p$ and $w$.

It follows from Proposition~\ref{prop:nonneg} that $A_1$ is a nonnegative self-adjoint operator in the Kre\u{\i}n space $\bigl(L_{w}^2(\nR),\kip\bigr)$. Moreover, since the equation $\bigl(py'\bigr)'= 0$ does not have nontrivial solutions in $L_{w}^2(\nR)$, we have $\ker(A_1) = \{0\}$.

As $\nu_\pm\in(0,1)$ we have $m_+, m_- \in \cA_\infty \cap \cA_0$. Now, Theorem~\ref{t:AllPos} implies that $A_1$ is fundamentally reducible in the Kre\u{\i}n space $L_{w}^2(\nR)$.
In the case $\alpha_\pm = \beta_\pm = 0$ this result was proved in \cite{CN} and for $\alpha_+ = \alpha_-$ and $\beta_\pm = 0$ in~\cite{FN98}.

However, Theorem~\ref{t:AllPos} implies more. Let $\wt{A}$ be an arbitrary nonnegative self-adjoint extension of $S_+\oplus(-S_-)$ in the Kre\u{\i}n space $\bigl(L_{w}^{2}(\nR),\kip\bigr)$. Then, clearly,  $\ker(\wt{A})  = \{0\}$ and Theorem~\ref{t:AllPos} yields that $\wt{A}$ is fundamentally reducible in the Kre\u{\i}n space $\bigl(L_{w}^2(\nR),\kip\bigr)$. In \cite{Kost06} similar results were obtained for the case $\alpha_+ = \alpha_-$ and $\beta_\pm = 0$.
\end{example}

\begin{example}
In this example we present a nonnegative coupling in a Kre\u{\i}n space which has a singular critical point at $0$. This example is modeled after the example in \cite[Section~5]{KaKost08}. It is proved in \cite{KaKost08} that for the expression in \eqref{eq2ndq} with $p=1$, $q = 0$ and $w(t) = (3|t|+1)^{-4/3}$ the corresponding Titchmarsh-Weyl coefficients $m_+$ and $m_-$ are $m_+(z) = m_-(z) = 1/\sqrt{-z}-1/z$. From this it was deduced in \cite{KaKost08} that $0$ is a singular critical point of the associated nonnegative differential operator in $(L_{w}^2(\nR),\kip)$. We think that it will be instructive to present a similar example as a coupling.

First define the Hilbert space $\cK_\pm = L^2(\nR_\pm)\oplus\nC$ with the inner product
\[
\left\la\! \left(\!\!\begin{array}{c} f_\pm \\ u_\pm \end{array}\!\! \right),
\left(\!\! \begin{array}{c} g_\pm \\ v_\pm \end{array}\!\!\right) \!\right\ra_\pm
  \! := \! \int_{\nR_\pm} \!\! f_\pm(t) g_\pm^*(t) dt + u_\pm v_\pm^*, \  \begin{array}{l}
 f_\pm, g_\pm \in L^2(\nR_\pm), \\ u_\pm, v_\pm \in \nC.
\end{array}
\]
In this space we consider the operator $S_\pm$ given by its graph
\[
\gr(S_\pm) = \left\{
 \left\{
\left(\!\!\!\begin{array}{c} f_\pm \\ \mp f(0\pm) \end{array}\!\! \right),
\left(\!\!\!\begin{array}{c} - f_\pm'' \\ 0 \end{array}\!\!\right)
\right\} \, : \, \begin{array}{c}
 f_\pm \in W^{2,2}(\nR_\pm) \\[3pt]
 f'(0\pm) = 0
\end{array}    \right\}.
\]
It is easy to see that $S_\pm$ is densely defined and positive. The graph of its adjoint is
\[
\gr(S_\pm^*) = \Biggl\{
 \left\{
\left(\!\!\begin{array}{c} f_\pm \\ u_\pm \end{array}\!\! \right),
\left(\!\! \begin{array}{c} - f_\pm''(t) \\ f_\pm'(0\pm) \end{array}\!\!\right)
\right\} \, : \,
 f_\pm \in W^{2,2}(\nR_\pm)  \Biggr\}.
\]
A boundary triple $\bigl(\nC, \Gamma_0^\pm, \Gamma_1^\pm\bigr)$ for $S_\pm$ is
\[
\Gamma_0^\pm\left(\!\!\!\begin{array}{c}
  f_\pm \\
  u_\pm \\
\end{array}\!\!\!\right) = f_\pm'(0\pm),\qquad \Gamma_1^\pm \left(\!\!\!\begin{array}{c}
  f_\pm \\
  u_\pm \\
\end{array}\!\!\!\right)= \mp f_\pm(0\pm)-u_\pm .
\]
The corresponding Weyl solution is
\[
\psi_\pm(t,z) = \left(\!\!\begin{array}{c}
  \mp \exp\bigl(\mp t \sqrt{-z} \bigr)/\sqrt{-z} \\
  1/z \\
\end{array}\!\!\right), \qquad z \in \nC^+.
\]
That is,
\[
\Gamma_0^\pm \psi_\pm(\cdot,z) = 1 \quad \text{and} \quad S_\pm^* \psi(\cdot, z) = z \psi(z), \qquad
 z \in \nC^+.
 \]
Therefore the corresponding Weil function is
\[
m_\pm(z) = \frac{1}{\sqrt{-z}}-\frac{1}{z}, \qquad z \in \nC^+.
\]

The graph of the coupling $S_1$ of the symmetric operators $S_-$ and $S_+$ in the direct sum Hilbert space
\[
\cK = \cK_+ \oplus \cK_- = L^2(\nR) \oplus \nC^2
\]
is
\begin{equation*} 
\left\{  \left\{ \begin{pmatrix} f \\ u_+ \\u_-
\end{pmatrix},
\begin{pmatrix} - f'' \\ f'(0+)\\ f'(0-) \end{pmatrix} \right\} :
 \,\begin{array}{c}
 f \in W^{2,2}(\nR_-)\oplus W^{2,2}(\nR_+)  \\
    f'(0+)=f'(0-), \\
    f(0-)-f(0+)=u_-+u_+, \\
 \end{array} \right\}.
\end{equation*}
By Corollary \ref{col:nonneg} $S_1$ is a nonnegative operator in $(\cK, \ahip)$.

The following indefinite inner product turns the Hilbert space $\cK = L^2(\nR) \oplus \nC^2$ into a Kre\u{\i}n space:
\[
\left[
\left(\!\!\begin{array}{c} f \\ u_+ \\ u_-  \end{array}\!\! \right),
\left(\!\! \begin{array}{c} g \\ v_+ \\ v_- \end{array}\!\!\right) \right] := \int_\nR (\sgn t) f(t)g(t)^* dt + u_+v_+^* -u_-v_-^*
\]
for all $f, g \in L^2(\nR)$ and all $u_+,u_-,v_+,v_- \in \nC$. The direct sum $\cK = \cK_+\oplus \cK_-$ is a corresponding fundamental decomposition of $(\cK,\kip)$.

The graph of the coupling $A_1$ of the operators $S_-$ and $S_+$ is
\begin{equation} \label{edefwAHC4pm}
\left\{  \left\{ \begin{pmatrix} f(t) \\ u_+ \\u_-
\end{pmatrix},
\begin{pmatrix} - (\sgn t) f''(t) \\ f'(0+) \\ - f'(0-) \end{pmatrix} \right\} :
 \,\begin{array}{c}
 f \in W^{2,2}(\nR_-)\oplus W^{2,2}(\nR_+)  \\
    f'(0+)=f'(0-), \\
    f(0-)-f(0+)=u_-+u_+, \\
 \end{array} \right\}.
\end{equation}
It follows from~\eqref{edefwAHC4pm} that
$\ker(A_1)=\ker(A_1^2)$ and $\ker(A_1)$ is spanned by $( 0 \, \ 1 \ -1 )^{\!\top}$.

Since the pair of Weyl functions $m_+(z) = m_-(z) = 1/\sqrt{-z}-1/z$ does not have $D_0$-property, it follows from Corollary~\ref{Dinf} that $0$ is a singular critical point for $A_1$.
\end{example}


\begin{example}
Consider the following boundary value problem.  For an arbitrary $g \in L^2(\nR_+)$ find $f \in L^2(\nR_+)$ such that
\begin{equation}\label{24Example}
  -f''-zf=g,
  \quad f'(0)-\sqrt{2}\sqrt[4]{z}f(0)=0.
\end{equation}
In this example we will show that a linearization of this boundary value problem in the Kre\u{i}n space $\bigl(L^2(\nR),\kip\bigr)$ is fundamentally reducible.

Here the Kre\u{\i}n space $\bigl(L^2(\nR),\kip\bigr)$ is the $L^2$ space introduced in Example~\ref{ex2ndq} with $w=1$. We also use the fundamental decomposition $L^2(\nR) = \cK_+ [\dot{+}]\cK_-$ with $\cK_+$ and $\cK_-$ from \eqref{eq:L2pm}.

Let $S_+$ be the minimal operator associated with the differential expression $-\frac{d^2}{dt^2}$  in the Hilbert space  $\cK_+$ and let $\bigl(\nC,\Gamma^+_0,\Gamma^+_1\bigr)$ be the boundary triple for $S_+^*$ given by
\[
\Gamma^+_0f_+=f_+(0),\quad \Gamma^+_1f_+=f'_+(0).
\]
The Weyl function of $S_+$ relative to this boundary triple is $m_+(z)=-\sqrt{-z}$.

Let $S_-$ be the differential operator in the Hilbert space $\bigl(\cK_-,-\kip\bigr)$ defined on
\[
\dom(S_-) = \bigl\{ f_- \in W^{4,2}(\nR_-)\, : \, f_-(0) = f_-'(0) =
f_-''(0) = 0\bigr\}
\]
by $S_-f_- = f_-^{(4)}$.  Then the adjoint operator $S_-^*$ is defined on
\[
\dom(S_-^*) = \bigl\{ f_- \in W^{4,2}(\nR_-)\, : \, f_-(0) = 0\bigr\}
\]
by the same expression $S_-^*f_- = f_-^{(4)}$. With
\[
\Gamma_0^-f_- = f_-'(0), \ \ \ \Gamma_1f_- = -f_-''(0), \ \ \ f_- \in \dom(S_-^*),
\]
the triple $\bigl(\nC,\Gamma_0^-,\Gamma_1^- \bigr)$ is a boundary
triple for $S_-^*$. It turns out that
\[
m_-({z})=  -\sqrt{2}\sqrt[4]{-{z}}
\]
is the corresponding Titchmarsh-Weyl coefficient.
By Theorem~\ref{tKsc} the coupling $A_1$ of the operators $S_+$ and $S_-$ in the graph notation  takes the form
\begin{equation*} 
  A_1=\left\{\left\{\begin{pmatrix}
    f_+ \\
    f_-
  \end{pmatrix},
  \begin{pmatrix}
    -f_+'' \\
    f_-^{(4)}
  \end{pmatrix}
  \right\}: \!\! \begin{array}{l}
    f_+ \in W^{2,2}(\nR_+),\,f_+(0)-f_-'(0)=0\\
    f_- \in W^{4,2}(\nR_-),\,
   f_+'(0)-f''_-(0)=f_-(0)=0
  \end{array}
  \!\right\}
\end{equation*}
and the compressed resolvent $P_{+}(A_1-z)^{-1}|_{\cK_+}$ of $A_1$ gives a solution of the boundary value problem~\eqref{24Example} by the formula
\[
f=P_{+}(A_1-z)^{-1}g.
\]
The operator $A_1$ is called a linearization of the boundary value problem~\eqref{24Example}.

By Corollary~\ref{col:nonnegF} the operator $A_1$ is nonnegative. Since clearly $\ker(A_1)=\{0\}$ and the functions $m_+$ and $m_-$ belong to $\cA_\infty \cap \cA_0$, Theorem~\ref{t:AllPos} yields that the linearization $A_1$ of~\eqref{24Example} is fundamentally reducible.
\end{example}

\end{document}